\documentclass[twoside]{article}
\usepackage{amsthm}
\usepackage{amsmath,amssymb,mathtools}
\usepackage{enumerate}
\usepackage{dsfont}
\usepackage{xcolor}
\usepackage{extarrows}
\usepackage{bm}

\usepackage{PRIMEarxiv}

\usepackage[utf8]{inputenc} 
\usepackage[T1]{fontenc}    
\usepackage{hyperref}       
\usepackage{url}            
\usepackage{booktabs}       
\usepackage{amsfonts}       
\usepackage{nicefrac}       
\usepackage{microtype}      
\usepackage{lipsum}
\usepackage{fancyhdr}       
\usepackage{graphicx}       
\graphicspath{{media/}}     

\newtheoremstyle{break}
  {10pt}{10pt}
  {\normalfont}
  {}
  {\bfseries}
  {.}
  {\newline}
  {}
\theoremstyle{plain}
\newtheorem{theorem}{Theorem}[section]
\newtheorem{proposition}[theorem]{Proposition}
\newtheorem{lemma}[theorem]{Lemma}
\newtheorem{definition}[theorem]{Definition}
\newtheorem{assumption}[theorem]{Assumption}

\theoremstyle{break}
\newtheorem*{remark}{Remark}

\numberwithin{equation}{section}

\title{Estimation of the elasticity for CKLS model from high-frequency observations
}

\author{
  Boyuan Ning* \\
  Graduate School of Fundamental Science and Engineering \\
  Waseda University \\
  3-4-1 Ohkubo, Shinjuku-ku, Tokyo, Japan\\
  \texttt{ningboyuan@akane.waseda.jp} \\
  \And
  Yasutaka Shimizu \\
  Department of Applied Mathematics \\
  Waseda University \\
  3-4-1 Ohkubo, Shinjuku-ku, Tokyo, Japan\\
  \texttt{shimizu@waseda.jp} \\
}

\begin{document}
\maketitle

\begin{abstract}
    We investigate parametric estimation of the elasticity parameter in the CKLS diffusion based on high-frequency data. First, we transform  the CKLS diffusion to a CIR-type one via a smooth state-space mapping and the general Girsanov change of measure. This transformation enables the applications of existing inference tools for CIR processes while ensuring possibilities of transferring the resulting limit theorems back to the original probability space. However, because Feller's condition fails, many existing high-frequency likelihood-based procedures cannot be applied directly, since their discretization schemes approximate likelihood terms involving the reciprocal of the process by Riemann sums that are no longer well-defined once the paths are allowed to hit zero. Instead, we estimate the drift coefficient of the transformed CIR-type model via a procedure based on its positive Harris recurrence, which is valid in the high-frequency regime. Exploiting the drift-elasticity relationship implied by the CKLS--CIR transformation, with the help of an initial estimation, we obtain an estimator of the CKLS elasticity from the CIR drift estimator in the transformed model. This yields a closed-form estimator of the elasticity parameter with an explicit asymptotic variance. We establish its $p$-consistency, stable convergence in law, and asymptotic normality. Finally, we show that stable convergence in law is invariant under equivalent changes of measure, thereby guaranteeing that the Gaussian limit remains invariant under the original measure.
\end{abstract}


\section{Introduction}
The Chan--Karolyi--Longstaff--Sanders (CKLS) model is introduced by \cite{chan1992empirical} as a flexible generalization of classical short-rate diffusions and, in particular, as a remedy for some of the restrictive features of the Cox--Ingersoll--Ross (CIR) specification, which has long served as a benchmark for describing short-rates dynamics (e.g., Heston model). It allows volatility of short-rates to respond more flexibly to the level of the rate through an elasticity parameter. 
\begin{definition}
Let $(\Omega, \mathcal{F}, \{\mathcal{F}_{t}\}_{t \geq 0}, \mathbb{P})$ be a filtered probability space supporting a standard one-dimensional Wiener process $W_{t}$. The CKLS process, denoted by $\lambda_{t}$, is defined as the solution to the SDE (the CKLS model):
\begin{equation}\label{CKLS-process}
    d\lambda_{t}=(a-b\lambda_{t})dt+\sigma\lambda_{t}^{k}dW_{t}, \quad \lambda_{0}>0,
\end{equation}
where $a$ denotes the level, $b$ the speed of mean-reversion, $a/b$ the long-term mean, $\sigma>0$ denotes the volatility, and $k\geq 0$ the elasticity. For some of its key properties, see Lemma~\ref{ergodicity-of-CKLS} in Appendix~\ref{aaa}.
\end{definition}
However, when $a\neq 0$, the CKLS model admits neither a closed-form solution nor an explicit transition density, which creates substantial technical difficulties for both stochastic analysis and statistical inference.
\begin{definition}
Let $(\Omega, \mathcal{F}, \{\mathcal{F}_{t}\}_{t \geq 0}, \mathbb{P})$ be a filtered probability space supporting a standard one-dimensional Wiener process $W_{t}$. The Cox--Ingersoll--Ross (CIR) process, denoted by $X_{t}$, is defined as the solution to SDE (the CIR model):
\begin{equation}\label{CIR-process}
    dX_{t}=(\alpha-\beta X_{t})dt+\gamma\sqrt{X_{t}}dW_{t},\quad X_{0}>0,
\end{equation}
where $\beta$ denotes the speed of mean-reversion, $\alpha/\beta$ the long-term mean, and $\gamma>0$ the volatility. The CIR model is a special case of the CKLS model with $k=\frac{1}{2}$. In general, only the ergodic case where $\beta>0$ is considered. For some of its key properties, see Lemma~\ref{ergodicity-of-CIR} in Appendix~\ref{aaa}.
\end{definition}
\noindent The square-root specification in the CIR model imposes a fixed elasticity of volatility with respect to the rate level, and is often viewed as too restrictive to accommodate the full empirical variability of short-rate movements. Empirical studies have indicated that short-rate volatility does depend on the level of the rate (heteroskedastic short-rates across samples), so elasticity $k$ is of central importance. In comparison, the CKLS model incorporates a state-dependent volatility whose magnitude adjusts systematically with the level of the process, thereby reflecting the empirically observed sensitivity of the conditional variance to the prevailing rate. Capturing this volatility-level relationship is a central element in modeling short-rates, and the estimation of the elasticity parameter becomes a crucial problem in empirical implementation. As an initial attempt, the seminal work \cite{chan1992empirical} employed GMM to estimate this parameter.\\[0.5\baselineskip]
\noindent A substantial body of literature investigates parametric inference for CKLS-type diffusions. We only cite some of the most recent results from roughly the past decade.\\
\noindent (\romannumeral1) \textbf{Regarding drift}, 
\cite{sanchez2016parameter} and \cite{monsalve2017parameter} address deterministic and periodic long-term trends, respectively, using a two-step Gaussian quasi-maximum likelihood procedure: they first estimate the drift and volatility coefficients from an Euler--Maruyama-based Gaussian regression with a smoothed functional trend, and then refine the trend component (via truncated Fourier series in the periodic case) and re-estimate the external parameters; \cite{mishura2022parameter} treat the Girsanov weight (also known as Radon-Nikod\'{y}m derivative/likelihood ratio or Dol\'{e}ans--Dade exponential in certain situations) as a likelihood function and derive continuous-time MLEs, establishing both strong consistency and asymptotic normality by combining CKLS ergodicity with a multidimensional martingale CLT. They also propose a consistent alternative estimator, inspired by the CIR-based methodology of \cite{dehtiar2022two}; Extending the mean structure, \cite{lyu2025inference} (see also \cite{lyu2023inference}) allow for a deterministic periodic mean-reversion level, establish ergodicity and positive Harris recurrence of the associated grid chain, and investigate both continuous unrestricted and restricted MLEs, deriving their joint asymptotic normality under local alternatives. They further introduce shrinkage and positive-part shrinkage estimators, which typically outperform the unrestricted MLEs in simulations; To capture regime changes, \cite{mazzonetto2025parameters}, building on \cite{su2015quasi}, develop a threshold CKLS (T-CKLS) model and obtain closed-form MLE/QMLE for interval-specific drift coefficients using localized power-time/path integrals. They establish asymptotic normality for these estimators under both continuous observation and high-frequency sampling over an infinite horizon.\\[0.5\baselineskip]
\noindent (\romannumeral2) \textbf{Regarding volatility and elasticity}, \cite{mazzonetto2025parameters} propose a nonparametric volatility estimator based on a It\^{o}-Tanaka formula: the quadratic variation of a suitably chosen functional can be written, on the one hand, as a diffusion-scaled occupation-time integral and, on the other, as a purely pathwise functional. Equating these two representations yields an estimator that does not rely on likelihoods and is valid on the event that the process spends positive time in the interval. The expression of the estimator's discretized version is also given. \cite{mishura2022parameter} discuss realized-volatility-type estimators for the elasticity that is consistent but not robust. Given an estimate of the elasticity, one can either obtain the volatility estimate directly by matching local quadratic variation increments to their theoretical form, or apply a Lamperti transformation to reduce the diffusion to unit volatility, and then estimate from the quadratic variation of the transformed process (see \cite{lyu2025inference}); Dedicated inference on the elasticity itself is scarce: \cite{dokuchaev2017pathwise} propose a pathwise nonparametric procedure based on an auxiliary complex-valued process whose log-modulus aggregates normalized increments; taking ratios at two tuning levels removes $\sigma$ and yields a consistent estimator of $k$, with $\sigma$ subsequently recovered—albeit at the cost of tuning sensitivity and the absence of a simple closed form.

\section{Assumptions, Main Results and Overall Flow}
\subsection{Standing Assumptions}
\begin{assumption}\label{ass1}[Model and parameter restriction]
\noindent \begin{itemize}
    \item Values of the following parameters are known: the level $a>0$, the mean-reversion speed $b>0$ (which sufficiently ensures CKLS's ergodicity) and the instantaneous volatility $\sigma>0$, a free scaling parameter constant $L>0$.
    \item The CKLS process is specified as in (\ref{CKLS-process}) with true elasticity $k_{0}$ satisfying
    \begin{equation}\nonumber
	   \frac{1}{2}<k_{0}<1 \quad\text{or}\quad k_{0}=\frac{1}{2}\text{ with }2a\geq\sigma^{2},
    \end{equation}
so that $\lambda_{t}$ enjoys the stability and ergodic properties collected in Lemma~\ref{ergodicity-of-CKLS} (existence and uniqueness, strict positivity and non-explosiveness, positive Harris recurrence, uniform moment bounds, ergodic limits of moments). Section~3 will elaborate on this.
\end{itemize}
\end{assumption}
\begin{assumption}\label{ass2}[Ultra-high-frequency sampling]\\
\noindent Let $T=T_{n}$ be the terminal time and $\Delta=\Delta_n>0$ the mesh size. For notational simplicity, we will continue to write $T$ and $\Delta$ in what follows.
\begin{itemize}
\item For each $n\in\mathbb{N}$, we observe
\begin{equation}\nonumber
	\{\lambda_{i\Delta}\}_{i=0}^{n}\quad\text{on }[0,T],\qquad T=n\Delta.
\end{equation} 
\item An ultra-high-frequency sampling scheme is such that, as $n\to\infty$,
\begin{equation}\nonumber
  \Delta\to0,\qquad n\Delta\to\infty,\qquad n\Delta^{2}\to0.
\end{equation}
\end{itemize}
\end{assumption}
\begin{remark}[More on the sampling scheme]
We observe the process at discrete, equally-spaced time points over an increasing time horizon. In particular, we consider an equidistant partition of $[0,T]$ with mesh size $\Delta_{n}:=T^{-\omega}$ for some $\omega>0$. Observations are then recorded at the grid points $t_{i}=i\Delta=iT^{-\omega}$, $i=0,\dots,n$, so that
\[
  n=\big[T/\Delta\big]=\big[T^{\omega+1}\big],\qquad \pi_n = \{t_{0},t_{1},\dots,t_{n}\}=\{0,\Delta,2\Delta,\dots,n\Delta=T\}.
\]
As $T\to\infty$, the condition $\omega>0$ is equivalent to
$\Delta\to0$ and $n\Delta^{2}=T^{1-\omega}$. Hence $0<\omega<1$ corresponds to $n\Delta^{2}\to\infty$ (the usual high-frequency setting), while $\omega>1$ yields $n\Delta^{2}\to0$, which is the ultra-high-frequency scheme adopted in our scenario.
\end{remark}
\subsection{Main Result: Construction of the Elasticity Estimator}
\noindent \textbf{First Step: Transformation from the CKLS to a CIR model (Section~3):} We revisit the main result of \cite{ning2025ckls}, in which a twice continuously differentiable state-space mapping $\mathcal{T}(\cdot\lvert k)$ together with a general Girsanov change of measure is used to transform the CKLS diffusion into a CIR-type diffusion. For each fixed $k\in[\frac{1}{2},1)$, set
\[
     X_{t}^{(k)}:=\mathcal{T}(\lambda_{t}\lvert k)=\frac{L^{2}}{4(1-k)^{2}}\lambda_{t}^{2-2k},
\]
and construct an equivalent probability measure $\mathbb{Q}$ under which
$X_{t}^{(k)}$ solves
\[
    dX_{t}^{(k)}=(\alpha-\beta X_{t}^{(k)})dt+\gamma\sqrt{X_{t}^{(k)}}d\widetilde{W}_{t},
\]
with $\beta\equiv\beta(k)=2b(1-k)$, while
$\alpha=\sigma^{2}L^{2}/4,\gamma=\sigma L$ are explicit functions of $\sigma$ alone. The resulting SDE only represents a subcase of the general CIR model, and Feller's condition fails (see Lemma~\ref{ergodicity-of-CIR})'s (\romannumeral1) because $2\alpha<\gamma^{2}$, thereby destroying the strict positivity of $X_{t}^{(k)}$. Section~3 revisits \cite{ning2025ckls}'s analysis of the explicit form of $\mathcal{T}$, the construction of $\mathbb{Q}$ and the verification of the martingale property of the Radon-Nikod\'{y}m density. Under Assumption~2.1, $X_{t}^{(k_{0})}$ follows a CIR-type diffusion is a CIR-type diffusion for which Feller's condition fails.\\[0.5\baselineskip]
\noindent \textbf{Second Step: Drift estimation and inversion (Sections~4,~6,~7):} Section~4 develops the asymptotic theory for the ultra-high-frequency estimator of \cite{prykhodko2025discretization} applied to the transformed CIR-type model under $\mathbb{Q}$. This estimator is applied because their method yields a valid drift estimator circumventing inverse-state terms and remains valid even when Feller's condition fails. Denote the drift estimator for mean-reversion speed $\beta$ (true value $\beta_{0}$) by $\hat{\beta}_{n}$ constructed from $\{X_{i\Delta}^{(k)}\}_{i=0}^{n}$ and the sampling scheme in Assumption~2.2; for its explicit expression, see \eqref{beta-form}. When $k_{0}$ is known and $\{X_{i\Delta}^{(k_{0})}\}_{i=0}^{n}$ is observed, the Prykhodko--Ralchenko's estimator of $\beta=\beta^{(k_{0})}$ is strongly consistent and asymptotically normal with rate $\sqrt{n\Delta}$ and limit variance $\frac{2\beta_{0}}{\alpha}(\alpha+\gamma^{2})$ with $\beta_{0}=\beta_{0}^{(k_{0})}$. In the transformed CIR-type model, we have $\alpha=\sigma^{2}L^{2}/4$ and $\gamma=\sigma L$, hence the limit variance is
\begin{equation}\nonumber
	\frac{2\beta_{0}}{\alpha}(\alpha+\gamma^{2})=\frac{2\cdot 2b(1-k_{0})}{\sigma^{2}L^{2}/4}\Big(\sigma^{2}L^{2}+\frac{\sigma^{2}L^{2}}{4}\Big)=20b(1-k_{0}).
\end{equation}
so that under $\mathbb{Q}$ we obtain the following CLT
\begin{equation}\nonumber
	\sqrt{n\Delta}\big(\hat{\beta}_{n}-\beta_{0}\big)\xrightarrow{d}\mathcal{N}\big(0,20b(1-k_{0})\big).
\end{equation}
In practice $k_{0}$ is unknown, so $\mathcal{T}(\cdot\lvert k_{0})$ cannot be evaluated directly. Section~6 constructs two kinds of model-based initial estimator $\hat{k}_{0,n}^{(l)}$ for $l=1,2$ (true value $k_{0}$) for the elasticity based on realized quadratic
variation of the CKLS process (see (\ref{initial-etimators}) for the explicit expressions), and shows that under Assumptions~\ref{ass1}-\ref{ass2},
\begin{equation}\nonumber
	\hat{k}_{0,n}\xrightarrow{a.s.}k_{0}\quad\text{and}\quad\hat{k}_{0,n}-k_{0}=o_{p}\big((n\Delta)^{-\frac{1}{2}}\big)\qquad\text{as }n\to\infty.
\end{equation}
Using $\hat{k}_{0,n}$ as a plug-in in the state-space map, we form the transformed/pseudo CIR data:
\begin{equation}\nonumber
	X^{(\hat{k}_{0,n})}_{i\Delta}:=\mathcal{T}(\lambda_{i\Delta}\lvert \hat{k}_{0,n})=\frac{L^{2}}{4(1-\hat{k}_{0,n})^{2}}\lambda_{i\Delta}^{2-2\hat{k}_{0,n}},\qquad i=0,\dots,n,
\end{equation}
then apply the Prykhodko--Ralchenko's estimation procedure to $\{X^{(\hat{k}_{0,n})}_{i\Delta}\}_{i=0}^{n}$ to obtain the estimator $\hat{\beta}_{n}^{(\hat{k}_{0,n})}$; for its explicit expression, see \eqref{final-estimation}.

Section~7 shows that the plug-in error from using the transformed data $\{X^{(\hat{k}_{0,n})}_{i\Delta}\}_{i=0}^{n}$ in place of the observed CKLS data $\{\lambda_{i\Delta}\}_{i=0}^{n}$ is asymptotically negligible at the $\sqrt{n\Delta}$-scale. An estimator of the CKLS elasticity under $\mathbb{Q}$ is obtained by inverting the drift-elasticity relationship $\beta(k)=2b(1-k)$:
\begin{equation}\nonumber
	\hat{k}_{n}:=1-\frac{\hat{\beta}_{n}^{(\hat{k}_{0,n})}}{2b}=1-\frac{\sigma^{2}(1-\hat{k}_{0,n})^{2}}{b}\frac{n\sum_{i=1}^{n}\lambda_{(i-1)\Delta}^{2-2\hat{k}_{0,n}}}{n\sum_{i=1}^{n}\lambda_{(i-1)\Delta}^{4-4\hat{k}_{0,n}}-\Big(\sum_{i=1}^{n}\lambda_{(i-1)\Delta}^{2-2\hat{k}_{0,n}}\Big)^{2}}.
\end{equation}
Thus $\hat{k}_{n}$ is an explicit function of the observed CKLS observations $\{\lambda_{i\Delta}\}_{i=0}^{n}$ only.

\subsection{Main Result: Asymptotic Properties of the Elasticity Estimator}
We now summarize the asymptotic behavior of the estimator $\hat{k}_{n}$. The proofs are carried out within \textbf{the stable convergence in law framework developed in Section~5}, which we use to transfer the CLT under $\mathbb{Q}$ to the original measure $\mathbb{P}$ and to verify that the estimator remains asymptotically valid after the equivalent change of measure.
\begin{theorem}[Consistency]\label{thm:consistency-k}
Suppose Assumptions~\ref{ass1}-\ref{ass2} hold. Then
\begin{equation}\nonumber
	\hat{k}_{n}\xrightarrow{p}k_{0}\qquad\text{as }n\to\infty.
\end{equation}
\end{theorem}

\begin{theorem}[Stable central limit theorem]\label{thm:stable-CLT-k}
Under Assumptions~\ref{ass1}-\ref{ass2}, the estimator $\hat{k}_{n}$ satisfies
\begin{equation}\nonumber
	\sqrt{n\Delta}(\hat{k}_{n}-k_{0})\xrightarrow{d_{st}}\mathcal{N}\Big(0,\frac{5(1-k_{0})}{b}\Big)\qquad \text{as }n\to\infty,
\end{equation}
where $\xrightarrow{d_{st}}$ denotes stable convergence in law with respect to the underlying filtration.
\end{theorem}
\noindent\textit{Sketch of argument.}
Under $\mathbb{Q}$, with the data $\{X_{i\Delta}\}_{i=0}^{n}$ the estimator of \cite{prykhodko2025discretization} together with the specific variance calculation for our case yields an asymptotic variance $20b(1-k_{0})$ for $\hat{\beta}_{n}$. By the criterion in \cite{heyde1997quasi}, this convergence is in fact stable. Since $\beta(k)=2b(1-k)$, the transformation $k \mapsto 1-\beta(k)/(2b)$ immediately yields a stable CLT for $\hat{k}_{n}:=1-\hat{\beta}_{n}/(2b)$ constructed from data $\{X_{i\Delta}\}_{i=0}^{n}$ with variance $5(1-k_{0})/b$.

Since the plug-in estimator $\hat{k}_{0,n}$ satisfies $\hat{k}_{0,n}-k_{0}=o_{p}((n\Delta)^{-\frac{1}{2}})$ in the ultra-high-frequency setting (asymptotically negligible at the $\sqrt{n\Delta}$-scale), the same stable CLT under $\mathbb{Q}$ continues to hold when $\hat{\beta}_{n}$ is
replaced by its plug-in version $\hat{\beta}_{n}^{(\hat{k}_{0,n})}$, that is, for the
statistic $1-\hat{\beta}_{n}^{(\hat{k}_{0,n})}/(2b)$
constructed from $\{X_{i\Delta}^{(\hat{k}_{0,n})}\}_{i=0}^{n}=\{\mathcal{T}(\lambda_{i\Delta}\lvert \hat{k}_{0,n})\}_{i=0}^{n}$.

Finally, \cite{mykland2009inference} provides an invariance theorem for stable convergence in law under equivalent changes of measure. Since $d\mathbb{Q}/d\mathbb{P}$ is $\mathcal{F}_{T}$-measurable, the stable limit under $\mathbb{Q}$ implies stable convergence under $\mathbb{P}$ to the same limit. This yields Theorem~\ref{thm:stable-CLT-k} for $\hat{k}_{n}:=\hat{k}_{n}^{\mathbb{P}}$, which has the same explicit expression as $\hat{k}_{n}^{\mathbb{Q}}$ but is defined under $\mathbb{P}$.\\[0.5\baselineskip]
\noindent Taken together, Theorems~\ref{thm:consistency-k} and~\ref{thm:stable-CLT-k} show that, in the ultra-high-frequency setting and under mild structural conditions on the CKLS model, the estimator $\hat{k}_{n}$ is a closed-form, $p$-consistent, and asymptotically normal estimator of the elasticity parameter, with an explicit asymptotic variance that can be used for inference and confidence interval construction.

\section{Transformation from the CKLS Process to the CIR Process}
This section establishes an analytical transformation from the general CKLS process ($\frac{1}{2}\leq k<1$) to the CIR process of a particular form in which Feller's condition fails ($2\alpha=\frac{\gamma^{2}}{2}<\gamma^{2}$). We show that a specific smooth state mapping combined with a general Girsanov change of measure maps the two processes. This construction, which forms the core result of \cite{ning2025ckls}, provides both a theoretical connection between the two model families and a practical tool for estimation.

\subsection{A State-space Mapping on CKLS Process}
\subsubsection{Designing \texorpdfstring{$\mathcal{T}$}{T} to Make the Diffusion into a Square-root Type}
\noindent Assume there exists a twice differentiable mapping $\mathcal{T}$ such that $X_{t}=\mathcal{T}(\lambda_{t})$, where $\lambda_{t}$ and $X_{t}$ solve (\ref{CKLS-process}) and (\ref{CIR-process}), respectively. Applying It\^{o}'s lemma to $X_{t}=\mathcal{T}(\lambda_{t})$ shows that the diffusion coefficient of the $X_{t}$'s SDE becomes $\frac{d\mathcal{T}(\lambda_{t})}{d\lambda_{t}}\sigma\lambda_{t}^{k}dW_{t}$. To make it a CIR-type diffusion term, this coefficient must match $\gamma\sqrt{\mathcal{T}(\lambda_{t})}$. Hence the mapping $\mathcal{T}$ is required to satisfy $\frac{d\mathcal{T}(\lambda_{t})}{d\lambda_{t}}\sigma\lambda_{t}^{k}=\gamma\sqrt{\mathcal{T}(\lambda_{t})}$. With $L:=\gamma/\sigma>0$ a positive constant, we obtain the defining ODE for $\mathcal{T}$: $\frac{d\mathcal{T}(x)}{dx}x^{k}=L\sqrt{\mathcal{T}(x)}$. This separable ODE yields (choosing zero integration constant for parsimony) $\mathcal{T}:\mathbb{R}\to[0,+\infty)$ such that
\begin{equation}\nonumber
    \mathcal{T}(x\lvert k)=\begin{cases}
    \frac{L^{2}}{4(1-k)^{2}}x^{2-2k},&\text{if }k\neq 1;\\
    \frac{L^{2}}{4}\big(\log x\big)^{2},&\text{if }k=1.
\end{cases}
\end{equation}
In particular, for the case $k\neq 1$, the mapping and its derivatives and inverse are:
\begin{align}\label{eq:T_solution}
    \mathcal{T}(x\lvert k)&=\frac{L^{2}}{4(1-k)^{2}}x^{2-2k}\\
    \frac{\partial}{\partial x}\mathcal{T}(x\lvert k)=\frac{L^{2}}{2(1-k)}x^{1-2k},~
    \frac{\partial^{2}}{\partial x^{2}}\mathcal{T}(x\lvert k)&=\frac{L^{2}(1-2k)}{2(1-k)}x^{-2k},~
    \mathcal{T}^{-1}(y\lvert k)=\Big[\frac{2(1-k)}{L}\sqrt{y}\Big]^{\frac{1}{1-k}}.\nonumber
\end{align}
\subsubsection{Parameter Restriction}
For the mapping (\ref{eq:T_solution}) to be well defined and for $X_{t}^{(k)}:=\mathcal{T}(\lambda_{t}\lvert k)$ to constitute a valid CIR diffusion, the following parameter restriction $\frac{1}{2}\leq k<1$ is essential. To briefly explain why, we first assume $x\in\mathbb{R}$ and will repeatedly invoke Lemma~\ref{ergodicity-of-CKLS}'s (\romannumeral1) in the following.\\[0.5\baselineskip]
\noindent \textit{Case $k=1$:} When deriving expression of $\mathcal{T}$, we used the identity $\sqrt{\mathcal{T}(\lambda_{t}\lvert k)}=\frac{L}{2}\log \lambda_{t}$, which requires $\frac{L}{2}\log \lambda_{t}\geq 0$. However, since $\lambda_{t}\in(0,+\infty)$ holds almost surely for $k>\frac{1}{2}$, it is possible that $\lambda_{t}$ falls in the range $(0,1)$, causing $\log \lambda_{t}<0$. Hence the identity cannot hold on the event $\{\lambda_{t}\in(0,1)\}$, and we cannot guarantee that $\sqrt{\mathcal{T}(\lambda_{t}\lvert k)}$ always remains non-negative. This means that the case $k=1$ must be excluded.\\
\noindent \textit{Case $k\neq 1$:} When deriving expression of $\mathcal{T}$, we used the identity $\sqrt{\mathcal{T}(\lambda_{t}\lvert k)}=\frac{L}{2(1-k)}\lambda_{t}^{1-k}$, which requires $\frac{L}{2(1-k)}x^{1-k}\geq 0$. As $L>0$ is already guaranteed, the sign of $\frac{\lambda_{t}^{1-k}}{1-k}$ is of concern:
  \begin{itemize}
    \item \textbf{Subcase $k>1$:} Since $1-k<0$ and $\lambda_{t}>0$, we have $\lambda_{t}^{1-k}>0$ almost surely. Thus $\frac{L}{2(1-k)}\lambda_{t}^{1-k}<0$, implying that this case must be excluded.
    \item \textbf{Subcase $0<k<\frac{1}{2}$:} Here $1-k>0$ while $\lambda_{t}$ ranges over $(-\infty,+\infty)$, so $\lambda_{t}^{1-k}$ need not be nonnegative. For instance, if $k=\frac{1}{5}$, then $x\mapsto x^{4/5}$ is even function and nonnegative on $\mathbb{R}$; if $k=\frac{2}{5}$, then $x\mapsto x^{3/5}$ is odd, positive on $(0,+\infty)$ but negative on $(-\infty,0)$. Thus $\frac{\lambda_{t}^{1-k}}{1-k}$ is not guaranteed to be nonnegative, and this case must be excluded.
  \end{itemize}
\noindent The restriction $\frac{1}{2}\leq k<1$ ensures a well-defined $\mathcal{T}:[0,+\infty)\to[0,+\infty)$. Note that the case $k=\frac{1}{2}$ with $2a<\sigma^{2}$ (Feller's condition fails) is not excluded as $\lambda_{t}\geq 0$ (remaining nonnegative, though hitting $0$ with positive probability).

\subsection{Dynamics of the Mapped Process under \texorpdfstring{$\mathbb{Q}$}{Q}}
\subsubsection{Dynamics of \texorpdfstring{$X_{t}$}{Xt} under \texorpdfstring{$\mathbb{P}$}{P}}
Assume $\frac{1}{2}\leq k<1$, the full result of applying It\^{o}'s lemma to $X_{t}^{(k)}=\mathcal{T}(\lambda_{t}\lvert k)$ is the following
\begin{align}\nonumber
    dX_{t}^{(k)}&=\frac{\partial}{\partial \lambda_{t}}\mathcal{T}(\lambda_{t}\lvert k)(a-b\lambda_{t})dt+\frac{1}{2}\frac{\partial^{2}}{\partial \lambda_{t}^{2}}\mathcal{T}(\lambda_{t}\lvert k)\sigma^{2}\lambda_{t}^{2k}dt+\frac{\partial}{\partial \lambda_{t}}\mathcal{T}(\lambda_{t}\lvert k)\sigma\lambda_{t}^{k}dW_{t}\\
    &=\Big[\frac{L^{2}}{2(1-k)}\lambda_{t}^{1-2k}(a-b\lambda_{t})+\frac{1}{2}\frac{L^{2}(1-2k)}{2(1-k)}\lambda_{t}^{-2k}\sigma^{2}\lambda_{t}^{2k}\Big]dt+\sigma L \frac{L}{2(1-k)}\lambda_{t}^{1-k}dW_{t}.\nonumber
\end{align}
Using $\sqrt{\mathcal{T}(x\lvert k)}=\frac{L}{2(1-k)}x^{1-k}$ and $X_{t}^{(k)}=\frac{L^{2}}{4(1-k)^{2}}\lambda_{t}^{2-2k}$ simplifies the expression:
\begin{equation} \label{eq:ito_P}
    dX_{t}^{(k)}=\Big[\underbrace{\frac{aL^{2}}{2(1-k)}\lambda_{t}^{1-2k}-2b(1-k)X_{t}^{(k)}+\frac{\sigma^{2}L^{2}(1-2k)}{4(1-k)}}_{=:\mu(\lambda_{t},X_{t}^{(k)})}\Big]dt+\underbrace{\sigma L \sqrt{X_{t}^{(k)}}}_{=:\nu(X_{t}^{(k)})}dW_{t}.
\end{equation}
The diffusion term already admits the desired CIR form, while the drift is written in terms of both $X_{t}^{(k)}$ and $\lambda_{t}$, with $\lambda_{t}$ itself an explicit function of $X_{t}^{(k)}$. Thus the drift in fact depends solely on $X_{t}^{(k)}$, so $X_{t}^{(k)}$ remains Markovian. However, once $\lambda_{t}$ is substituted out, the drift becomes nonlinear in $X_{t}^{(k)}$, namely $-2b(1-k)X_{t}^{(k)}+\big[[\frac{2(1-k)}{L}]^{2}X_{t}^{(k)}\big]^{1-\frac{1}{2}\frac{1}{1-k}}$, which is not of a CIR form. This nonlinearity makes direct analysis cumbersome, motivating us to circumvent it by applying an appropriate Girsanov change of measure that reshapes the drift into a linear function of $X_{t}^{(k)}$.

\subsubsection{The Measure Change via General Girsanov's Theorem}
The goal is to find an equivalent probability measure $\mathbb{Q}$ under which the drift term of $X_{t}^{(k)}$ becomes of the form $\alpha-\beta X_{t}^{(k)}$. Compared with (\ref{eq:ito_P}), we choose to preserve the drift term $-2b(1-k)X_{t}^{(k)}$ (also to make the difference of new and old drifts depend solely on $\lambda_{t}$ and no $X_{t}^{(k)}$) and modify the rest. Let the target drift under $\mathbb{Q}$ be:
$$\widetilde{\mu}(X_{t}^{(k)}):=\frac{\sigma^{2}L^{2}}{4}-2b(1-k)X_{t}^{(k)}.$$
This expression is deliberately designed so that $\widetilde{\mu}$ is linear in $X_{t}^{(k)}$ while verifying $\Lambda_{t}$'s martingale property constituted of Girsanov kernel (see the theorem that follows) is not intractably complicated.

By the general Girsanov's theorem, we can adjust the drift to make it be of the desired form by defining a new process $\widetilde{W}_{t}:=W_{t}+\int_{0}^{t}q_{s}ds$, where $\widetilde{W}_{t}$ is a Brownian motion under $\mathbb{Q}$, with $q_{t}:=\big(\mu(\lambda_{t},X_{t}^{(k)})-\widetilde{\mu}(X_{t}^{(k)})\big)/\nu(X_{t}^{(k)})$ (Girsanov kernel) admitting the expression:
\begin{equation}\label{eq:girsanov_kernel}
    \frac{\Big[\frac{aL^{2}}{2(1-k)}\lambda_{t}^{1-2k}-2b(1-k)X_{t}^{(k)}+\frac{\sigma^{2}L^{2}(1-2k)}{4(1-k)}\Big]-\Big[\frac{\sigma^{2}L^{2}}{4}-2b(1-k)X_{t}^{(k)}\Big]}{\sigma L \sqrt{X_{t}^{(k)}}}=\frac{a}{\sigma}\lambda_{t}^{-k}-\frac{k\sigma}{2}\lambda_{t}^{k-1}.
\end{equation}
Finally we can state the main result of \cite{ning2025ckls}.
\begin{theorem}\label{thm:main}[\cite{ning2025ckls}, \lbrack Section~2\rbrack]\\
Let the CKLS process $\lambda_{t}$ be defined under the measure $\mathbb{P}$ with $\frac{1}{2}\leq k<1$. Let the mapped process $X_{t}^{(k)}:=\mathcal{T}(\lambda_{t}\lvert k)$ with $\mathcal{T}$ defined by (\ref{eq:T_solution}). Define an equivalent probability measure $\mathbb{Q}$ by $\mathbb{Q}(B):=\mathbb{E}^{\mathbb{P}}\big[e^{\Lambda_{t}}\mathds{1}_{B}\big]$, for any $B\in\mathcal{F}_{t}$ via the Dol\'{e}ans--Dade exponential (not yet to be the Radon-Nikod\'{y}m derivative until it is verified that $\mathbb{E}^{\mathbb{P}}[\Lambda_{t}]=1$, i.e., $\Lambda_{t}$ is a (true) martingale):
\begin{equation}\label{Dol\'eans-Dade exponential}
    \Lambda_{t}:=\frac{d\mathbb{Q}}{d\mathbb{P}}\Big\lvert_{\mathcal{F}_{t}}=\exp\Big\{-\int_{0}^{t}q_{s}dW_{s}-\frac{1}{2}\int_{0}^{t}q_{s}^{2}ds\Big\},
\end{equation}
where $q_{t}$ is the Girsanov kernel given by (\ref{eq:girsanov_kernel}). Since it can be proven that $\Lambda_{t}$ is indeed a (true) martingale, under the measure $\mathbb{Q}$, the process $X_{t}^{(k)}$ follows the following transformed CIR-type SDE:
\begin{equation}\label{eq:final_cir}
    dX_{t}^{(k)}=\Big(\frac{\sigma^{2}L^{2}}{4}-2b(1-k)X_{t}^{(k)}\Big)dt+\sigma L\sqrt{X_{t}^{(k)}}d\widetilde{W}_{t},
\end{equation}
where $\widetilde{W}_{t}:=W_{t}+\int_{0}^{t}q_{s}ds$ is a standard $\mathbb{Q}$-Brownian motion.
\end{theorem}

\begin{table}[!ht]
\caption{Parameter correspondence between the CKLS and resulting CIR models}
\centering
\begin{tabular}{lccc}
\toprule
\textbf{CIR Parameter} & \textbf{Description} & \textbf{CKLS Parameter} & \textbf{Condition} \\ \hline
$\alpha$ & Level & $\frac{\sigma^{2} L^{2}}{4}$ & $\sigma>0$, $L>0$\\
$\beta$ & Mean-reversion speed & $2b(1-k)$ & $b>0, \frac{1}{2}\leq k<1$ \\
$\gamma$ & Volatility & $\sigma L$ & $\sigma>0, L>0$ \\ 
\bottomrule
\end{tabular}
\label{tab:param_map}
\end{table}

\begin{remark}[CEV special case, V\'a\v{s}\'i\v{c}ek extension and scaling parameter $L$]
\noindent \begin{enumerate}
    \item The three parameters of the resulting CIR-type model do not depend on the parameter $a$ of the CKLS model. Consequently, the transformation also applies to the CEV (constant elasticity of variance) model with $a$ set to $0$, which, like the CIR model, can be represented as a time-changed squared Bessel process.
    \item If the CIR coefficients satisfy $4\alpha=\gamma^{2}$, then the square-root transform $Y_{t}:=\sqrt{X_{t}}$ solves the V\'a\v{s}\'i\v{c}ek model $dY_{t}=-\frac{\beta}{2}Y_{t}dt+\frac{\gamma}{2}dW_{t}$, so that $Y_{t}= e^{-\frac{\beta}{2}t}Y_{0}+\frac{\gamma}{2}\int_{0}^{t}e^{-\frac{\beta}{2}(t-s)}dW_{s}$, i.e., $Y_{t}$ is a mean-reverting Gaussian process and is also called Ornstein--Uhlenbeck process. Notably, our transformed CIR-type SDE satisfies this condition. Since the further implications of this property are not directly relevant to the present paper, we do not elaborate on them here; see \cite{ning2025ckls} for details.
    \item The parameter $L=\gamma/\sigma>0$, introduced during the construction of the mapping $\mathcal{T}$, serves as a free scaling parameter. As shown in Table~\ref{tab:param_map}, the mean-reversion speed $\beta$ is independent of $L$, while both volatility $\gamma$ and level $\alpha$ are directly determined by $L$. This provides a valuable degree of freedom in model calibration: one can match the mean-reversion speed $\beta$ via the choice of $k$, while selecting $L$ to scale the volatility $\gamma$ and level $\alpha$ of the mapped process to desired values.
\end{enumerate}
\end{remark}

\subsubsection{On the Condition for Proving Martingale Property}
Theorem~\ref{thm:main} requires (\ref{Dol\'eans-Dade exponential}) to be a (ture) martingale in order to define a valid equivalent measure $\mathbb{Q}$. Classical criteria such as Novikov's and Kazamaki's would require integrability of exponential functionals involving negative powers of the CKLS process—conditions difficult to verify because they requires precise moment bounds and tail estimates for $\lambda_{t}$ over $t\in[0,+\infty)$. As discussed in \cite{ning2025ckls}, these tests are ineffective in this setting. Instead, using the martingale criterion proposed by \cite{mijatovic2012martingale}---a sufficient and necessary condition that bypasses Novikov--Kazamaki-type sufficient conditions, validated via Feller's boundary classification and the Feller's test of explosion originally by \cite{feller1952parabolic}---they prove that (\ref{Dol\'eans-Dade exponential}) is indeed a (true) martingale.

To apply the martingale criterion of \cite{mijatovic2012martingale}, $k$-regimes in which the CKLS process is not strictly positive must be excluded (known as Engelbert-Schmidt condition, see \cite{mijatovic2012martingale}). As a result, in addition to the restriction $\frac{1}{2}\leq k<1$, the case $k=\frac{1}{2}$ with $2a<\sigma^{2}$ should also be ruled out. Consequently, from this point onward, we impose the following parameter condition on the CKLS model:
\begin{equation}\label{prop:conditions2}
    \frac{1}{2}<k<1 \quad \text{ or }\quad k=\frac{1}{2}\text{ with }2a\geq\sigma^{2},
\end{equation}
so that $\mathcal{T}$ (\ref{eq:T_solution}) maps from $(0,+\infty)$ to $(0,+\infty)$.

\section{Estimation of \texorpdfstring{$\beta$}{beta} in CIR Model when Feller's Condition Fails}
This CKLS--CIR transformation produces only a subclass of the CIR model: the level parameter $\alpha=\sigma^{2}L^{2}/4$ is mechanically tied to the volatility $\gamma=\sigma L$, hence cannot be determined independently. A more serious drawback concerns Feller's condition (see Lemma~\ref{ergodicity-of-CIR}'s (\romannumeral1)) for strict positivity of the mapped process $X_{t}^{(k)}$, since $4\cdot(\sigma^{2}L^{2}/4)=(\sigma L)^{2},~\text{hence}~2\alpha=\frac{\gamma^{2}}{2}<\gamma^{2}$. Hence $X_{t}^{(k)}$ can hit $0$ with positive probability, introducing nontrivial boundary effects that complicate parameter estimation.
\subsection{Literature Review on CIR Model Drift Parameter Estimation}
\subsubsection{Precursor Studies Leading to \texorpdfstring{\cite{prykhodko2025discretization}}{Prykhodko et al.}'s Approach}
\noindent \noindent For continuous-time inference in the CIR model (with the full path $\{X_{t}\}_{t\in[0,T]}$ observed), \cite{alaya2012parameter} study the MLE based on the Girsanov-type likelihood ratio. In the three regimes $\beta>0$ (ergodic/mean-reverting), $\beta=0$ (nonergodic/critical) and $\beta<0$ (nonergodic/explosive), they obtain asymptotic results separately for $\hat{\alpha}^{\text{MLE}}_{T}:=\frac{\beta T+\int_{0}^{T}X_{t}^{-1}dX_{t}}{\int_{0}^{T}X_{t}^{-1}dt}$ and $\hat{\beta}^{\text{MLE}}_{T}:=\frac{\alpha T+X_{0}-X_{T}}{\int_{0}^{T}X_{t}dt}$. The key ingredients in the proof are the CIR process's representation as a squared Bessel process, the use of explicit joint Laplace transforms for suitable path functionals, and an analysis of their large-time asymptotics combined with the uniqueness of Laplace transforms to identify the limiting laws.

Sharing the same basic principle for constructing estimators from a continuous Girsanov-type likelihood ratio, \cite{overbeck1998estimation} study the CIR process within the framework of continuous branching processes with immigration and curved exponential families. The model is treated as a semimartingale statistical experiment, and inference is developed by exploiting its exponential-family structure. In this formulation, key functionals are regarded as natural sufficient statistics and components of the information process. Asymptotic behaviors of the MLEs are then characterized through the general theory of local asymptotic (mixed) normality (rather than explicit Laplace-transform calculations).

A subsequent study \cite{alaya2013asymptotic} further develops the joint asymptotic behavior of the MLE in the three $\beta$-regimes (equivalent expression derived via applying It\^{o}'s lemma to $X_{t}\mapsto \log X_{t}$):
\begin{equation}
\hspace{-0.2cm}
\begin{aligned}
    \hat{\alpha}^{\text{mle}}_{T}&=\frac{\int_{0}^{T}X_{t}dt\cdot\int_{0}^{T}\frac{1}{X_{t}}dX_{t}-T(X_{T}-X_{0})}{\int_{0}^{T}X_{t}dt\int_{0}^{T}\frac{1}{X_{t}}dt-T^{2}}\equiv\frac{\big(\log \frac{X_{T}}{X_{0}}+\gamma\int_{0}^{T}\frac{1}{X_{t}}dt\big)\int_{0}^{T}X_{t}dt-T(X_{T}-X_{0})}{\int_{0}^{T}X_{t}dt\int_{0}^{T}\frac{1}{X_{t}}dt-T^{2}},
    \\
    \hat{\beta}^{\text{mle}}_{T}&=\frac{(X_{0}-X_{T})\int_{0}^{T}\frac{1}{X_{t}}dt+T\int_{0}^{T}\frac{1}{X_{t}}dX_{t}}{\int_{0}^{T}X_{t}dt\int_{0}^{T}\frac{1}{X_{t}}dt-T^{2}}\equiv\frac{T\big(\log\frac{X_{T}}{X_{0}}+\gamma\int_{0}^{T}\frac{1}{X_{t}}dt\big)-(X_{T}-X_{0})\int_{0}^{T}\frac{1}{X_{t}}dt}{\int_{0}^{T}X_{t}dt \int_{0}^{T}\frac{1}{X_{t}}dt-T^{2}}.\label{AlayaDriftEstimators}
\end{aligned}
\end{equation}
Under the ergodic regime $\beta>0$, $(\hat{\alpha}^{\mathrm{mle}}_{T}, \hat{\beta}^{\mathrm{mle}}_{T})$ is $p$-consistent and asymptotically normal. The authors also discretize the first expression of \eqref{AlayaDriftEstimators} via the discretization scheme $dt\approx\Delta$, and demonstrate that the discretization error is $o_{p}(1)$. Most recently, \cite{chernova2024rate} revisits this problem by discretizing the second expression of \eqref{AlayaDriftEstimators} via the discretization scheme $dX_{t}\approx X_{t_{j+1}}-X_{t_{j}}$ and $dt\approx\Delta$, and demonstrates that the discretization error is $o_{p}(1)$.

All four studies above impose Feller's condition $2\alpha\geq\gamma^{2}$, which ensures
that the CIR process remains strictly positive and that $X_{t}^{-1}$ is integrable. In particular, $\mathbb{P}(\int_{0}^{T}X_{t}^{-1}dt<\infty)=\mathbb{P}(\int_{0}^{T}X_{t}^{-1}dX_{t}<\infty)=1$, so the continuous-time expressions (and its Riemann-sum approximations based on discrete observations) is well defined.

Instead of using an MLE approach, \cite{dehtiar2022two} propose a new continuous-time parametric estimator that exploits the positive Harris recurrence of the CIR process, and they establish its strong consistency. 
\cite{prykhodko2025discretization} further develops these results: in a high-frequency setting, they prove asymptotic normality of the continuous-time statistic, propose a discrete counterpart, and establish strong consistency and asymptotic normality for the discrete statistic as well. Their method addresses the limitation of conventional MLE approaches that are not applicable when $2\alpha<\gamma^{2}$. As a result, it is perfectly suited to the scenario considered in our study.

\subsubsection{Some Other Noteworthy Literature}
With equally spaced discrete observations in the ergodic case, \cite{overbeck1997estimation} define the (weighted) conditional least squares estimators (CLSE) for the drift by using the linear relation by treating the discrete-time skeleton of the CIR process as a linear AR(1) regression of $X_{j}$ on $X_{j-1}$. Within a martingale estimating-function framework they establish strong consistency and asymptotic normality of these discrete-time estimators, and also derive a continuous-time least squares estimator from the full trajectory as a benchmark for the loss of information under discrete sampling. Although they do not impose Feller's condition, a main drawback of the CLS
approach is that it exploits only the first-moment regression and thus ignores much of the likelihood information in the exact transition density.

\cite{tang2009parameter} propose a pseudo-MLE. In line with the CLS approach of \cite{overbeck1997estimation}, they bypass the intractable noncentral chi-square transition density by working with a discrete-time Gaussian approximation. They adopt the pseudo-likelihood based on \cite{bergstrom1984continuous}'s Euler--Maruyama-type discretization, which yields a Gaussian quasi-transition density that matches both the conditional mean and the conditional variance (second-moment regression). Their approach, too, hinges on Feller's condition to ensure $X_{t}$'s strict positivity ($\sum$ denotes $\sum_{j=1}^{n}$): 
\begin{equation}\label{pseudoMLE}
    \hspace{-0.1cm} \hat{\alpha}_{n}^{\text{pd}}:=\hat{\beta}_{n}^{\text{pd}}\chi,~\hat{\beta}_{n}^{\text{pd}}:=\frac{\log \phi}{-\Delta},\chi=\frac{\frac{1}{n}\sum X_{t_{j}}X_{t_{j-1}}^{-1}-\phi}{(1-\phi)\frac{1}{n}\sum X_{t_{j-1}}^{-1}},\phi=\frac{\frac{1}{n^{2}}\sum X_{t_{j}}\sum X_{t_{j-1}}^{-1}-\frac{1}{n}\sum X_{t_{j}}X_{t_{j-1}}^{-1}}{\frac{1}{n^{2}}\sum X_{t_{j-1}}\sum X_{t_{j-1}}^{-1}-1}.
\end{equation}

Under a high-frequency sampling scheme, \cite{cheng2024estimation} also replaces the exact noncentral chi-square transition density by a Gaussian
approximation via the Euler--Maruyama discretization and constructs a quasi-Gaussian likelihood together with the associated GQMLE. Although this likelihood has a simple analytic form, the GQMLE yields no closed-form solution. Starting from $p$-consistent initial estimators, they apply a one-step procedure to obtain $(\hat{\alpha}_{n},\hat{\beta}_{n},\hat{\gamma}_{n})$. After verifying the usual LAN conditions for the log-likelihood, LeCam's LAN theory yields asymptotic normality of these one-step estimators. The initial drift estimators $(\hat{\alpha}_{0,n},\hat{\beta}_{0,n})$ are chosen to be the CLSE of \cite{overbeck1997estimation}. Plugging them into the likelihood yields $\hat{\gamma}_{0,n}$. A standing assumption is the finiteness of $\mathbb{E}[X_{t}^{-1}]=\frac{2\beta_{0}}{2\alpha_{0}-\gamma_{0}^{2}}$, which enforces the stricter Feller condition $2\alpha>\gamma^{2}$. This condition is used repeatedly to ensure that $X_{t_{j}}^{-1}$ and $X_{t_{j}}^{-2}$ have finite moments and empirical averages of order $O_{p}(1)$; at one point an even stronger requirement $2\alpha>5\gamma^{2}$ is imposed to obtain sufficiently high-order (up to 5th) negative moments and the corresponding ergodic and tightness bounds.

\subsection{Prykhodko--Ralchenko's Estimator of \texorpdfstring{$\beta$}{beta}: A Brief Introduction}
\begin{theorem}\label{beta-theorem}[\cite{prykhodko2025discretization}, \lbrack Theorem 2.2\rbrack]\\
Assume a known diffusion coefficient $\gamma$. In the ultra-high-frequency setting, the estimator of the CIR drift parameters $(\alpha,\beta)$ based on observations $\{X_{i\Delta}\}_{i=0}^{n}$ are defined as follows:
\begin{equation}\label{beta-form}
    \hat{\alpha}_{n}:=\frac{\gamma^{2}}{2}\frac{(\sum_{i=1}^{n}X_{(i-1)\Delta})^{2}}{n\sum_{i=1}^{n}X_{(i-1)\Delta}^{2}-(\sum_{i=1}^{n}X_{(i-1)\Delta})^{2}},\quad
    \hat{\beta}_{n}:=\frac{\gamma^{2}}{2}\frac{n\sum_{i=1}^{n}X_{(i-1)\Delta}}{n\sum_{i=1}^{n}X_{(i-1)\Delta}^{2}-(\sum_{i=1}^{n}X_{(i-1)\Delta})^{2}}.
\end{equation}
Denote by $(\alpha_{0},\beta_{0})$ the true values of the drift parameters. In the ultra-high-frequency setting $\Delta\to0$, $n\Delta\to\infty$ and $n\Delta^{2}\to0$, the estimator $(\hat{\alpha}_{n},\hat{\beta}_{n})$ is strong consistent. Furthermore, in the same setting, the estimator is asymptotically normal as $n\to\infty$:
\begin{equation}\nonumber
    \sqrt{n\Delta}(\hat{\alpha}_{n}-\alpha_{0},\hat{\beta}_{n}-\beta_{0})\xrightarrow{d}\mathcal{N}(\textbf{0}_{2},\Sigma)~
\text{ with }\Sigma:=
\begin{pmatrix}
    \frac{\alpha_{0}}{\beta_{0}}(2\alpha_{0}+\gamma^{2}) & 2\alpha_{0}+\gamma^{2} \\
    2\alpha_{0}+\gamma^{2} & \frac{2\beta_{0}}{\alpha_{0}}(\alpha_{0}+\gamma^{2})
    \end{pmatrix}.
\end{equation}
\end{theorem}
\begin{proof}
In our scenario we are primarily concerned with estimating mean-reversion speed $\beta$. To keep the presentation focused, we omit the construction and proofs for $\alpha$, which proceed in an analogous way, and henceforth treat $\alpha$ as known with value $\alpha_{0}$. Tool Lemmas used are found in Appendix~\ref{bbb}.\\[0.5\baselineskip]
\noindent \textbf{(1) Strong consistency:} This is a combination of continuous mapping theorem and Lemma~\ref{ergod-law-first-moment}, Lemma~\ref{ergod-law-second-moment} and Lemma~\ref{discrete-continuous-error}'s (\romannumeral1), (\romannumeral2), (\romannumeral3). Specifically, as $n\to\infty$:
\begin{equation}\nonumber
    \hat{\beta}_{n}=\frac{\gamma^{2}}{2}\frac{\frac{1}{n\Delta}\sum_{i=1}^{n}X_{(i-1)\Delta}\Delta}{\frac{1}{n\Delta}\sum_{i=1}^{n}X_{(i-1)\Delta}^{2}\Delta-(\frac{1}{n\Delta}\sum_{i=1}^{n}X_{(i-1)\Delta}\Delta)^{2}}\xrightarrow{a.s.}\frac{\gamma^{2}}{2}\frac{\frac{\alpha}{\beta_{0}}}{\frac{\alpha^{2}}{\beta_{0}^{2}}+\frac{\alpha\gamma^{2}}{2\beta_{0}^{2}}-(\frac{\alpha}{\beta_{0}})^{2}}=\beta_{0}.
\end{equation}
When applying the continuous mapping theorem, note the denominator is the empirical variance $\hat{D}_{n}:=\frac{1}{n}\sum_{i=1}^{n}X_{(i-1)\Delta}^{2}-\big(\frac{1}{n}\sum_{i=1}^{n}X_{(i-1)\Delta}\big)^{2}$, which is non-negative by the Cauchy--Schwarz inequality. The denominator vanishes only when the event that all observations coincide happens. For a nondegenerate diffusion such a event has probability zero, so the denominator is strictly positive almost surely. In line with \cite{prykhodko2025discretization} and \cite{dehtiar2022two}, we do not consider the pathological case $\hat{D}_{n}=0$.\\[0.5\baselineskip]
\noindent \textbf{(2) Asymptotic normality:} First, we define $D_{T}$ and $\hat{D}_{n}$ as follows and obtain the respective convergence from Lemma~\ref{ergod-law-first-moment}, Lemma~\ref{ergod-law-second-moment} and the previous proof of strong convergence:
\begin{align}\nonumber
    D_{T}&:=\frac{1}{T}\int_{0}^{T}X_{t}^{2}dt-\Big(\frac{1}{T}\int_{0}^{T}X_{t}dt\Big)^{2}\xrightarrow[T\to\infty]{a.s.}\frac{\alpha(\alpha+\frac{\gamma^{2}}{2})}{\beta_{0}^{2}}-\frac{\alpha^{2}}{\beta_{0}^{2}}=\frac{\alpha\gamma^{2}}{2\beta_{0}^{2}}=:c=\text{Var}_{\pi_{0}}(X_{t}),\\
    \hat{D}_{n}&
    =\frac{1}{n\Delta}\sum_{i=1}^{n}X_{(i-1)\Delta}^{2}\Delta-\Big(\frac{1}{n\Delta}\sum_{i=1}^{n}X_{(i-1)\Delta}\Delta\Big)^{2}=\underbrace{(\hat{D}_{n}-D_{T})}_{\to0~a.s.}+\underbrace{D_{T}}_{\to c~ a.s.}\xrightarrow{a.s.}\frac{\alpha\gamma^{2}}{2\beta_{0}^{2}}.\nonumber
\end{align}
\noindent As a result, we have:
\begin{equation}
\begin{aligned}\nonumber
    &\sqrt{n\Delta}(\hat{\beta}_{n}-\beta_{0})=\sqrt{n\Delta}\Big[\frac{\frac{\gamma^{2}}{2}\frac{1}{n\Delta}\sum_{i=1}^{n}X_{(i-1)\Delta}\Delta}{\hat{D}_{n}}-\beta_{0}\Big]\\
    =&\frac{1}{\hat{D}_{n}}\Big[-\beta_{0}(n\Delta)^{-\frac{1}{2}}\sum_{i=1}^{n}X_{(i-1)\Delta}^{2}\Delta+\beta_{0}(n\Delta)^{-\frac{3}{2}}\big(\sum_{i=1}^{n}X_{(i-1)\Delta}\Delta\big)^{2}+\frac{\gamma^{2}}{2}(n\Delta)^{-\frac{1}{2}}\sum_{i=1}^{n}X_{(i-1)\Delta}\Delta\Big]\nonumber\\
    =&\frac{1}{\hat{D}_{n}}\Big[-\beta_{0}(n\Delta)^{-\frac{1}{2}}\int_{0}^{n\Delta}X_{t}^{2}dt+\beta_{0}(n\Delta)^{-\frac{3}{2}}\Big(\int_{0}^{n\Delta}X_{t}dt\Big)^{2}+\frac{\gamma^{2}}{2}(n\Delta)^{-\frac{1}{2}}\int_{0}^{n\Delta}X_{t}dt\\
    &~~~~~~~~~-\beta_{0}\xi_{2}(n)+\beta_{0}\xi_{3}(n)+\frac{\gamma^{2}}{2}\xi_{1}(n)\Big].\nonumber
\end{aligned}
\end{equation}
The definitions of $\xi_{l}$, $l=1,2,3$ can be found in Lemma~\ref{discrete-continuous-error}, whence $\xi_{l}(n)\xrightarrow{p}0$ for $l=1,2,3$. Denote the discretization error by $E(n):=-\beta_{0}\xi_{2}(n)+\beta_{0}\xi_{3}(n)+\frac{\gamma^{2}}{2}\xi_{1}(n)$, which yields $e(n):=\frac{E(n)}{\hat{D}_{n}}\xrightarrow{p}0$ as $n\to\infty$ when $n\Delta^{2}\to0$.\\[0.5\baselineskip]
\noindent Denote the continuous counterpart of $\hat{\beta}_{n}$ by
\begin{align}\nonumber
    \tilde{\beta}_{n}\equiv\tilde{\beta}_{T}:=\frac{\gamma^{2}}{2}\frac{\frac{1}{T}\int_{0}^{T}X_{t}dt}{\frac{1}{T}\int_{0}^{T}X_{t}^{2}dt-\frac{1}{T^{2}}(\int_{0}^{T}X_{t}dt)^{2}},
\end{align}
and note that, with $T\equiv n\Delta$:
\begin{align}\nonumber
    &\sqrt{T}(\tilde{\beta}_{n}-\beta_{0})=\sqrt{T}\Big(\frac{\frac{\gamma^{2}}{2}\frac{1}{T}\int_{0}^{T}X_{t}dt}{\frac{1}{T}\int_{0}^{T}X_{t}^{2}dt-(\frac{1}{T}\int_{0}^{T}X_{t}dt)^{2}}-\beta_{0}\Big)\\
    =&\frac{1}{D_{T}}\Big[\frac{\gamma^{2}}{2}T^{-\frac{1}{2}}\int_{0}^{T}X_{t}dt-\beta_{0}T^{-\frac{1}{2}}\int_{0}^{T}X_{t}^{2}dt+\beta_{0}T^{-\frac{3}{2}}(\int_{0}^{T}X_{t}dt)^{2}\Big]\nonumber\\
    =&\frac{1}{D_{T}}\Big[-\beta_{0}(n\Delta)^{-\frac{1}{2}}\int_{0}^{n\Delta}X_{t}^{2}dt+\beta_{0}(n\Delta)^{-\frac{3}{2}}\big(\int_{0}^{n\Delta}X_{t}dt\big)^{2}+\frac{\gamma^{2}}{2}(n\Delta)^{-\frac{1}{2}}\int_{0}^{n\Delta}X_{t}dt\Big]\nonumber\\
    =&\frac{\hat{D}_{n}}{D_{T}}\Big(\sqrt{n\Delta}(\hat{\beta}_{n}-\beta_{0})-e(n)\Big).\label{beta-decomp}
\end{align}
Since $\frac{\hat{D}_{n}}{D_{T}}\xrightarrow{a.s.}1$ as $n\to\infty$ and $T\to\infty$, $Q_{n}\xrightarrow{p}0$, we note that this continuous-discrete version relationship implies $\hat{\beta}_{n}-\tilde{\beta}_{n}=o_{p}(T^{-\frac{1}{2}})$. We see that it suffices to prove that the continuous counterpart of $\hat{\beta}_{n}$ is asymptotically normally distributed with the rate $\sqrt{n\Delta}$.\\[0.5\baselineskip]
\noindent A direct result from Lemma~\ref{continuous-convergence-in-Lq} is:
\begin{align}\nonumber
    &\sqrt{T}(\tilde{\beta}_{n}-\beta_{0})=\frac{1}{D_{T}}\Big[-\frac{\alpha(\alpha+\frac{\gamma^{2}}{2})}{\beta_{0}}T^{\frac{1}{2}}-\frac{\gamma(\alpha+\frac{\gamma^{2}}{2})}{\beta_{0}}T^{-\frac{1}{2}}\int_{0}^{T}X_{t}^{\frac{1}{2}}dW_{t}-\gamma T^{-\frac{1}{2}}\int_{0}^{T}X_{t}^{\frac{3}{2}}dW_{t}-\beta_{0}\rho_{2}(T)\nonumber\\
    &~~~~~~~~~~~~~~~~~~~~~~~~~~~+\frac{\alpha^{2}}{\beta_{0}}T^{\frac{1}{2}}+\frac{2\alpha\gamma}{\beta_{0}}T^{-\frac{1}{2}}\int_{0}^{T}X_{t}^{\frac{1}{2}}dW_{t}+\beta_{0}\rho_{3}(T)+\frac{\alpha\gamma^{2}}{2\beta_{0}}T^{\frac{1}{2}}+\frac{\gamma^{3}}{2\beta_{0}}T^{-\frac{1}{2}} \int_{0}^{T}X_{t}^{\frac{1}{2}}dW_{t}+\frac{\gamma^{2}}{2}\rho_{1}(T)\Big]\nonumber\\
    =&\frac{1}{D_{T}}\Big[\frac{\alpha\gamma}{\beta_{0}}T^{-\frac{1}{2}}\int_{0}^{T}X_{t}^{\frac{1}{2}}dW_{t}-\gamma T^{-\frac{1}{2}}\int_{0}^{T}X_{t}^{\frac{3}{2}}dW_{t}+\frac{\gamma^{2}}{2}\rho_{1}(T)+\beta_{0}\rho_{3}(T)-\beta_{0}\rho_{2}(T)\Big].\nonumber
\end{align}
The definitions of $\rho_{l}(T)\xrightarrow[T\to\infty]{p}0$ for $l=1,2,3$ can be found in Lemma~\ref{continuous-convergence-in-Lq}, whence $\rho_{l}(T)\xrightarrow[T\to\infty]{p}0$. Denote the estimation error $R(T):=\frac{1}{D_{T}}[\frac{\gamma^{2}}{2}\rho_{1}(T)+\beta_{0}\rho_{3}(T)-\beta_{0}\rho_{2}(T)]$, which yields $R(T)\xrightarrow[T\to\infty]{p}0$.\\[0.5\baselineskip]
\noindent As a result, we have the following representation:
\begin{equation}\label{beta-estimator}
\sqrt{T}(\tilde{\beta}_{n}-\beta_{0})
    =\frac{\gamma}{D_{T}}\Big(\frac{\alpha}{\beta_{0}},1\Big)
    \Big(\frac{1}{\sqrt{T}}\int_{0}^{T}X_{t}^{\frac{1}{2}}dW_{t},
    -\frac{1}{\sqrt{T}} \int_{0}^{T}X_{t}^{\frac{3}{2}}dW_{t}\Big)^{\intercal}+R(T).
\end{equation}
\noindent Note that the vector process 
$M_{T}:=\big(\int_{0}^{T}X_{t}^{\frac{1}{2}}dW_{t},
-\int_{0}^{T}X_{t}^{\frac{3}{2}}dW_{t}\big)^{\intercal}$ is a 2-dimensional Brownian motion (martingale) with the quadratic variation matrix
\begin{equation}\nonumber
    [M]_{T}=\int_{0}^{T}
    \begin{pmatrix}
    X_{t} & -X_{t}^{2}\\
    -X_{t}^{2} & X_{t}^{3}
    \end{pmatrix}dt=
    \begin{pmatrix}
    \int_{0}^{T}X_{t}dt & -\int_{0}^{T}X_{t}^{2}dt\\
    -\int_{0}^{T}X_{t}^{2}dt & \int_{0}^{T}X_{t}^{3}dt
    \end{pmatrix}.
\end{equation}
By Lemma~\ref{ergodicity-of-CIR}'s (\romannumeral3), we have for $q=1,2,3$:
\begin{equation}\nonumber
	\frac{1}{T}\int_{0}^{T}X_{t}^{q}dt\xrightarrow[T\to\infty]{a.s./\mathcal{L}^{1}}m_{q},\quad m_{1}:=\frac{\alpha}{\beta},~m_{2}:=\frac{\alpha(\alpha+\frac{\gamma^{2}}{2})}{\beta^{2}},~m_{3}:=\frac{\alpha(\alpha+\gamma^{2})(\alpha+\frac{\gamma^{2}}{2})}{\beta^{3}}.
\end{equation}
\noindent By the central limit theorem for multidimensional martingales, we obtain
\begin{equation}\nonumber
    \frac{1}{T}[M]_{T}\xrightarrow[T\to\infty]{p} \Gamma=
\begin{pmatrix}
    \frac{\alpha}{\beta_{0}} & -\frac{\alpha}{\beta_{0}^{2}}(\alpha+\frac{\gamma^{2}}{2}) \\
    -\frac{\alpha}{\beta_{0}^{2}}(\alpha+\frac{\gamma^{2}}{2}) & \frac{\alpha}{\beta_{0}^{3}}(\alpha+\gamma^{2})(\alpha+\frac{\gamma^{2}}{2})
\end{pmatrix}
~\Longrightarrow~
\frac{1}{\sqrt{T}}M_{T}\xrightarrow[T\to\infty]{d} \mathcal{N}(\textbf{0}_{2},\Gamma).
\end{equation}
For a precise statement of such CLT (for multidimensional martingales) with a discussion on how to apply it to the estimator $\tilde{\beta}_{n}$, see Theorem~\ref{theorem-heyde} in the next section.\\[0.5\baselineskip]
\noindent Taking into account the almost-sure convergence of $D_{T}$, that is, $D_{T}\xrightarrow[T\to\infty]{a.s.}\frac{\alpha\gamma^{2}}{2\beta_{0}^{2}}$ and Slutsky's theorem, we derive that $\sqrt{T}(\tilde{\beta}_{n}-\beta_{0})$ converges in law to the normal distribution $\mathcal{N}(0,\Sigma_{22})$ with
\begin{align}\nonumber
\Sigma_{22}&=
\Big(\gamma\frac{2\beta_{0}^{2}}{\alpha\gamma^{2}}\Big)^{2}\Big(\frac{\alpha}{\beta_{0}},1\Big)
\begin{pmatrix}
\frac{\alpha}{\beta_{0}} & -\frac{\alpha}{\beta_{0}^{2}}(\alpha+\frac{\gamma^{2}}{2}) \\
-\frac{\alpha}{\beta_{0}^{2}}(\alpha+\frac{\gamma^{2}}{2}) & \frac{\alpha}{\beta_{0}^{3}}(\alpha+\gamma^{2})(\alpha+\frac{\gamma^{2}}{2})
\end{pmatrix}
\Big(\frac{\alpha}{\beta_{0}},1\Big)=
\frac{2\beta_{0}}{\alpha} (\alpha+\gamma^{2}).\tag*{\qed}
\end{align}
\renewcommand{\qed}{}
\end{proof}
\begin{remark}[Efficiency comparisons of different estimators]
As discussed in \cite{prykhodko2025discretization}, one may compare the efficiency of $(\hat{\alpha}_{n},\hat{\beta}_{n})$ with those of the discrete-time MLE \eqref{AlayaDriftEstimators} proposed by \cite{alaya2013asymptotic}, of the discrete-time pseudo-MLE \eqref{pseudoMLE} proposed by \cite{tang2009parameter} and of the discrete-time QMLE proposed by \cite{cheng2024estimation}. All three MLE-type estimators are asymptotically normal and require Feller's condition 
$2\alpha\geq\gamma^{2}$. Their asymptotic covariance matrices are:
\[
\Sigma_{\mathrm{mle}}=\Sigma_{\mathrm{qmle}}=
\begin{pmatrix}
\dfrac{\alpha}{\beta_{0}}(2\alpha-\gamma^{2}) & 2\alpha-\gamma^{2}\\[0.6em]
2\alpha-\gamma^{2} & 2\beta_{0}
\end{pmatrix},
\quad
\Sigma_{\mathrm{pd}}=
\begin{pmatrix}
\dfrac{\alpha}{\beta_{0}}(2\alpha+4\beta_{0}+\gamma^{2}) & 2\alpha+2\beta_{0}\\[0.6em]
2\alpha+2\beta_{0} & 2\beta_{0}
\end{pmatrix},
\]
For the estimation of $\beta$, a comparison of $\Sigma$ respectively with $\Sigma_{\mathrm{mle}}=\Sigma_{\mathrm{qmle}}$ and $\Sigma_{\mathrm{pd}}$ shows that MLE, QMLE and pseudo-MLE are always more efficient than $\hat{\beta}_{n}$ whenever they are well-defined.
\end{remark}

\section{Stable Convergence in Law for the Estimator \texorpdfstring{$\hat{\beta}_{n}$}{beta-hat-n}}
In this section, we extend the results of \cite{prykhodko2025discretization} by strengthening the convergence of the discrete-time estimator $\hat{\beta}_{n}$ from convergence in law to stable convergence in law. This upgrade enables the use of Theorem~\ref{preservation of convergence}, which ensures that stable convergence with a normal limit is invariant under equivalent measure changes when the Radon-Nikod\'{y}m derivative is $\mathcal{F}_{T}$-measurable. Consequently, we can identify the asymptotic distribution of $\hat{\beta}_{n}$ under the new measure.

\subsection{Verifying Stable Convergence in Law for Estimator \texorpdfstring{$\hat{\beta}_{n}$}{beta-hat-n} via its Continuous-time Counterpart \texorpdfstring{$\tilde{\beta}_{n}$}{beta-tilde-n}}
We now investigate whether the convergence, as $n\to\infty$,
\begin{equation}\nonumber
    \sqrt{n\Delta}(\hat{\beta}_{n}-\beta_{0})\xrightarrow{d}\mathcal{N}\Big(0,\frac{2\beta_{0}}{\alpha}(\alpha+\gamma^{2})\Big).
\end{equation}
can be strengthened from convergence in law to stable convergence in law. Although the multidimensional martingale CLT de facto yields stable convergence in law, \cite{prykhodko2025discretization} did not discuss this refinement, as their analysis requires only ordinary convergence in law.

\begin{theorem}\label{theorem-heyde}[\cite{heyde1997quasi}, \lbrack Theorem 12.6\rbrack]\\
Let $S_{T}:=(S_{1T},\ldots,S_{dT})^{\intercal}$ with respect to $\mathcal{F}_{T}$ be a $d$-dimensional continuous martingale with quadratic variation matrix $[S]_{T}$. Suppose there exists a non-random vector function $\kappa_{T}=(\kappa_{1T},\ldots,\kappa_{dT})$ with $\kappa_{iT}>0$ increasing to infinity as $T\to\infty$ for $j=1,2,\ldots,d$ such that as $T\to\infty$:
\begin{enumerate}[(i)]
    \item (Lindeberg-type condition): $
	\kappa_{jT}^{-1}\sup_{0\leq s\leq T}\big\lvert S_{js}-S_{js-}\big\lvert \xrightarrow{p}0,\quad j=1,2,\ldots,d$.
\item (Quadratic variation convergence): $K_{T}^{-1}[S]_{T}K_{T}^{-1}\xrightarrow{p}\mathcal{X}_{1}$, where $K_{T}=\mathrm{diag}(\kappa_{1T},\ldots,\kappa_{dT})$ and $\mathcal{X}_{1}$ is a random nonnegative definite matrix.
\item (Second-moment convergence): $K_{T}^{-1}\big(\mathbb{E}\big[S_{T}S_{T}^{\intercal}\big]\big)K_{T}^{-1}\to\mathcal{X}_{2}$ where $\mathcal{X}_{2}$ is a positive definite matrix.
\end{enumerate}
Then the stable convergence in law holds: $K_{T}^{-1}S_{T} \xrightarrow[T\to\infty]{d_{st}}Y$,
where $Y$ is a normal variance mixture with characteristic function $\mathbb{E}\big[\exp\{-\frac{1}{2}u^{\intercal}\mathcal{X}_{1}u\}\big],~u\in\mathbb{R}^{d}$.
\end{theorem}

\begin{proposition}
Let $\hat{\beta}_{n}$ be the estimator of $\beta$ defined in (\ref{beta-form}) (true value $\beta_{0}$). Assume that the conditions of Theorem~\ref{theorem-heyde} hold. Then, as $n\to\infty$,
\begin{equation}\nonumber
    \sqrt{n\Delta}(\hat{\beta}_{n}-\beta_{0})\xrightarrow{d_{st}}\mathcal{N}\Big(0,\frac{2\beta_{0}}{\alpha}(\alpha+\gamma^{2})\Big).
\end{equation}
\end{proposition}
\begin{proof}
The idea is to utilize the equivalence between the discrete-time $\hat{\beta}_{n}$ and the continuous-time $\tilde{\beta}_{n}$ in (\ref{beta-decomp}):
\begin{equation}\nonumber
    \sqrt{T}(\tilde{\beta}_{n}-\beta_{0})=\frac{\hat{D}_{n}}{D_{T}}\Big(\sqrt{n\Delta}(\hat{\beta}_{n}-\beta_{0})-e(n)\Big).
\end{equation}
Recall the continuous 2-dimensional Brownian motion (martingale) with quadratic variation, with the notations $M_{1T}:=\int_{0}^{T}X_{s}^{\frac{1}{2}}dW_{s}$ and $M_{2T}:=-\int_{0}^{T}X_{s}^{\frac{3}{2}}dW_{s}$:
\begin{equation}\nonumber
    M_{T}=\Big(\int_{0}^{T}X_{s}^{\frac{1}{2}}dW_{s},-\int_{0}^{T}X_{s}^{\frac{3}{2}}dW_{s}\Big)^{\intercal}\quad\Longrightarrow\quad[M]_{T}=
    \begin{pmatrix}
    \int_{0}^{T}X_{t}dt & -\int_{0}^{T}X_{t}^{2}dt\\
    -\int_{0}^{T}X_{t}^{2}dt & \int_{0}^{T}X_{t}^{3}dt
    \end{pmatrix}.
\end{equation}
Define a $2\times2$ deterministic diagonal scaling matrix $S_{T}:=\sqrt{T}I_{2}=\text{diag}(S_{1T},S_{2T})$.\\[0.5\baselineskip]
\noindent \textit{Step~1. (Verifying conditions of Theorem~\ref{theorem-heyde})}\\
\noindent For (\romannumeral1): Since $X_{t}$ has continuous paths (see Lemma~\ref{ergodicity-of-CIR}'s (\romannumeral6), established results (see e.g., \cite{protter2005stochastic} [Chapter 6, Corollary 3, pp.73-74] imply that both components of $M_{T}$ are continuous $\mathcal{L}^{2}$-martingales. Hence $M_{t}$ has no jumps, and
\begin{equation}\nonumber
    S_{jT}^{-1}\sup_{0\leq s\leq T}\big\lvert M_{js}-M_{js-}\big\lvert=0 \xrightarrow{p}0,\qquad \text{for $j=1,2$}\quad\text{as $T\to\infty$},
\end{equation}
which verifies condition (\romannumeral1) of Theorem~\ref{theorem-heyde}.\\[0.5\baselineskip]
\noindent For (\romannumeral2): Recall that
\begin{equation}\nonumber
    \frac{1}{T} [M]_{T} \xrightarrow[T\to\infty]{p}\Gamma=
\begin{pmatrix}
    \frac{\alpha}{\beta_{0}} & -\frac{\alpha}{\beta_{0}^{2}}(\alpha+\frac{\gamma^{2}}{2}) \\
    -\frac{\alpha}{\beta_{0}^{2}}(\alpha+\frac{\gamma^{2}}{2}) & \frac{\alpha}{\beta_{0}^{3}}(\alpha+\gamma^{2})(\alpha+\frac{\gamma^{2}}{2})
\end{pmatrix}
    =:\begin{pmatrix}
        m_{1} & -m_{2} \\
        -m_{2} & m_{3}
    \end{pmatrix}.
\end{equation}
Since $S_{T}^{-1}=T^{-\frac{1}{2}}I_{2}$, we have
\begin{equation}\nonumber
    S_{T}^{-1}[M]_{T} S_{T}^{-1}
    =\frac{1}{T}[M]_{T}
    \xrightarrow[T\to\infty]{p}
    \Gamma.
\end{equation}
The limit matrix $\Gamma$ is positive definite: by Cauchy--Schwarz inequality and the nondegeneracy of $X_{t}\geq 0$, $m_{2}^{2}=\big(\mathbb{E}[X_{t}^{2}]\big)^{2}<\mathbb{E}[X_{t}]\mathbb{E}[X_{t}^{3}]=m_{1}m_{3}$, so $\det(\Gamma)=m_{1}m_{3}-m_{2}^{2}>0$.  
Hence condition (\romannumeral2) of Theorem~\ref{theorem-heyde} is verified with deterministic limit $\Gamma$.\\[0.5\baselineskip]
\noindent For (\romannumeral3): Since $M_{T}$ is continuous, $\mathbb{E}[M_{T}M_{T}^{\intercal}]=\mathbb{E}\big[[M]_{T}\big]$.  
Using $\frac{1}{T}[M]_{T}\xrightarrow[T\to\infty]{p}\Gamma$, we get $S_{T}^{-1}\mathbb{E}[M_{T}M_{T}^{\intercal}]S_{T}^{-1}=\mathbb{E}\big[T^{-1}[M]_{T}\big]\to\Gamma$, thus condition (iii) of Theorem~\ref{theorem-heyde} holds.\\[0.5\baselineskip]
\noindent Since limits in (\romannumeral2) and (\romannumeral3) coincide, the normal variance mixture limit reduces to a fixed Gaussian one:
\begin{equation}\label{eq:vector-stable}
    \frac{M_{T}}{\sqrt{T}}\xrightarrow[T\to\infty]{d_{st}}Y,\quad Y\sim\mathcal{N}(\bm{0}_{2},\Gamma).
\end{equation}
\noindent \textit{Step~2 (Projection and stable limit for $\hat{\beta}_{n}$)}\\
\noindent Let $u:=(\alpha/\beta_{0},-1)^{\intercal}$.  By the stable limit (\ref{eq:vector-stable}) and Cram\'{e}r--Wold theorem:
\begin{equation}\nonumber
	u^{\intercal}\frac{M_{T}}{\sqrt{T}}\xrightarrow[T\to\infty]{d_{st}}\mathcal{N}\big(0,u^{\intercal}\Gamma u\big),
\end{equation}
and a direct substitution of $(m_{1},m_{2},m_{3})$ yields
\begin{equation}\label{u-eta-u}
	u^{\intercal}\Gamma u=\Big(\frac{\alpha}{\beta_{0}}\Big)^{2}m_{1}-2\frac{\alpha}{\beta_{0}}m_{2}+m_{3}=\frac{\alpha^{2}}{\beta_{0}^{2}}\frac{\alpha}{\beta_{0}}-\frac{2\alpha}{\beta_{0}}\frac{\alpha(\alpha+\frac{\gamma^{2}}{2})}{\beta_{0}^{2}}+\frac{\alpha(\alpha+\gamma^{2})(\alpha+\frac{\gamma^{2}}{2})}{\beta_{0}^{3}}=\frac{\alpha\gamma^{2}}{2\beta_{0}^{3}}(\alpha+\gamma^{2}).
\end{equation}
Next, we can rewrite $\tilde{\beta}_{n}$ (\ref{beta-estimator}) as:
\begin{equation}\nonumber
    \sqrt{T}(\tilde{\beta}_{n}-\beta_{0})=\frac{\gamma}{D_{T}\sqrt{T}}\int_{0}^{T}\Big(\frac{\alpha}{\beta_{0}}X_{t}^{\frac{1}{2}}-X_{t}^{\frac{3}{2}}\Big)dW_{t}+R(T)~\Longleftrightarrow~\sqrt{T}(\tilde{\beta}_{n}-\beta_{0})=\frac{\gamma}{D_{T}}u^{\intercal}\frac{M_{T}}{\sqrt{T}}+R(T).
\end{equation}
Recall that $D_{T}\xrightarrow{a.s.}c=\frac{\alpha\gamma^{2}}{2\beta_{0}^{2}}>0$ and $R(T)\xrightarrow{p}0$ as $T\to\infty$. Combining (\ref{eq:vector-stable}) and (\ref{u-eta-u}) and applying the stable Slutsky's theorem~\ref{stable-slutsky}'s (\romannumeral3) yields, as $n\to\infty$,
\begin{equation}\nonumber
	\sqrt{T}(\tilde{\beta}_{n}-\beta_{0})\xrightarrow{d_{st}}\mathcal{N}\Big(0,\frac{\gamma^{2}}{c^{2}}u^{\intercal}\Gamma u\Big)=\mathcal{N}\Big(0,\frac{2\beta_{0}}{\alpha}(\alpha+\gamma^{2})\Big).
\end{equation}
Since $T\equiv n\Delta$ and $\hat{\beta}_{n}-\tilde{\beta}_{T}=o_{p}\big(T^{-\frac{1}{2}}\big)$ by (\ref{beta-decomp}), another stable Slutsky argument yields
\begin{equation}\nonumber
	\sqrt{n\Delta}(\hat{\beta}_{n}-\beta_{0})\xrightarrow{d_{st}}\mathcal{N}\Big(0,\frac{2\beta_{0}}{\alpha}(\alpha+\gamma^{2})\Big)\qquad\text{as }n\to\infty.\qedhere
\end{equation}
\renewcommand{\qed}{}
\end{proof}

\subsection{Stable Convergence in Law under Changes of Measure}
Since we have established the stable convergence in law for the sequence $\{{\hat{\beta}_{n}}\}_{n\geq 1}$ to an asymptotically normal limit, in our CKLS--CIR scenario, it is natural to ask how this limit behaves under an equivalent change of measure. This subsection analyzes how a stably convergent sequence with a known limit distribution is affected by an equivalent change of measure. 

We first state LeCam's third lemma, followed by a sufficient condition ensuring invariance of stable convergence under an equivalent measure when the limit is normal—a specialization of the lemma tailored for our scenario.

\begin{lemma}[LeCam's Third Lemma]\label{LeCam}
Let $(\Omega_{n},\mathcal{F}_{n})$ be measurable spaces, $\{\mathbb{P}_{n}\}_{n\geq 1}$ and $\{\mathbb{Q}_{n}\}_{n\geq 1}$ be contiguous sequences of probability measures on $(\Omega_{n},\mathcal{F}_{n})$ with the likelihood ratios $\Lambda_{n}:=d\mathbb{Q}_{n}/d\mathbb{P}_{n}$. Let $Y_{n}$  be a sequence of $\mathbb{R}^{q}$-valued statistics defined on $(\Omega_{n},\mathcal{F}_{n})$. Assume that under $\mathbb{P}_{n}$, as $n\to\infty$, we have the joint convergence in law
\[
(Y_{n},\Lambda_{n})\xrightarrow{d}(Y,\Lambda)
\quad\text{ on }\mathbb{R}^{q}\times[0,+\infty)\quad\text{ with }\mathbb{E}[\Lambda]=1.
\]
Then, for every bounded continuous $g:\mathbb{R}^{q}\to\mathbb{R}$, $\lim_{n\to\infty}\mathbb{E}^{\mathbb{Q}_{n}}[g(Y_{n})]
=\mathbb{E}[g(Y)\Lambda]=\int_{\mathbb{R}^{q}\times[0,+\infty)}g(y)zd\mu(y,z)$ where $\mu$ denotes the joint probability measure of $(Y,\Lambda)$. Equivalently speaking, under $\mathbb{Q}_{n}$, the law of $Y_{n}$ converges to the $\Lambda$-tilted distribution of $Y$.
\end{lemma}
\begin{proof}
    See e.g., \cite{van2002statistical} [Lemma 3.1], pp.641.
\end{proof}
\begin{remark}[Interpretation of LeCam's third lemma]
The lemma explains how the limit distribution of a statistic changes when switching from a reference sequence of measures $\{\mathbb{P}_{n}\}_{n\geq 1}$ to another $\{\mathbb{Q}_{n}\}_{n\geq 1}$ that is locally close to $\{\mathbb{P}_{n}\}$ in the sense of likelihood ratios. Under $\mathbb{Q}_{n}$, the limiting law of $Y_{n}$ is obtained by exponentially tilting the joint limit $(Y,\Lambda)$ under $\mathbb{P}_{n}$ with the likelihood ratio $\Lambda$. More precisely, under $\mathbb{P}_{n}$, suppose we know the joint limit $(Y,\Lambda)$ of $(Y_{n},\Lambda_{n})$. LeCam's third lemma asserts that the limit law of $Y_{n}$ under $\mathbb{Q}_{n}$ is obtained simply by $\Lambda$-tilting this limit; no separate asymptotic analysis under $\mathbb{Q}_{n}$ is required.
\end{remark}

\noindent To apply Lemma~\ref{LeCam} in our scenario, we take 
$\Omega_{n}\equiv\Omega$, $\mathcal{F}_{n}\equiv\mathcal{F}_{T}$,
$\mathbb{P}_{n}=\mathbb{P}$, $\mathbb{Q}_{n}\equiv\mathbb{Q}$,
so that $\Lambda_{n}\equiv \Lambda_{T}=\frac{d\mathbb{Q}}{d\mathbb{P}}\big\lvert_{\mathcal{F}_{T}}$. As $n\to\infty$, if $Y_{n}\xrightarrow{d_{st}}Y$ with respect to $\mathcal{F}_{T}$ , then for any bounded $\mathcal{F}_{T}$-measurable $\zeta$ we have $(Y_{n},\zeta)\xrightarrow{d}(Y,\zeta)$ under $\mathbb{P}$. Since $\Lambda_{T}$ is $\mathcal{F}_{T}$-measurable, 
$(Y_{n},\Lambda_{T})\xrightarrow{d}(Y,\Lambda_{T})$ follows. The next theorem is an immediate corollary of Lemma~\ref{LeCam}.

\begin{theorem}\label{preservation of convergence}[\cite{mykland2009inference}, \lbrack Proposition 1\rbrack]\\
Let $(\Omega,\mathcal{F},(\mathcal{F}_{t})_{t\in[0,T]},\mathbb{P})$ be a filtered probability space, $\Lambda_{T}:=\frac{d\mathbb{Q}}{d\mathbb{P}}\big\lvert_{\mathcal{F}_{T}}$ be the likelihood ratio of a measure $\mathbb{Q}$ equivalent to $\mathbb{P}$ on $(\Omega,\mathcal{F}_{T})$; in particular, $\Lambda_{T}\geq 0$, $\mathcal{F}_{T}$-measurable, $\Lambda_{T}\in\mathcal{L}^{1}(\mathbb{P})$, $\mathbb{E}^{\mathbb{P}}[\Lambda_{T}]=1$. Suppose $Y_{n}\xrightarrow{d_{st}}Y$ as $n\to\infty$ under $\mathbb{P}$ with respect to $\mathcal{F}_{T}$, $Y$ is a normal variance mixture:
\begin{equation}\nonumber
	Y=U+V\varphi,\quad U,V\text{ are }\mathcal{F}_{T}\text{-measurable, and }\varphi\sim\mathcal{N}(0,1),\varphi\!\perp\!\!\!\perp \mathcal{F}_{T}\text{ (on a product extension)}.
\end{equation}
Then $Y_{n}\xrightarrow{d_{st}}Y$ as $n\to\infty$ remains valid under $\mathbb{Q}$ with respect to $\mathcal{F}_{T}$, and $\varphi$ remains independent of $\mathcal{F}_{T}$ under $\mathbb{Q}$. Moreover, the reverse implication (from $\mathbb{Q}$ to $\mathbb{P}$) also holds.
\end{theorem}
\begin{proof}
    See Appendix~\ref{ccc}.
\end{proof}

\begin{remark}[Invariance of mixed normal limits under equivalent measures]
For a sequence stably converging to a mixed normal limit, both its consistency and its asymptotic (mixed normal) distribution are invariant under equivalent measures. Since $Y=U+V\varphi$ is a normal variance mixture, $U$ and $V$ are not necessarily deterministic; they may be stochastic and even admit different asymptotic behaviors under two equivalent measures. The mixed normal limit reduces to an ordinary normal law only when both $U$ and $V$ are constants. Moreover, the $\mathcal{F}_{T}$-measurability of $\Lambda_{T}$ ensures that the measure change only affects the "past" information contained in $\mathcal{F}_{T}$, without tilting the independent noise component in the mixed normal limit.

Specifically, suppose that $\theta$ is the parameter of interest (true value $\theta_{0}$), $U\in\mathbb{R},~V\in\mathbb{R}_{+}$. Let $\hat{\theta}_{n}$ be an estimator constructed under the probability measure $\mathbb{Q}$, and assume that the following asymptotic result (convergence rate $\delta_{n}$) has been established under $\mathbb{Q}$:
\begin{equation}\nonumber
    \sqrt{\delta_{n}}(\hat{\theta}_{n}-\theta_{0})\xrightarrow{d_{st}}\mathcal{N}(U,V)\qquad \text{as }n\to\infty.
\end{equation}
Since $\mathbb{P}$ is obtained from $\mathbb{Q}$ via a general Girsanov transformation, the Radon-Nikod\'{y}m density is automatically $\mathcal{F}_{T}$-measurable and belongs to $\mathcal{L}^{1}$. It then follows that the same sequence $\sqrt{\delta_{n}}(\hat{\theta}_{n}-\theta_{0})$ also converges stably in law to the identical law under $\mathbb{P}$. In particular, the consistency of $\hat{\theta}_{n}$ and the convergence rate $\sqrt{\delta_{n}}$ are unaffected by the change of measure. 
\end{remark}
\begin{lemma}[Selected Results for Stable Convergence in Law]\label{stable-slutsky}
Assume $Z_{n}\xrightarrow{d_{st}}Z$ as $n\to\infty$.
\begin{enumerate}[(i)]
    \item \textbf{Joint convergence with $\mathcal{F}$-measurable limits}: If $\zeta_{n}\xrightarrow{p}\zeta$, then $(Z_{n},\zeta_{n})\xrightarrow{d_{st}}(Z,\zeta)$.
    \item \textbf{Continuous mapping theorem}: If $g\in\mathcal{C}(\mathbb{R})$, then $g(Z_{n})\xrightarrow{d_{st}}g(Z)$.
    \item \textbf{Stable Slutsky's Theorem}: If $g\in\mathcal{C}(\mathbb{R}^{2})$, then $g(Z_{n},\zeta_{n})\xrightarrow{d_{st}}g(Z,\zeta)$. As a result, if $\zeta_{n}\xrightarrow{p}c\in\mathbb{R}$, then $Z_{n}+\zeta_{n}\xrightarrow{d_{st}}Z+c;~ Z_{n}\zeta_{n}\xrightarrow{d_{st}}cZ;~ Z_{n}/\zeta_{n}\xrightarrow{d_{st}}Z/c~(c\neq 0)$. If $\zeta_{n}\xrightarrow{p}\zeta$ with $\zeta$ being $\mathcal{F}$-measurable and bounded away from $0$, then $Z_{n}/\zeta_{n}\xrightarrow{d_{st}}Z/\zeta$.
\end{enumerate}
\end{lemma}
\begin{proof}
    See Appendix~\ref{ccc}.
\end{proof}

\section{An Initial Estimator for Elasticity for Data Mapping}
To map the observed CKLS data $\{\lambda_{i\Delta}\}_{i=0}^{n}$ of (\ref{CKLS-process}) (under the original measure $\mathbb{P}$) into the transformed/pseudo CIR data $\{X_{i\Delta}\}_{i=0}^{n}$ of (\ref{CIR-process}) (used under the equivalent measure $\mathbb{Q}$), we rely on the mapping (\ref{eq:T_solution}) $\mathcal{T}(\cdot\lvert k)$, which depends on the elasticity $k$. Consequently, exploiting the estimation-ready CIR data for inference on $k$ requires an initial estimate of $k$; without such an estimate, the mapping from $\{\lambda_{i\Delta}\}_{i=0}^{n}$ to $\{X_{i\Delta}\}_{i=0}^{n}$ cannot be carried out.

In this section, we adopt the estimator (denoted as $\hat{k}_{0,n}$) for the elasticity parameter proposed by \cite{mishura2022parameter} (see also \cite{lyu2025inference}). They establish strong consistency of $\hat{k}_{0,n}$. We strengthen their result by deriving its convergence rate (in probability) in the ultra-high-frequency setting. With $\hat{k}_{0,n}$, the mapped data are given by $X_{i\Delta}^{(\hat{k}_{0,n})}:=\mathcal{T}\big(\lambda_{i\Delta}\big\lvert\hat{k}_{0,n}\big)$, $i=0,\dots,n$.

\subsection{A Model-based Estimator by \texorpdfstring{\cite{mishura2022parameter}}{Mishura et al. (2022)} for Elasticity: A Brief Introduction}
By the definition of the (realized) quadratic variation of a stochastic process $\lambda_{t}$, we obtain the following result.
\begin{lemma}
The quadratic variation of the CKLS process $\lambda_{t}$ is defined as the limit of its realized counterpart as $n\to\infty$: $[\lambda]_{\tau}:=\lim_{n\to\infty}\sum_{i=1}^{n}\big(\lambda_{\frac{i}{n}\tau}-\lambda_{\frac{i-1}{n}\tau}\big)^{2}$, then
\begin{equation}\nonumber
	\frac{\log \frac{[\lambda]_{\tau+\Delta}-[\lambda]_{\tau}}{\sigma^{2}\Delta}}{2\log\lambda_{\tau}}\xrightarrow[\Delta\to0]{a.s.}k \qquad\text{ for any fixed }\tau\in[0,T]. 
\end{equation}
\end{lemma} 
\begin{proof}
Under condition (\ref{prop:conditions2}), we have $\lambda_{t}>0~a.s.$ for any $t\in[0,T]$, and hence $\log\lambda_{\tau}$ is finite almost surely. Since $\lambda_{\tau}$ admits a (Lebesgue) density for any fixed $\tau\in[0,T]$, we have $\mathbb{P}(\lambda_{\tau}=1)=0$. As $\log\lambda_{\tau}=0$ if and only if $\lambda_{\tau}=1$, it follows that $\mathbb{P}(\log\lambda_{\tau}\neq 0)=1$. \\[0.5\baselineskip]
By the quadratic variation identity for It\^{o} integrals, $[\lambda]_{\tau}=\sigma^{2}\int_{0}^{\tau}\lambda_{s}^{2k}ds$.
As $t\mapsto\lambda_{t}$ is almost surely continuous at $t=\tau$, local averages over $[\tau,\tau+\Delta]$ converge to the point value:
\begin{equation}\label{eq:convergence a.s.}
	\frac{[\lambda]_{\tau+\Delta}-[\lambda]_{\tau}}{\Delta}=\frac{\sigma^{2}}{\Delta}\int_{\tau}^{\tau+\Delta}\lambda_{s}^{2k}ds\xrightarrow[\Delta\to0]{a.s.}\sigma^{2}\lambda_{\tau}^{2k}\quad\Longrightarrow\quad\log\frac{[\lambda]_{\tau+\Delta}-[\lambda]_{\tau}}{\sigma^{2}\Delta}\xrightarrow[\Delta\to0]{a.s.} 2k\log\lambda_{\tau}.\qedhere
\end{equation}
\renewcommand{\qed}{}
\end{proof}
\noindent Consider this convergence at two different fixed time points $\tau$ and $\tau'$ in $[0,T]$. The following alternative convergence allows us to cancel $\sigma$:
\begin{equation}\nonumber
	\frac{[\lambda]_{\tau+\Delta}-[\lambda]_{\tau}}{[\lambda]_{\tau'+\Delta}-[\lambda]_{\tau'}}=\frac{\sigma^{2}\int_{\tau}^{\tau+\Delta}\lambda^{2}_{s}ds}{\sigma^{2}\int_{\tau}^{\tau'+\Delta}\lambda_{s}^{2}ds}\xrightarrow[\Delta\to 0]{a.s.}\frac{\lambda_{\tau}^{2k}}{\lambda_{\tau'}^{2k}}\quad\Longrightarrow\quad\frac{\log\frac{[\lambda]_{\tau+\Delta}-[\lambda]_{\tau}}{[\lambda]_{\tau'+\Delta}-[\lambda]_{\tau'}}}{2\log\frac{\lambda_{\tau}}{\lambda_{\tau'}}}\xrightarrow[\Delta\to0]{a.s.}k.
\end{equation}
As an immediate corollary of the above analysis, we obtain the following two kinds of estimators proposed by \cite{mishura2022parameter}.

\begin{theorem}\label{initial-etimators}[\cite{mishura2022parameter}, \lbrack 
Proposition 5.1, Proposition 5.2\rbrack]\\
Denote the true value of $k$ by $k_{0}$. Having data observed at two fixed time points $\tau:=(i-1)\Delta$ and $\tau'=(j-1)\Delta$, we obtain the following two estimators of $k$ that both converge to $k_{0}$ almost surely:
\begin{equation}\nonumber
	\hat{k}_{0,n}^{(1)}:=\frac{\log\frac{[\lambda]_{i\Delta}-[\lambda]_{(i-1)\Delta}}{\sigma^{2}\Delta}}{2\log\lambda_{(i-1)\Delta}}\xrightarrow[\Delta\to0]{a.s.}k_{0},\qquad
    \hat{k}_{0,n}^{(2)}:=\frac{\log\frac{[\lambda]_{i\Delta}-[\lambda]_{(i-1)\Delta}}{[\lambda]_{j\Delta}-[\lambda]_{(j-1)\Delta}}}{2\log\frac{\lambda_{(i-1)\Delta}}{\lambda_{(j-1)\Delta}}}\xrightarrow[\Delta\to0]{a.s.}k_{0}.
\end{equation}
\end{theorem}

\begin{remark}[Technical advice on robustification near the boundary $\lambda_{\tau}\approx 1$]
Following \cite{mishura2022parameter} (see also \cite{lyu2025inference}), robust estimation of $k$ requires excluding time points $\tau$ with $\lambda_{\tau}\approx 1$ and pairs $(\tau,\tau')$ with $\lambda_{\tau}/\lambda_{\tau'}\approx 1$. In these cases the logarithmic terms in the denominators of the pointwise estimators become nearly zero, which amplifies numerical noise and produces unrobust estimates. It is therefore recommended to choose time points $\{\tau_{1},\dots,\tau_{m}\}$ for which $\lambda_{\tau_{s}}$ is bounded away from $1$ (e.g., choose some threshold $\varepsilon>0$ such that $\lvert 1-\lambda_{\tau_{s}}\lvert\geq\varepsilon$), and to use the aggregated statistics:
\begin{equation}\nonumber
	\tilde{k}_{0,n,m}^{(1)}:=\frac{\sum_{s=1}^{m}\Big\lvert\log\frac{[\lambda]_{\tau_{s}+\Delta}-[\lambda]_{\tau_{s}}}{\sigma^{2}\Delta}\Big\lvert}{2\sum_{s=1}^{m}\big\lvert\log\lambda_{\tau_{s}}\big\lvert},\qquad \tilde{k}_{0,n,m}^{(2)}:=\frac{\sum_{s=1}^{m}\Big\lvert\log\frac{[\lambda]_{\tau_{s}+\Delta}-[\lambda]_{\tau_{s}}}{[\lambda]_{\tau'_{s}+\Delta}-[\lambda]_{\tau'_{s}}}\Big\lvert}{2\sum_{s=1}^{m}\Big\lvert\log\frac{\lambda_{\tau_{s}}}{\lambda_{\tau'_{s}}}\Big\lvert}.
\end{equation}
By (\ref{eq:convergence a.s.}), both $\tilde{k}_{0,n,m}^{(1)}$ and $\tilde{k}_{0,n,m}^{(2)}$ converge almost surely to $k_{0}$ as $\Delta\to 0$. Compared with naive averaging of pointwise estimators, these aggregated estimators remain robust even when some $\lambda_{\tau_{s}}$ lie close to $1$, exhibiting better numerical performance in practical applications.
\end{remark}

\subsection{Convergence Rate of the Initial Estimator}
Next, we state the convergence rate of the initial estimator in the high-frequency setting, as this result will be essential for the analysis in the next section.
\begin{theorem}\label{k-rate}[Convergence rate of the initial elasticity estimator]\\
For the estimator defined in Theorem~\ref{initial-etimators}, $\ell=1,2$
\begin{equation}\nonumber
    \hat{k}_{0,n}^{(\ell)}-k_{0}=O_{p}(\Delta^{\frac{1}{2}}).
\end{equation}
If $n\Delta^{2}\to0$ is further assumed, since $\Delta^{\frac{1}{2}}/(n\Delta)^{-\frac{1}{2}}=n\Delta^{2}\to0$, then
\begin{equation}\nonumber
    \hat{k}_{0,n}^{(\ell)}-k_{0}=o_{p}\big((n\Delta)^{-\frac{1}{2}}\big).
\end{equation}
\end{theorem}

\begin{proof}
Recall the convergence in (\ref{eq:convergence a.s.}), we define the convergence error by
\begin{equation}\nonumber
	r_{\tau}(\Delta):=[\lambda]_{\tau+\Delta}-[\lambda]_{\tau}-\sigma^{2}\lambda_{\tau}^{2k}\Delta=\sigma^{2}\int_{\tau}^{\tau+\Delta}\lambda_{s}^{2k}ds-\sigma^{2}\lambda_{\tau}^{2k}\Delta=\sigma^{2}\int_{\tau}^{\tau+\Delta}(\lambda_{s}^{2k}-\lambda_{\tau}^{2k})ds.
\end{equation}
If we can establish the uniform $1/2$-H\"{o}lder continuity in probability of $t\mapsto\lambda_{t}^{2k}$:
\begin{equation}\nonumber
    \lvert \lambda_{s}^{2k}-\lambda_{\tau}^{2k}\lvert=O_{p}(\lvert s-\tau\lvert^{\frac{1}{2}})=O_{p}(\Delta^{\frac{1}{2}})\quad \text{for any $\tau\in[0,T]$, $s\in[\tau,\tau+\Delta]\subset [0,T]$},
\end{equation}
then the following will hold
\begin{equation}\nonumber
    r_{\tau}(\Delta)=\int_{\tau}^{\tau+\Delta}O_{p}(\lvert s-\tau\lvert^{\frac{1}{2}})ds=O_{p}(\Delta^{\frac{3}{2}})\quad\Longrightarrow\quad[\lambda]_{\tau+\Delta}-[\lambda]_{\tau}=\sigma^{2}\lambda_{\tau}^{2k}\Delta+O_{p}(\Delta^{\frac{3}{2}}).\nonumber
\end{equation}
This yields
\begin{equation}\label{error v}
    \frac{[\lambda]_{\tau+\Delta}-[\lambda]_{\tau}}{\sigma^{2}\Delta}=\lambda_{\tau}^{2k}+O_{p}(\Delta^{\frac{1}{2}}),~
    \frac{[\lambda]_{\tau+\Delta}-[\lambda]_{\tau}}{[\lambda]_{\tau'+\Delta}-[\lambda]_{\tau'}}=\frac{\lambda_{\tau}^{2k}\Delta\big(1+O_{p}(\Delta^{\frac{1}{2}})\big)}{\lambda_{\tau'}^{2k}\Delta\big(1+O_{p}(\Delta^{\frac{1}{2}})\big)}\overset{\ast}{=}\frac{\lambda_{\tau}^{2k}}{\lambda_{\tau'}^{2k}}\Big(1+O_{p}(\Delta^{\frac{1}{2}})\Big),
\end{equation}
where $\ast$ holds because for $a_{n}=O_{p}(\Delta^{\frac{1}{2}})$ and $b_{n}=O_{p}(\Delta^{\frac{1}{2}})$, thus we have $b_{n}\to 0$ in probability, so $\frac{1+a_{n}}{1+b_{n}}=(1+a_{n})\big(1-b_{n}+O_p(\Delta)\big)=1+a_{n}-b_{n}+O_{p}(\Delta)=1+O_{p}(\Delta^{\frac{1}{2}})$.\\[0.5\baselineskip]
\noindent On one hand, for the former equation in (\ref{error v}), we have
\begin{align}\nonumber
    &\log\frac{[\lambda]_{\tau+\Delta}-[\lambda]_{\tau}}{\sigma^{2}\Delta}=\log\lambda_{\tau}^{2k}+\log\Big(O_{p}(\Delta^{\frac{1}{2}})\Big)=\log\lambda_{\tau}^{2k}+\log\Big(1+\frac{O_{p}(\Delta^{\frac{1}{2}})}{\lambda_{\tau}^{2k}}\Big)=2k\log\lambda_{\tau}+O_{p}(\Delta^{\frac{1}{2}}).\\
    \Longrightarrow&\quad\hat{k}_{0,n}^{(1)}:=\frac{\log\frac{[\lambda]_{i\Delta}-[\lambda]_{(i-1)\Delta}}{\sigma^{2}\Delta}}{2\log\lambda_{(i-1)\Delta}}=\frac{2k_{0}\log\lambda_{(i-1)\Delta}+O_{p}(\Delta^{\frac{1}{2}})}{2\log\lambda_{(i-1)\Delta}}=k_{0}+O_{p}(\Delta^{\frac{1}{2}}).\nonumber
\end{align}
On the other hand, for the latter equation in (\ref{error v}), we have
\begin{align}\nonumber
    &\log\frac{[\lambda]_{\tau+\Delta}-[\lambda]_{\tau}}{[\lambda]_{\tau'+\Delta}-[\lambda]_{\tau'}}=\log\frac{\lambda_{\tau}^{2k}}{\lambda_{\tau'}^{2k}}+\log\Big(1+O_{p}(\Delta^{\frac{1}{2}})\Big)=2k\log\frac{\lambda_{\tau}}{\lambda_{\tau'}}+O_{p}(\Delta^{\frac{1}{2}}).\\
    \Longrightarrow&\quad\hat{k}_{0,n}^{(2)}:=\frac{\log\frac{[\lambda]_{i\Delta}-[\lambda]_{(i-1)\Delta}}{[\lambda]_{j\Delta}-[\lambda]_{(j-1)\Delta}}}{2\log\frac{\lambda_{(i-1)\Delta}}{\lambda_{(j-1)\Delta}}}=\frac{2k_{0}\log\frac{\lambda_{(i-1)\Delta}}{\lambda_{(j-1)\Delta}}+O_{p}(\Delta^{\frac{1}{2}})}{2\log\frac{\lambda_{(i-1)\Delta}}{\lambda_{(j-1)\Delta}}}=k_{0}+O_{p}(\Delta^{\frac{1}{2}}).\nonumber
\end{align}
\textit{Step~1 (Uniform $1/2$-H\"{o}lder continuity in $\mathcal{L}^{q}$ and probability of the process $\lambda_{t}$)}\\
For $\tau\in [0,T]$ and $s\in[\tau,\tau+\Delta]\subset [0,T]$, we have
\begin{equation}\nonumber
    \lambda_{s}-\lambda_{\tau}=\int_{\tau}^{s}(a-b\lambda_{u})du+\int_{\tau}^{s}\sigma\lambda_{u}^{k}dW_{u}.
\end{equation}
Denote by $U_{\tau,s}:=\int_{\tau}^{s}(a-b\lambda_{u})du$ and $V_{\tau,s}:=\int_{\tau}^{s}\sigma\lambda_{u}^{k}dW_{u}$ (a continuous square-integrable martingale with $V_{\tau,s}\lvert_{\tau=s}=0$). Set $S:=\sup_{u\in[0,T]}\lambda_{u}$. By Lemma~\ref{ergodicity-of-CKLS}'s (\romannumeral1), $S<\infty$. With
\begin{equation}\label{inequalities}
(\lvert a\lvert+\lvert b\lvert S)^{q}\leq
\begin{cases}
	\lvert a\lvert^{q}+\lvert b\lvert^{q}S^{q}, & 0<q<1,\\[2mm]
	2^{q-1}\big(\lvert a\lvert^{q}+\lvert b\lvert^{q}S^{q}\big), & q\geq 1,
\end{cases}\quad\Longrightarrow\quad(\lvert a\lvert+\lvert b\lvert S)^{q}\lesssim \lvert a\lvert^{q}+\lvert b\lvert^{q}S^{q}.
\end{equation}
hence there exists $K_{1}\in(0,+\infty)$ such that $\mathbb{E}[(\lvert a\lvert+\lvert b\lvert S)^{q}]\leq K_{1}(1+\mathbb{E}[S^{q}])$. Since $S^{q}=(\sup_{u\in[0,T]}\lambda_{u})^{q}= \sup_{u\in[0,T]}\lambda_{u}^{q}$, by Lemma~\ref{ergodicity-of-CKLS}'s (\romannumeral3), $\mathbb{E}[S^{q}]<\infty$. Therefore,
\begin{equation}\nonumber
    \lvert U_{\tau,s}\lvert\leq\big(\lvert a\lvert+\lvert b\lvert \sup_{u\in[\tau,s]}\lambda_{u}\big)\lvert s-\tau\lvert \quad\Longrightarrow\quad \mathbb{E}\big[\lvert U_{\tau,s}\lvert^{q}\big]\leq K_{1}\lvert s-\tau\lvert^{q}\quad\Longrightarrow\quad\lvert U_{\tau,s}\lvert=O_{p}(\lvert s-\tau\lvert).
\end{equation}
For a continuous local martingale $M_{t}$ with $M_{0}=0$ and any $q>0$,
\begin{equation}\label{bdg}
	\text{(BDG's inequality)}~~~\mathbb{E}\big[\lvert M_{T}\lvert^{q}\big]\leq\mathbb{E}\big[(\sup_{0\leq t\leq T}\lvert M_{t}\lvert)^{q}\big]\leq C_{q}\mathbb{E}\big[[M]_{T}^{q/2}\big],
\end{equation}
since $\lvert M_{T}\lvert\leq \sup_{0\leq t\leq T}\lvert M_{t}\lvert$. When $M_{t}$ takes the form $\int_{0}^{t}H_{u}dW_{u}$ for some predictable $H_{t}$ with $\int_{0}^{\infty}H_{s}^{2}ds<\infty~a.s.$, its quadratic variation $[M]_{t}$ is $\int_{0}^{t}H_{u}^{2}du$ by It\^{o}'s isometry. Applying (\ref{bdg}) to the increments $\widetilde{V}_{t}:=\int_{\tau}^{\tau+t}\sigma\lambda_{u}^{k}dW_{u}$ yields
\begin{equation}\label{upperbound}
	\mathbb{E}\big[\lvert V_{\tau,s}\lvert^{q}\big]=\mathbb{E}\big[\lvert\widetilde{V}_{s-\tau}\lvert^{q}\big]\leq C_{q}\mathbb{E}\big[[V]_{s}-[V]_{\tau}\big]^{q/2}\quad\Longrightarrow\quad\mathbb{E}\big[\lvert V_{\tau,s}\lvert^{q}\big] \leq C_{q}\mathbb{E}\Big[\Big(\int_{\tau}^{s}\sigma^{2} \lambda_{u}^{2k}du\Big)^{q/2}\Big].
\end{equation}
\textit{Case I:} For $q\geq 2$ ($r:=q/2\geq 1$), H\"{o}lder's inequality with $(r,\frac{r}{r-1})$ yields
\begin{equation}\nonumber
	\Big(\int_{\tau}^{s}\sigma^{2}\lambda_{u}^{2k}du\Big)^{q/2}\leq \lvert s-\tau\lvert^{q/2-1}\int_{\tau}^{s}\sigma^{q}\lambda_{u}^{kq}du .
\end{equation}
Taking expectations and applying Tonelli's theorem (since the integrand $\sigma^{2}\lambda_{u}^{2k}$ is nonnegative),
\begin{equation}\nonumber
    \mathbb{E}\Big[\Big(\int_{\tau}^{s}\sigma^{2}\lambda_{u}^{2k}du\Big)^{q/2}\Big]\leq \sigma^{q}\lvert s-\tau\lvert^{q/2-1}\int_{\tau}^{s}\mathbb{E}[\lambda_{u}^{kq}]du\leq \sigma^{q}\sup_{0\leq u\leq T}\mathbb{E}[\lambda_{u}^{kq}]\lvert s-\tau\lvert^{q/2}.
\end{equation}
Set a positive $K_{2}:=\sigma^{q}\sup_{0\leq u\leq T}\mathbb{E}[\lambda_{u}^{kq}]<\infty$ (by Lemma~\ref{ergodicity-of-CKLS}'s (\romannumeral3)), then (\ref{upperbound}) becomes
\begin{equation}\nonumber
	\mathbb{E}\big[\lvert V_{\tau,s}\lvert^{q}\big]\leq C_{q}\mathbb{E}\Big[\Big(\int_{\tau}^{s}\sigma^{2}\lambda_{u}^{2k}du\Big)^{q/2}\Big]\leq K_{2}\lvert s-\tau\lvert^{q/2}.
\end{equation}
\textit{Case II:} For $0<q\leq 2$, set a positive $K_{2}':=\sigma^{2}\sup_{0\leq u\leq T}\mathbb{E}[\lambda_{u}^{2k}]<\infty$ (by Lemma~\ref{ergodicity-of-CKLS}'s (\romannumeral3)). It\^{o} isometry and Lyapunov's monotonicity of $\mathcal{L}_{q}$-norms (passing from order $2$ to $q$) yields
\begin{equation}\nonumber
	\mathbb{E}\big[\lvert V_{\tau,s}\lvert^{2}\big]=\sigma^{2}\int_{\tau}^{s}\mathbb{E}\big[\lambda_{u}^{2k}\big]du\leq K_{2}'\lvert s-\tau\lvert\quad\Longrightarrow\quad\mathbb{E}\big[\lvert V_{\tau,s}\lvert^{q}\big]\leq \big(\mathbb{E}\big[\lvert V_{\tau,s}\lvert^{2}\big]\big)^{q/2}\leq K_{2}'\lvert s-\tau\lvert^{q/2}.
\end{equation}
Combining \textit{Case I} and \textit{Case II}, for any $q>0$, let $\mathbb{R}_{+}\ni K^{*}:=\max\{K_{2},K_{2}'\}$, then
\begin{equation}\nonumber
    \mathbb{E}\big[\lvert V_{\tau,s}\lvert^{q}\big]\leq K^{*} \lvert s-\tau\lvert^{q/2}\qquad\Longrightarrow\qquad
    \lvert V_{\tau,s}\lvert=O_{p}(\lvert s-\tau\lvert^{\frac{1}{2}}).
\end{equation}
Since $V_{\tau,s}$ is dominant over $U_{\tau,s}$ when $\lvert s-\tau\lvert\to0$, using inequalities (\ref{inequalities}) again, we conclude: For any $q>0$, there exists $C:=\max\{C_{1},C_{2}\}$ independent of $\tau,s$ such that ($0<q\leq 1$ and $q>1$, respectively)
\begin{equation}\label{uniform-rate}
\begin{aligned}
    &\mathbb{E}\big[\lvert \lambda_{s}-\lambda_{\tau}\lvert^{q}\big]\leq 
\begin{cases}
\mathbb{E}\big[\lvert U_{\tau,s}\lvert^{q}\big]+\mathbb{E}\big[\lvert V_{\tau,s}\lvert^{q}\big]\leq C_{1}\lvert s-\tau\lvert^{q/2} & 0<q\leq 1,\\
2^{q-1}\big(\mathbb{E}\big[\lvert U_{\tau,s}\lvert^{q}\big]+\mathbb{E}\big[\lvert V_{\tau,s}\lvert^{q}\big]\big)\leq C_{2}\lvert s-\tau\lvert^{q/2}, & q>1.
\end{cases}\\
    \Longrightarrow&\quad
    \lvert\lambda_{s}-\lambda_{\tau}\lvert=O_{p}(\lvert s-\tau\lvert^{\frac{1}{2}}).
\end{aligned}
\end{equation}
The above estimate holds for any $\tau\in[0,T]$ and $s\in[\tau,\tau+\Delta]\subset[0,T]$, with a constant $C$ independent of $\tau$ and $s$. Hence, for any $\tau, s\in[0,T]$, $\lambda_{t}$ is uniform $1/2$-H\"{o}lder continuous in $\mathcal{L}^{q}$ and also in probability.\\[0.5\baselineskip]
\noindent\textit{Step~2 (Uniform $1/2$-H\"{o}lder continuity in $\mathcal{L}^{1}$ and probability of the process $\lambda_{t}^{2k}$)}\\
Let $f(x)=x^{2k}$. For $y>x> 0$, with $\xi\in(x,y)$, mean-value theorem yields
\begin{equation}\nonumber
	\lvert f(x)-f(y)\lvert=\lvert f'(\xi)\lvert \lvert x-y\lvert=2k\xi^{2k-1}\lvert x-y\lvert \leq 2k(x^{2k-1}+y^{2k-1})\lvert x-y\lvert.
\end{equation}
With $x=\lambda_{\tau}$, $y=\lambda_{s}$ and the inequalities Cauchy--Schwarz and $(x+y)^{2}\leq 2(x^{2}+y^{2})$:
\begin{align}\nonumber
	&\mathbb{E}\big[\lvert\lambda_{s}^{2k}-\lambda_{\tau}^{2k}\lvert\big] \leq 2k\mathbb{E}\big[(\lambda_{s}^{2k-1}+\lambda_{\tau}^{2k-1})\lvert\lambda_{s}-\lambda_{\tau}\lvert\big] \\
    \leq & 2k\Big(\mathbb{E}\big[(\lambda_{s}^{2k-1}+\lambda_{\tau}^{2k-1})^{2}\big]\Big)^{\frac{1}{2}}\Big(\mathbb{E}\big[\lvert\lambda_{s}-\lambda_{\tau}\lvert^{2}\big]\Big)^{\frac{1}{2}}\leq 2^{\frac{3}{2}}k\Big(\mathbb{E}\big[\lambda_{s}^{4k-2}+\lambda_{\tau}^{4k-2}\big]\Big)^{\frac{1}{2}}\Big(\mathbb{E}\big[\lvert\lambda_{s}-\lambda_{\tau}\lvert^{2}\big]\Big)^{\frac{1}{2}}\label{power-process}.
\end{align}
As $k\in[\frac{1}{2},1)\Rightarrow 4k-2\in[0,1)$, by Lemma~\ref{ergodicity-of-CKLS}'s (\romannumeral3):
\begin{equation}\nonumber
	\sup_{u\in[0,T]}\mathbb{E}[\lambda_{u}^{4k-2}]<\infty\qquad\Longrightarrow\qquad\sup_{\tau,s\in[0,T]}\mathbb{E}\big[\lambda_{s}^{4k-2}+\lambda_{\tau}^{4k-2}\big]<\infty.\tag{$\ast$}
\end{equation}
By (\ref{uniform-rate}) we also have the quadratic increment bound
\begin{equation}\nonumber
	\quad\mathbb{E}\big[\lvert\lambda_{s}-\lambda_{\tau}\lvert^{2}\big]\leq C \lvert s-\tau\lvert\quad\text{for all }\tau,s\in[0,T].\tag{$\circ$}
\end{equation}
Combining $(\ast)$ and $(\circ)$ for (\ref{power-process}) yields uniform $1/2$-H\"{o}lder continuity in $\mathcal{L}^{1}$ and also in probability
\begin{equation}\nonumber
	\mathbb{E}\big[\lvert\lambda_{s}^{2k}-\lambda_{\tau}^{2k}\lvert\big]\leq C\lvert s-\tau\lvert^{\frac{1}{2}},\text{ for any }\tau,s\in[0,T]\quad \Rightarrow\quad\lvert\lambda_{s}^{2k}-\lambda_\tau^{2k}\lvert=O_{p}(\lvert s-\tau\lvert ^{\frac{1}{2}}),\text{ for any }\tau,s\in[0,T].
\end{equation}
\textit{Step~$\ast$ (Regularity simplification for CIR case)}
When $k=\frac{1}{2}$, Lemma~\ref{ergodicity-of-CIR}'s (\romannumeral6) shows that $\lambda_{t}$ is uniform $1/2$-H\"{o}lder continuity in $\mathcal{L}^{q}$ on $[0,T]$, so \textit{Step~1} follows immediately. Now that $2k=1$ and $x\mapsto x^{2k}=x$ is globally Lipschitz, \textit{Step~2} is trivial: $\lvert\lambda_{s}-\lambda_{\tau}\lvert=O_{p}(\lvert s-\tau\lvert^{\frac{1}{2}})$.
\end{proof}

\begin{remark}[Upgrade to uniform almost sure convergence rate]
The uniform $1/2$-H\"{o}lder continuity in $\mathcal{L}^q$ for some $q \geq 1$ does not imply the uniform almost sure estimate $\lvert \lambda_{s}-\lambda_{\tau}\lvert^{q}=O_{a.s.}(\lvert s-\tau\lvert^{q/2})$. This failure is due to the requirements of the Kolmogorov-Chentsov continuity theorem, which necessitates an exponent $1+\eta>q/2$ (i.e., $\eta>0$) in the moment condition to guarantee almost sure $\delta$-H\"{o}lder paths. When the exponent is exactly $q/2$, the theorem only guarantees $\delta$-H\"{o}lder continuity for $\delta<\frac{1}{2}$ (specifically, $\delta<\frac{1}{2}-\frac{1}{q}$ when $q>2$). Consequently, the ratio $\frac{\lvert\lambda_{s}-\lambda_{\tau}\lvert}{\lvert s-\tau\lvert^{1/2}}$ remains almost surely unbounded as $\lvert s-\tau\lvert\to 0$.
\end{remark}

\section{Asymptotic Negligibility of the Plug-in Error}
\subsection{Plug-in Version Estimator of \texorpdfstring{$\beta$}{beta} and Asymptotic Error Decomposition}
Using either of $\hat{k}_{0,n}^{(\ell)}$, $\ell\in\{1,2\}$ given in Theorem~\ref{initial-etimators}, together with the explicit form of the mapping \eqref{eq:T_solution}, the observed CKLS data, $i=0,\dots,n$, can be mapped into 
the transformed/pseudo CIR data via the following formula
\[
    X_{i\Delta}^{(\hat{k}_{0,n})}:=\mathcal{T}\big(\lambda_{i\Delta}\big\lvert\hat{k}_{0,n}\big)=\frac{L^{2}}{4(1-\hat{k}_{0,n})^{2}}\lambda_{i\Delta}^{2-2\hat{k}_{0,n}}.
\]
Denote the estimator of $\beta$ defined in \eqref{beta-form} computed from the transformed/pseudo data $\{X_{i\Delta}^{(\hat{k}_{0,n})}\}_{i=0}^{n}$ by  $\hat{\beta}_{n}^{(\hat{k}_{0,n})}$, and note that $\beta_{0}^{(k_{0})}\equiv\beta_{0}$. Assume a known $\sigma$, substituting $\gamma$ with $\sigma L$ and the explicit form of $\mathcal{T}(\cdot\lvert\hat{k}_{0,n})$ yields
\begin{align}\nonumber
    \hat{\beta}_{n}^{(\hat{k}_{0,n})}:=&\frac{\frac{\sigma^{2}L^{2}}{2}\cdot n\sum_{i=1}^{n}\frac{L^{2}}{4(1-\hat{k}_{0,n})^{2}}\lambda_{(i-1)\Delta}^{2-2\hat{k}_{0,n}}}{n\sum_{i=1}^{n}\frac{L^{4}}{16(1-\hat{k}_{0,n})^{4}}\lambda_{(i-1)\Delta}^{4-4\hat{k}_{0,n}}-\Big(\sum_{i=1}^{n}\frac{L^{2}}{4(1-\hat{k}_{0,n})^{2}}\lambda_{(i-1)\Delta}^{2-2\hat{k}_{0,n}}\Big)^{2}}\\
    =&\frac{2\sigma^{2}(1-\hat{k}_{0,n})^{2}\cdot n\sum_{i=1}^{n}\lambda_{(i-1)\Delta}^{2-2\hat{k}_{0,n}}}{n\sum_{i=1}^{n}\lambda_{(i-1)\Delta}^{4-4\hat{k}_{0,n}}-\Big(\sum_{i=1}^{n}\lambda_{(i-1)\Delta}^{2-2\hat{k}_{0,n}}\Big)^{2}}.\label{final-estimation}
\end{align}
Next, given $\sqrt{n\Delta}(\hat{\beta}_{n}^{(k_{0})}-\beta_{0}^{(k_{0})})\xrightarrow{d_{st}}\mathcal{N}\big(0,\frac{2\beta_{0}}{\alpha}(\alpha+\gamma^{2})\big)$ as $n\to\infty$, to determine if the stable CLT
\begin{equation}\nonumber
    \sqrt{n\Delta}(\hat{\beta}_{n}^{(\hat{k}_{0,n})}-\beta_{0}^{(k_{0})})\xrightarrow{d_{st}}\mathcal{N}\Big(0,\frac{2\beta_{0}}{\alpha}(\alpha+\gamma^{2})\Big)\qquad\text{as }n\to\infty,
\end{equation}
continues to hold when $k_{0}$ is replaced by the estimated value $\hat{k}_{0,n}$, where $\hat{\beta}_{n}^{(k_{0})}$ coincides with the estimator of $\beta$ defined in \eqref{beta-form}, is the crux. Observe the decomposition
\begin{equation}\nonumber
    \sqrt{n\Delta}(\hat{\beta}_{n}^{(\hat{k}_{0,n})}-\beta_{0}^{(k_{0})})=\underbrace{\sqrt{n\Delta}(\hat{\beta}_{n}^{(\hat{k}_{0,n})}-\hat{\beta}_{n}^{(k_{0})})}_{\text{plug-in error}}+\underbrace{\sqrt{n\Delta}(\hat{\beta}_{n}^{(k_{0})}-\beta_{0}^{(k_{0})})}_{\text{stable convergence in law}}.
\end{equation}
Thus, it remains only to show that the plug-in error is asymptotically negligible.

\subsection{First-order Taylor's Expansion with Lagrange Remainder}
With $X_{(i-1)\Delta}^{(k)}=\frac{L^{2}}{4(1-k)^{2}}\lambda_{(i-1)\Delta}^{2-2k}$, regard the estimator $\hat{\beta}_{n}^{(k)}$ as a function with respect to elasticity:
\begin{align}\nonumber
    &f_{n}(k):=\hat{\beta}_{n}^{(k)}=\frac{\gamma^{2}}{2}\frac{n\sum_{i=1}^{n}X_{(i-1)\Delta}}{n\sum_{i=1}^{n}X_{(i-1)\Delta}^{2}-(\sum_{i=1}^{n}X_{(i-1)\Delta})^{2}}=:\frac{\gamma^{2}}{2}\frac{A_{n}(k)}{B_{n}(k)-A_{n}^{2}(k)},\\
    &\text{with}~A_{n}(k)=\frac{L^{2}}{4(1-k)^{2}}\frac{1}{n}\sum_{i=1}^{n}\lambda_{(i-1)\Delta}^{2-2k}\text{ and }
    B_{n}(k)=\frac{L^{4}}{16(1-k)^{4}}\frac{1}{n}\sum_{i=1}^{n}\lambda_{(i-1)\Delta}^{4-4k}.\label{AB}
\end{align}
\noindent Applying First-order Taylor's expansion with Lagrange remainder to $f_{n}(k)\equiv\hat{\beta}_{n}^{(k)}$ at $k_{0}$ yields
\begin{equation}\nonumber
	\sqrt{n\Delta}(\hat{\beta}_{n}^{(k)}-\hat{\beta}_{n}^{(k_{0})})\equiv\sqrt{n\Delta}\big(f_{n}(k)-f_{n}(k_{0})\big)=\sqrt{n\Delta}f'_{n}(k_{0})(k-k_{0})+\sqrt{n\Delta}\frac{1}{2}f''_{n}(\tilde{k})(k-k_{0})^{2},
\end{equation}
with $\tilde{k}\in(k_{0},k)$. Imposing $k=\hat{k}_{0,n}$ yields the following expression:
\begin{equation}\nonumber
	\sqrt{n\Delta}(\hat{\beta}_{n}^{(\hat{k}_{0,n})}-\hat{\beta}_{n}^{(k_{0})})=\sqrt{n\Delta}f'_{n}(k_{0})(\hat{k}_{0,n}-k_{0})+\sqrt{n\Delta}\frac{1}{2}f''_{n}(\tilde{k})(\hat{k}_{0,n}-k_{0})^{2}.
\end{equation}
Recall that $\hat{k}_{0,n}\xrightarrow{p}k_{0}$ as $n\to\infty$ and that the relevant error terms are $O_{p}(\Delta^{\frac{1}{2}})$. In addition, under the condition $n\Delta^{2}\to0$, we have a sharper rate $\hat{k}_{0,n}-k_{0}=o_{p}\big((n\Delta)^{-\frac{1}{2}}\big)$. If we can prove both $f'_{n}(k)=\frac{\partial}{\partial k}f_{n}(k)$ and $f''_{n}(k)=\frac{\partial^{2}}{\partial k^{2}}f_{n}(k)$ are $O_{p}(1)$ (as $n\to\infty$), then
\begin{align}\nonumber
    \sqrt{n\Delta}f'_{n}(k_{0})(\hat{k}_{0,n}-k_{0})&=\sqrt{n\Delta}O_{p}(1)o_{p}\big((\sqrt{n\Delta})^{-\frac{1}{2}}\big)=o_{p}(1),\\
    \sqrt{n\Delta}\frac{1}{2}f''_{n}(\tilde{k})(\hat{k}_{0,n}-k_{0})^{2}&=\sqrt{n\Delta}O_{p}(1)o_{p}\big((n\Delta)^{-1}\big)=o_{p}\big((n\Delta)^{-\frac{1}{2}}\big)=o_{p}(1)\nonumber.
\end{align}
Consequently, as $n\to\infty$:
\begin{equation}\nonumber
    \sqrt{n\Delta}(\hat{\beta}_{n}^{(\hat{k}_{0,n})}-\hat{\beta}_{n}^{(k_{0})})=o_{p}(1).
\end{equation}

\subsubsection{\texorpdfstring{$\hat{\beta}_{n}$}{betan} as a Function with respect to Elasticity}
\noindent By the following lemma, we notice that $A_{n}(k)=O_{p}(1)$, $B_{n}(k)=O_{p}(1)$ and $C_{n}(k):=B_{n}(k)-A_{n}^{2}(k)=O_{p}(1)$, as $n\to\infty$. Consequently, $f_{n}(k)=\frac{O_{p}(1)}{O_{p}(1)}=O_{p}(1)$.
\begin{lemma}
Assume that the initial value $\lambda_{0}$ has the distribution $\pi_{0}$ (i.e., $\lambda_{0}\sim\pi_{0}$). For any $q>0$,
\begin{equation}\label{RiemannSum-RiemannIntegral}
	\frac{1}{n}\sum_{i=1}^{n}\lambda_{(i-1)\Delta}^{q}\xrightarrow[\Delta\to0]{\mathcal{L}^{1}}\mathbb{E}[\lambda_{t}^{q}].
\end{equation}
\end{lemma}
\begin{proof} 
Define the discretization error $\mathsf{R}_{n}$, and by Tonelli's theorem (nonnegative integrand after taking absolute values), for $s\in[(i-1)\Delta,i\Delta]$:
\begin{equation}\label{Tonelli}
	\mathsf{R}_{n}:=\frac{1}{T}\sum_{i=1}^{n}\lambda_{(i-1)\Delta}^{q}\Delta-\frac{1}{T}\int_{0}^{T}\lambda_{s}^{q}ds\quad\Longrightarrow\quad
	\mathbb{E}\big[\lvert\mathsf{R}_{n}\lvert\big]\leq\frac{1}{T}\sum_{i=1}^{n}\int_{(i-1)\Delta}^{i\Delta}\mathbb{E}\big[\lvert\lambda_{(i-1)\Delta}^{q}-\lambda_{s}^{q}\lvert\big]ds.
\end{equation}
\textit{Case 1:} By \eqref{inequalities} for $0<q<1$ and the pathwise $1/2$-H\"{o}lder bound \eqref{uniform-rate}:
\begin{equation}\nonumber
	\lvert \lambda_{(i-1)\Delta}^{q}-\lambda_{s}^{q}\lvert \leq \lvert \lambda_{(i-1)\Delta}-\lambda_{s}\lvert ^{q}
    ~~\Rightarrow~~
    \mathbb{E}\big[\lvert\lambda_{(i-1)\Delta}^{q}-\lambda_{s}^{q}\lvert\big]\leq\mathbb{E}\big[\lvert\lambda_{(i-1)\Delta}-\lambda_{s}\lvert^{q}\big]\leq C_{q} \lvert (i-1)\Delta-s\lvert^{q/2} \leq C_{q}\Delta^{q/2}.
\end{equation}
\textit{Case 2:} The proof for a general exponent $q\geq1$ is identical to the argument in \textit{Step~2} in the proof of Theorem~\ref{k-rate} for the case $f(x)=x^{2k}$. Applying the mean-value theorem to $x\mapsto x^{q}$, together with the moment bounds in Lemma~\ref{ergodicity-of-CKLS}'s (\romannumeral3) and the quadratic increment estimate \eqref{uniform-rate}, yields, 
\[
    \mathbb{E}\big[\lvert\lambda_{(i-1)\Delta)}^{q}-\lambda_{s}^{q}\lvert\big]\leq C_{q}'\lvert (i-1)\Delta-s\lvert^{\frac{1}{2}}\leq C_{q}'\Delta^{\frac{1}{2}}.
\]
Insert the two cases into \eqref{Tonelli} yields
\begin{equation}\nonumber
	\mathbb{E}\big[\lvert\mathsf{R}_{n}\lvert\big]\leq
    \begin{cases}
    \frac{1}{T}\sum_{i=1}^{n}\int_{(i-1)\Delta}^{i\Delta} C_{q}\Delta^{\frac{q}{2}}ds=C_{q}\Delta^{\frac{q}{2}}\xrightarrow[\Delta\to0]{}0,&0<q<1;\\
    \frac{1}{T}\sum_{i=1}^{n}\int_{(i-1)\Delta}^{i\Delta}C'_{q}\Delta^{\frac{1}{2}}ds=C'_{q}\Delta^{\frac{1}{2}}\xrightarrow[\Delta\to0]{}0,&q\geq 1.
    \end{cases}
    ~\Longrightarrow~\mathbb{E}\big[\lvert\mathsf{R}_{n}\lvert\big]\lesssim\Delta^{\frac{q}{2}\wedge \frac{1}{2}},
\end{equation}
whence $\mathsf{R}_{n}\to 0$ as $T\to\infty$ in $\mathcal{L}^{1}$ and hence in probability.\\[0.5\baselineskip]
\noindent Next, since $\lambda_{t}$ is positive Harris recurrent with the unique invariant law $\pi_{0}(dx)=p_{\infty}(x)dx$, the ergodic theorem yields, for any $q>0$ with $0\leq x\mapsto x^{q}\in \mathcal{L}^{1}(\pi_{0})$ (Lemma~\ref{ergodicity-of-CKLS}'s (\romannumeral4)):
\begin{equation}\nonumber
	\frac{1}{T}\int_{0}^{T} \lambda_{t}^{q}dt\xrightarrow[T\to\infty]{a.s./\mathcal{L}^{1}}\int_{0}^{\infty}x^{q}p_{\infty}(x)dx\overset{*}=\mathbb{E}[\lambda_{t}^{q}].
\end{equation}
$*$ further requires $\lambda_{0}\sim\pi_{0}$, whereas $a.s/\mathcal{L}^{1}$ convergence hold for any initial distribution.
\end{proof}

\noindent We can extend this lemma to a theorem covering a broader class of cases, which will be useful later in this section. Hereafter—except when discussing selected sufficient conditions that allow upgrading convergence to $\mathcal{L}^{q}$ or almost-sure convergence—we will not dwell on these stronger modes of convergence and will focus mainly on verifying convergence in probability.

\begin{theorem}\label{theorem-Riemann}[Riemann-sum approximation for positive Harris recurrent diffusions]\\
Let $\{Z_{t}\}_{t\geq 0}$ be a strictly positive, positive Harris recurrent diffusion on $(0,+\infty)$ with invariant distribution $\pi_{0}$. In the high-frequency setting, define $\mathcal{S}_{n}:=\frac{1}{n}\sum_{i=1}^{n}g(Z_{(i-1)\Delta})$. Assume:
\begin{itemize}
	\item[(A1)] (uniform H\"{o}lder condition) $\exists\delta\in(0,1]$ such that $\sup_{\lvert s-\tau\lvert\leq\Delta}\lvert Z_{s}-Z_{\tau}\lvert=O_{p}(\Delta^{\delta})$, $\forall \tau,s\in[0,T]$.
	\item[(A2)] (integrability and regularity of $g$) $g\in\mathcal{L}^{1}(\pi_{0})$, is locally Lipschitz on $(0,+\infty)$, and has polynomial growth: $\lvert g(x)\lvert\lesssim 1+x^{r}$ for some $r>0$.
\end{itemize}
Then the discretization error $\mathcal{R}_{n}:=\mathcal{S}_{n}-\frac{1}{T}\int_{0}^{T}g(Z_{t})dt$ satisfies
\begin{equation}\nonumber
	\mathcal{R}_{n}=O_{p}(\Delta^{\delta})\quad\text{and}\quad\mathcal{S}_{n}\xrightarrow{p}\int gd\pi_{0}<\infty,~\text{as }n\to\infty\quad(\text{i.e., }\mathcal{S}_{n}=O_{p}(1)).
\end{equation}
If moreover $\big\{\lvert g(Z_{i\Delta})\lvert^{q}\big\}_{i=0}^{n}$ is uniformly integrable (e.g., by the polynomial growth of $g$ and bounded $(qr+\varepsilon)$-moments of $Z_{t}$), then $\mathbb{E}\big[\lvert \mathcal{R}_{n}\lvert^{q}\big]\lesssim \Delta^{q\delta}~\text{and}~\mathcal{S}_{n}\xrightarrow{\mathcal{L}^{q}}\int gd\pi_{0}<\infty$ as $n\to\infty$; If moreover (A1) holds almost surely, i.e., $\sup_{\lvert s-\tau\lvert \leq \Delta}\lvert Z_{s}-Z_{\tau}\lvert=O_{a.s.}(\Delta^{\delta})$ for any $\tau,s\in[0,T]$, then $\mathcal{R}_{n}=O_{a.s.}(\Delta^{\delta})~\text{and}~\mathcal{S}_{n}\xrightarrow{a.s.}\int gd\pi_{0}<\infty$ as $n\to\infty$.
\end{theorem}

\begin{remark}[Borel--Cantelli criterion for almost sure convergence]
The almost sure convergence of $\mathcal{S}_{n}$ also holds if the discretization error satisfies a summable tail condition, i.e., $\sum_{n=1}^{\infty}\mathbb{P}\big(\big\lvert\mathcal{S}_{n}-\frac{1}{T}\int_{0}^{T}g(Z_{t})dt\big\lvert>\varepsilon\big)<\infty,~\forall\varepsilon>0$, by the Borel--Cantelli lemma. This serves as an alternative to the uniform H\"{o}lder condition in probability.
\end{remark}

\begin{proof}
Write
\begin{equation}\nonumber
	\mathcal{S}_{n}=\frac{1}{T}\sum_{i=1}^{n}g\big(Z_{(i-1)\Delta}\big)\Delta=\frac{1}{T}\int_{0}^{T}g(Z_{t})dt+\mathcal{R}_{n},\quad\text{with }\mathcal{R}_{n}:=\frac{1}{T}\sum_{i=1}^{n}\int_{(i-1)\Delta}^{i\Delta}\big(g(Z_{(i-1)\Delta})-g(Z_{s})\big)ds.
\end{equation}
\noindent \textit{Step~1 (Discretization error under localization)} Fix $K\geq 1$, set $\tau_{K}:=\{\inf\{t\in[0,T]:Z_{t}\notin[K^{-1},K]\}\}\wedge T$ and let $Z_{t}^{(K)}:=Z_{t\wedge\tau_{K}}$. On $\{\tau_{K}=T\}$, the process stays in $[K^{-1},K]$, where the function $g$ is Lipschitz with some constant $L_{K}$ by (A2), hence $\lvert \mathcal{R}_{n}\lvert\leq\frac{L_{K}}{T}\sum_{i=1}^{n}\int_{(i-1)\Delta}^{i\Delta}\big\lvert Z^{(K)}_{(i-1)\Delta}-Z^{(K)}_{s}\big\lvert ds\leq L_{K}\sup_{\lvert s-\tau\lvert\leq\Delta}\big\lvert Z^{(K)}_{s}-Z^{(K)}_{\tau}\big\lvert$.
By (A1) applied to $Z$ (and hence to $Z^{(K)}$), $\sup_{\lvert s-\tau\lvert\leq\Delta}\big\lvert Z^{(K)}_{s}-Z^{(K)}_{\tau}\big\lvert=O_{p}(\Delta^{\delta})$, so $\mathcal{R}_{n}=O_{p}(\Delta^{\delta})$ on $\{\tau_{K}=T\}$. Since $Z_{t}$ is strictly positive and nonexplosive on finite horizons, we have $\mathbb{P}(\tau_{K}<T)\to0$ as $K\to\infty$, and a standard localization argument yields $\mathcal{R}_{n}=O_{p}(\Delta^{\delta})$. If in addition (A1) holds almost surely, the same argument shows $\mathcal{R}_{n}=O_{a.s.}(\Delta^{\delta})$.\\
\noindent \textit{Step~2 (Convergence of the time average)} By positive Harris recurrence of $Z_{t}$ and $g\in\mathcal{L}^{1}(\pi_{0})$, the continuous-time ergodicity for additive functionals yields $\frac{1}{T}\int_{0}^{T}g(Z_{t})dt\xrightarrow[T\to\infty]{a.s./\mathcal{L}^{1}}\int gd\pi_{0}<\infty$. Combining this with the bound in \textit{Step~1} yields $\mathcal{S}_{n}=\frac{1}{T}\int_{0}^{T}g(Z_{t})dt+\mathcal{R}_{n}\xrightarrow{p}\int gd\pi_{0}$ as $n\to\infty$, hence $\mathcal{S}_{n}=O_{p}(1)$. If $\mathcal{R}_{n}=O_{a.s.}(\Delta^\delta)$, the same reasoning yields $\mathcal{S}_{n}\xrightarrow{a.s.}\int gd\pi_{0}<\infty$ as $n\to\infty$.\\
\noindent \textit{Step~$\ast$ ($\mathcal{L}^{q}$-convergence):} If, for some $q>0$, the family $\{\lvert g(Z_{i\Delta})\lvert^{q}\}_{i=0}^{n}$ is uniformly integrable (for instance by the polynomial growth of $g$ and a uniform moment bound $\sup_{0\leq t\leq T}\mathbb{E}\big[\lvert Z_{t}\lvert^{qr+\varepsilon}\big]<\infty$ for some $\varepsilon>0$), the estimate in Step~1 improves to $\mathcal{R}_{n}=O_{\mathcal{L}^{q}}(\Delta^{\delta})$, and the ergodic theorem in $\mathcal{L}^{1}$ together with uniform integrability implies $\mathcal{S}_{n}\xrightarrow{\mathcal{L}^{q}}\int gd\pi_{0}<\infty$ as $n\to\infty$.
\end{proof}

\subsubsection{First-order Derivative of \texorpdfstring{$\hat{\beta}_{n}$}{betan} with respect to Elasticity}
Recall \eqref{AB}, the first-order derivatives $A'_{n}(k)=\frac{\partial}{\partial k}A_{n}(k)$ and $B'_{n}(k)=\frac{\partial}{\partial k}B_{n}(k)$ are
\begin{equation}
\begin{aligned}\nonumber
    A'_{n}(k)&\overset{(\ast)}{=}\frac{1}{n}\sum_{i=1}^{n}\frac{L^{2}}{4}\frac{\partial}{\partial k}\Big(\frac{\lambda_{(i-1)\Delta}^{2-2k}}{(1-k)^{2}}\Big)=\frac{1}{n}\sum_{i=1}^{n}\frac{L^{2}}{4}\frac{\lambda_{(i-1)\Delta}^{2-2k}}{(1-k)^{2}}\Big(\frac{2}{1-k}-2\log\lambda_{(i-1)\Delta}\Big)=:\frac{1}{n}\sum_{i=1}^{n}g_{1}(\lambda_{(i-1)\Delta}),\nonumber\\
    B'_{n}(k)&\overset{(\diamond)}{=}\frac{1}{n}\sum_{i=1}^{n}\frac{L^{4}}{16}\frac{\partial}{\partial k}\Big(\frac{\lambda_{(i-1)\Delta}^{4-4k}}{(1-k)^{4}}\Big)=\frac{1}{n}\sum_{i=1}^{n}\frac{L^{4}\lambda_{(i-1)\Delta}^{4-4k}}{16(1-k)^{4}}\Big(\frac{4}{1-k}-4\log\lambda_{(i-1)\Delta}\Big)=:\frac{1}{n}\sum_{i=1}^{n}g_{2}(\lambda_{(i-1)\Delta}),\nonumber
\end{aligned}
\end{equation}
where $g_{1}(x):=\frac{L^{2}}{2(1-k)^{3}}x^{2-2k}-\frac{L^{2}}{2(1-k)^{2}}x^{2-2k}\log x$ and $g_{2}(x):=\frac{L^{4}}{4(1-k)^{5}}x^{4-4k}-\frac{L^{4}}{4(1-k)^{4}}x^{4-4k}\log x$. Note $(\ast)$ and $(\diamond)$ (interchangeability of summation and partial derivative) are ensured by dominated convergence theorem (Lemma~\ref{CKLS-process}'s (\romannumeral3)).\\[0.5\baselineskip]
\noindent The goal is to take advantage of Theorem~\ref{theorem-Riemann} to obtain the convergences:
\begin{equation}\nonumber
    A'_{n}(k)=\frac{1}{n}\sum_{i=1}^{n}g_{1}(\lambda_{(i-1)\Delta})=O_{p}(1)\quad\text{ and }\quad
    B'_{n}(k)=\frac{1}{n}\sum_{i=1}^{n}g_{2}(\lambda_{(i-1)\Delta})=O_{p}(1),
\end{equation}
so that $C'_{n}(k)=B'_{n}(k)-2A'_{n}(k)A_{n}(k)=O_{p}(1)$ and
\begin{equation}\nonumber
    f'_{n}(k)=\frac{\gamma^{2}}{2}\frac{A'_{n}(k_{0})C_{n}(k_{0})-A_{n}(k_{0})C'_{n}(k_{0})}{C_{n}^{2}(k_{0})}=O_{p}(1)\frac{O_{p}(1)O_{p}(1)-O_{p}(1)O_{p}(1)}{O_{p}(1)O_{p}(1)}=O_{p}(1).
\end{equation}
\noindent  Let $h_{1}(x):=x^{2-2k}\log x$, $h_{2}(x):=x^{4-4k}\log x$. Recalling assumptions (A1)-(A2) in Theorem~\ref{theorem-Riemann}, note that the uniform $1/2$-H\"{o}lder continuity of $\lambda_{t}$ in probability has been established in Theorem~\ref{k-rate}. Since the building blocks $x\mapsto x^{q}$ and $x\mapsto \log x$ are $\mathcal{C}^{1}$ on $(0,+\infty)$, both $h_{1}$ and $h_{2}$ are $\mathcal{C}^{1}$ and hence locally Lipschitz on $(0,+\infty)$. Therefore, it remains only to verify that $h_{1},h_{2}\in\mathcal{L}^{1}(\pi_{0})$.\\[0.5\baselineskip]
\noindent For the $\mathcal{L}^{1}$-integrability of $h_{1}$, split the integral at $x=1$. For any $r,r'>0$, the elementary bounds $\log x\lesssim x^{r}$ for $x\geq 1$ and $\lvert\log x \lvert\lesssim x^{-r'}$ for $0<x\leq 1$ yield
\begin{equation}\nonumber
	\int_{1}^{\infty}\lvert h_{1}(x)\lvert \pi_{0}(dx)\lesssim \int_{1}^{\infty}x^{2-2k+r}p_{\infty}(x)dx\quad\text{and}\quad
	\int_{0}^{1}\lvert h_{1}(x)\lvert\pi_{0}(dx)\lesssim \int_{0}^{1}x^{2-2k-r'}p_{\infty}(x)dx.
\end{equation}
with $2-2k+r\in(r,r+1]$ and $2-2k-r'\in(-r',1-r']$ (when $\frac{1}{2}<k<1$) Lemma~\ref{ergodicity-of-CKLS}'s (\romannumeral4) implies finiteness of both integrals for any $r,r'>0$. In the case $k=\frac{1}{2}$, by Lemma~\ref{ergodicity-of-CIR}'s (\romannumeral3), with $p_\infty(x)\propto x^{(2a/\sigma^{2})-1}e^{-(2b/\sigma^{2})x}$, the same estimates give
\begin{equation}\nonumber
	\int_{1}^{\infty}x\log x p_{\infty}(x)dx\lesssim \int_{1}^{\infty}x^{r+\frac{2a}{\sigma^{2}}}e^{-\frac{2b}{\sigma^{2}}x}dx<\infty\quad\text{and}\quad
	~\int_{0}^{1}x\lvert\log x\lvert p_{\infty}(x)dx \lesssim \int_{0}^{1}x^{-r'+\frac{2a}{\sigma^{2}}}e^{-\frac{2b}{\sigma^{2}}x}dx<\infty.
\end{equation}
As both integrals can be written as (scaled) upper/lower incomplete Gamma functions, their respective finiteness are ensured for $r+\frac{2a}{\sigma^{2}}+1>0$ and $-r'+\frac{2a}{\sigma^{2}}+1>0$. Since $\frac{2a}{\sigma^{2}}+1>0$, one may take any $r>0$, $0<r'<\frac{2a}{\sigma^{2}}+1$. Consequently, $h_{1}\in\mathcal{L}^{1}(\pi_{0})$.

The $\mathcal{L}^{1}$-integrability of $h_{2}$ follows immediately from the above argument for $h_{1}$. Indeed, on $(0,1]$ we have $x^{4-4k}\leq 1$, so $\lvert h_{2}(x)\lvert \leq \lvert h_{1}(x)\lvert $ and the finiteness of $\int_{0}^{1}\lvert h_{2}\lvert d\pi_{0}$ follows from $h_{1}\in\mathcal{L}^{1}(\pi_{0})$. On $[1,\infty)$, using $\log x\lesssim x^{r}$ for any $r>0$ yields $\lvert h_{2}(x)\lvert=x^{4-4k}\lvert \log x\lvert\lesssim x^{4-4k+r}$, and both the CKLS stationary density (super-polynomial right tail) and the CIR density (exponential right tail) ensure $\int_{1}^{\infty}x^{4-4k+r}p_{\infty}(x)dx<\infty$ for all $r>0$. Consequently, $h_{2}\in\mathcal{L}^{1}(\pi_{0})$.

\subsubsection{Second-order Derivative of \texorpdfstring{$\hat{\beta}_{n}$}{betan} with respect to Elasticity}
\noindent The expression of $f''_{n}(k):=\frac{\partial^{2}}{\partial k^{2}}f_{n}(k)$ is
\begin{equation}\nonumber
    f''_{n}(k)=\frac{\big(A''_{n}(k)C_{n}(k)-A_{n}(k)C''_{n}(k)\big)C_{n}(k)-2C'_{n}(k)\big(A'_{n}(k)C_{n}(k)-A_{n}(k)C'_{n}(k)\big)}{C^{3}_{n}(k)}.
\end{equation}
Analogously to the idea in the previous subsection, if we can prove $A''_{n}(k)=O_{p}(1)$, $B''_{n}(k)=O_{p}(1)$, then $C''_{n}(k)=B''_{n}(k)-2A''_{n}(k)A_{n}(k)-2A'_{n}(k)A'_{n}(k)=O_{p}(1)$, and we will have
\begin{equation}\nonumber
    f''_{n}(k)=\frac{\big(O_{p}(1)O_{p}(1)-O_{p}(1)O_{p}(1)\big)O_{p}(1)-O_{p}(1)\big(O_{p}(1)O_{p}(1)-O_{p}(1)O_{p}(1)\big)}{O_{p}(1)O_{p}(1)O_{p}(1)}=O_{p}(1).\nonumber
\end{equation}
Recall \eqref{AB}, the second-order derivatives $A''_{n}(k)=\frac{\partial^{2}}{\partial k^{2}}A_{n}(k)$ and $B''_{n}(k)=\frac{\partial^{2}}{\partial k^{2}}B_{n}(k)$ are
\begin{align}\nonumber
    &A''_{n}(k)\overset{(\ast)}{=}\frac{1}{n}\sum_{i=1}^{n}\frac{\partial}{\partial k}\Big[\frac{L^{2}}{4}\frac{\lambda_{(i-1)\Delta}^{2-2k}}{(1-k)^{2}}\Big(\frac{2}{1-k}-2\log \lambda_{(i-1)\Delta}\Big)\Big]\\
    =&\frac{1}{n}\sum_{i=1}^{n}\Bigg[\frac{3L^{2}}{2(1-k)^{4}}\lambda_{(i-1)\Delta}^{2-2k}-\frac{2L^{2}}{(1-k)^{3}}\lambda_{(i-1)\Delta}^{2-2k}\log\lambda_{(i-1)\Delta}+\frac{L^{2}}{(1-k)^{2}}\lambda_{(i-1)\Delta}^{2-2k}\big(\log \lambda_{(i-1)\Delta}\big)^{2}\Bigg],\nonumber\\
    &B''_{n}(k)\overset{(\diamond)}{=}\frac{1}{n}\sum_{i=1}^{n}\frac{\partial}{\partial k}\Big[\frac{L^{4}\lambda_{(i-1)\Delta}^{4-4k}}{16(1-k)^{4}}\Big(\frac{4}{1-k}-4\log\lambda_{(i-1)\Delta}\Big)\Big]\nonumber\\
    =&\frac{1}{n}\sum_{i=1}^{n}\Bigg[\frac{5L^{4}}{4(1-k)^{6}}\lambda_{(i-1)\Delta}^{4-4k}-\frac{2L^{4}}{(1-k)^{5}}\lambda_{(i-1)\Delta}^{4-4k}\log\lambda_{(i-1)\Delta}+\frac{L^{4}}{(1-k)^{4}}\lambda_{(i-1)\Delta}^{4-4k}\big(\log\lambda_{(i-1)\Delta}\big)^{2}\Bigg].\nonumber
\end{align}
Again, $(\ast)$ and $(\diamond)$ are ensured by dominated convergence theorem. By analogous arguments as in the previous subsection. Analogous to the idea as stated in the previous subsection, by the local Lipschitz and $\mathcal{L}^{1}(\pi_{0})$-integrability of $x^{2-2k}(\log x)^{q'},~x^{4-4k}(\log x)^{q'}$, $(q'=0,1,2)$ (verified previously using Lemma~\ref{ergodicity-of-CKLS}'s (\romannumeral4) and Lemma~\ref{ergodicity-of-CIR}'s (\romannumeral3), all functions inside the braces satisfy the assumptions of Theorem~\ref{theorem-Riemann}. Hence each average in the expressions of $A''_{n}(k)$ and $B''_{n}(k)$ are $O_p(1)$, and so $A''_{n}(k)=O_{p}(1),~B''_{n}(k)=O_{p}(1)$.

\bibliographystyle{unsrt}  
\bibliography{references} 

\begin{appendix}
\section{Preliminaries: Key Properties of CKLS and CIR Processes}\label{aaa}
\begin{lemma}\label{ergodicity-of-CKLS}[Key Properties of the CKLS process]\\
(\romannumeral1) \textbf{Existence and uniqueness of the (strong) solution and its pathwise strict positivity:} For $k>\frac{1}{2}$, $\lambda_{t}$ is a pathwise unique strong solution (given that the initial value $\lambda_{0}>0$), being strictly positive-valued almost surely. For $k=\frac{1}{2}$, $\lambda_{t}$ is pathwise unique and strong: for subcase $2a\geq\sigma^{2}$, $\lambda_{t}$ is strictly positive-valued; for subcase $2a<\sigma^{2}$, $\lambda_{t}$ can reach $0$ with probability one, and $0$ is an instantaneously reflecting boundary: the process spends zero Lebesgue time at $0$ and leaves $0$ continuously (i.e., it immediately re-enters $(0,+\infty)$ with no jump). For $0<k<\frac{1}{2}$, the origin is an accessible boundary, and a boundary condition at $\lambda_{t}=0$ is required for uniqueness. A standard approach is to choose reflection at the origin, which can render the model well-posed. For any $k>0$, $\infty$ is an unattainbale boundary.\\ 
(\romannumeral2) \textbf{Positive Harris recurrence:} On $t\in[0,+\infty)$ (meaning the time filtration is modified as $\{\mathcal{F}_{t}\}_{t\in[0,+\infty)}$), $\lambda_{t}$ is positive Harris recurrent (provided that $C_{k}<\infty$, for which a simple sufficient condition is $b>0$ and $a>0$ for $k\geq \frac{1}{2}$) with a unique invariant law $\pi_{0}(dx)=p_{\infty}(x)dx$ (though not essential, we may set $\lambda_{0}\sim\pi_{0}$ w.l.o.g. to ensure $*$ in (\romannumeral4)) and the unique stationary density $p_{\infty}(x)=C_{k}x^{-2k}e^{\psi(x;k)}$ with $C_{k}=\big(\int_{0}^{\infty}u^{-2k}e^{\Lambda(u;k)}du\big)^{-1}$, where
$$\psi(x;k)=\left\{
\begin{aligned}
    &\frac{2}{\sigma^{2}}\Big(\frac{ax^{1-2k}}{1-2k}-\frac{bx^{2-2k}}{2-2k}\Big),&&k\in(0,\frac{1}{2}) \cup(\frac{1}{2},1) \cup (1,\infty);\\
    &\frac{2}{\sigma^{2}}\big(a\log x-bx\big),&&k=\frac{1}{2}\text{, (Cox--Ingersoll--Ross model)};\\
    &\frac{2}{\sigma^{2}}\big(-\frac{a}{x}-b\log x\big),&&k=1\text{, (Brennan and Schwartz model)},
\end{aligned}
\right.
$$
For notational convenience, expectations with respect to the invariant law $\pi_{0}$ will simply be denoted by $\mathbb{E}$; that is, $\mathbb{E}_{\pi_{0}}[f(\lambda_{t})]\equiv\mathbb{E}[f(\lambda_{t})]:= \int_{0}^{\infty}f(x)\pi_{0}(dx)$ for some functions $f$.
\\
(\romannumeral3) \textbf{Uniform $q$-th moment bound:} For $k\in(\frac{1}{2},1)$, for any $q\geq 0 $, there exists $M=M(q,T,a,b,\sigma,k,\lambda_{0})<\infty$ such that $\mathbb{E}\big[\sup_{t\in[0,T]}\lambda_{t}^{\pm q}\big]\leq M$. Particularly, $\mathbb{E}[\lambda_{t}^{\pm q}]\leq\sup_{t\in[0,T]}\mathbb{E}[\lambda_{t}^{\pm q}]\leq M<\infty$.\\
(\romannumeral4) \textbf{Ergodic limits of moments:} For $k\in(\frac{1}{2},1)$, then for any $q$ with $\mathbb{R}_{+}\ni x\mapsto x^{q}\in \mathcal{L}^{1}(\pi_0)$, by Birkhoff's theorem, time average integral of $\lambda_{t}^{q}$ has the following almost sure and $\mathcal{L}^{1}$ ergodic limit
\begin{equation}\nonumber
     \frac{1}{T}\int_{0}^{T}\lambda_{t}^{q}dt\xrightarrow[T\to\infty]{a.s./\mathcal{L}^{1}}\int_{0}^{\infty}x^{q}p_{\infty}(x)dx\overset{*}=\mathbb{E}[\lambda_{t}^{q}];
\end{equation}
In particular, in the case of $k\in(\frac{1}{2},1)$, the integral $\int_{0}^{\infty}x^{q} p_{\infty}(x)dx<\infty$ for any $q\in\mathbb{R}$ (when $q=0$, the above convergence holds trivially).
\end{lemma}
\begin{proof}
    See \cite{ning2025ckls}, \cite{mishura2022parameter}, \cite{andersen2007moment}.
\end{proof}

\begin{lemma}\label{ergodicity-of-CIR} [Key Properties of the CIR process]\\
(\romannumeral1) \textbf{Feller's condition:} Following from Feller's boundary classification, $\infty$ is an unattainable boundary. Moreover, the following dichotomy holds: If $2\alpha\geq\gamma^{2}$ is satisfied, then the boundary $0$ is inaccessible and the process remains strictly positive, i.e., $X_{t}>0$ for all $t\geq 0$ almost surely; If $2\alpha<\gamma^{2}$, then the origin is an attainable (instantaneously reflecting) boundary, and the process can hit $0$ in finite time with positive probability.\\
(\romannumeral2) \textbf{Uniform $q$-th moment bound} (As a special case of Lemma~\ref{ergodicity-of-CKLS}'s (\romannumeral3)) For any $q>0$, there exists a constant $C=C(q,T,\alpha,\beta,\gamma,X_{0})$ such that $\sup_{0\leq t\leq T}\mathbb{E}\big[X_{t}^{q}\big]\leq C$.\\
(\romannumeral3) \textbf{Positive Harris recurrence:} (As a special case of Lemma~\ref{ergodicity-of-CKLS}'s (\romannumeral4)) The CIR process admits a unique invariant distribution $\pi_{0}=\Gamma(\kappa,\theta)$, shape $\kappa=\frac{2\alpha}{\gamma^{2}}>0$, rate $\theta=\frac{\gamma^{2}}{2\beta}$, with density $p_{\infty}(x)=\frac{(1/\theta)^{\kappa}}{\Gamma(\kappa)}x^{\kappa-1}e^{-x/\theta}\mathds{1}_{[0,+\infty)}(x)$. All moments exist for any $q>-\kappa$, and in particular for all integers $q\geq 0$. For notational convenience, expectations with respect to the invariant law $\pi_{0}$ will simply be denoted by $\mathbb{E}$; that is, $\mathbb{E}_{\pi_{0}}[f(X_{t})]\equiv\mathbb{E}[f(X_{t})]:= \int_{0}^{\infty}f(x)\pi_{0}(dx)$ for some functions $f$.
By Birkhoff's ergodic theorem, 
\begin{equation}\nonumber
	\frac{1}{T}\int_{0}^{T}X_{t}^{q}dt\xrightarrow[T\to\infty]{a.s./L^{1}}\mathbb{E}[X^{q}]=\theta^{q}\frac{\Gamma(\kappa+q)}{\Gamma(\kappa)}=\theta^{q}\prod_{j=0}^{q-1}(\kappa+j)=\Big(\frac{\gamma^{2}}{2\beta}\Big)^{q}\prod_{j=0}^{q-1}\Big(\frac{2\alpha}{\gamma^{2}}+j\Big)<\infty.
\end{equation}
With the notation $m_{q}:=\mathbb{E}[X^{q}]$, for $q=1,2,3$ in particular, this yields
\begin{equation}\nonumber
   m_{1}=\theta\kappa=\frac{\alpha}{\beta},\quad
   m_{2}=\theta\kappa(\kappa+1)=\frac{\alpha(\alpha+\frac{\gamma^{2}}{2})}{\beta^{2}},\quad
   m_{3}=\theta^{3}\kappa(\kappa+1)(\kappa+2)=\frac{\alpha(\alpha+\gamma^{2})(\alpha+\frac{\gamma^{2}}{2})}{\beta^{3}}.
\end{equation}
(\romannumeral4) \textbf{Representation via squared Bessel process:} Let $\mathrm{BESQ}_{(d,R_{0})}(s)$ denote the squared Bessel process of dimension $d:=4\alpha/\gamma^{2}$ starting from $R_{0}:=4X_{0}/\gamma^{2}$. Then the solution to (\ref{CIR-process}) admits the explicit representation $X_{t}=e^{-\beta t}\mathrm{BESQ}_{(d,R_{0})}\big(\frac{\gamma^{2}}{4\beta}(e^{\beta t}-1)\big),~t\geq 0$.\\
(\romannumeral5) \textbf{Transition density:} For $t>0$, define $\chi_{t}:=\frac{4\beta e^{-\beta t}}{\gamma^{2}(1-e^{-\beta t})},~\phi_{t}:=\chi_{t}e^{-\beta t}X_{0}$. Then the random variable $Y_{t}:=\chi_{t}X_{t}$ is non-central chi-square distributed with degrees of freedom $d$ and non-centrality parameter $2\phi_{t}$. Consequently, the transition density of $X_{t}=x$ is
\begin{equation}\nonumber
	f_{t}(x\lvert X_{0})=\chi_{t}e^{-\phi_{t}-\chi_{t}u}\Big(\frac{\chi_{t}x}{\phi_{t}}\Big)^{\frac{d}{4}-\frac{1}{2}}I_{\frac{d}{2}-1}\Big(2\sqrt{\chi_{t}\phi_{t}x}\Big),\qquad x>0.
\end{equation}
where $I_{q}(u)=\sum_{j=0}^{\infty}\frac{1}{j!\Gamma(j+q+1)}(\frac{u}{2})^{2j+q}$ denotes the modified Bessel function of the first kind.\\
(\romannumeral6) \textbf{$1/2$-H\"{o}lder continuity in $\mathcal{L}^{q}$:} for any $q>0$, there exists $C=C(q,T,\alpha,\beta,\gamma,X_{0})<\infty$ such that
\begin{equation}\nonumber
    \mathbb{E}\big[\lvert X_{t}-X_{s}\lvert^{q}\big] \leq C\lvert t-s\lvert^{q/2},\quad\text{ for any }s,t\in[0,T].
\end{equation}
\end{lemma}
\begin{proof}
    See e.g., \cite{ning2025ckls}, \cite{jeanblanc2009mathematical}, \cite{alaya2013asymptotic}.
\end{proof}

\section{Key Lemmas Underlying \texorpdfstring{\cite{prykhodko2025discretization}}{Prykhodko et al.}'s Main Result on CIR Drift Estimation}\label{bbb}
\begin{proposition}\label{lem:r-sq-increment}
For any $q>0$, there exists $C=C(q,T,\alpha,\beta,\gamma,X_{0})<\infty$ such that
\begin{equation}\nonumber
  \mathbb{E}\big[\lvert X_{t}^{2}-X_{s}^{2}\lvert^{q}\big]\leq C(t-s)^{q/2},\quad \text{ for any }s,t\in[0,T].
\end{equation}
\end{proposition}
\begin{proof}
Note that $\mathbb{E}\big[\lvert X_{t}^{2}-X_{s}^{2}\lvert^{q}\big]=\mathbb{E}\big[(X_{t}+X_{s})^{q}\big\lvert X_{t}-X_{s}\big\lvert^{q}\big]\leq\sqrt{\mathbb{E}\big[(X_{t}+X_{s})^{2q}\big]\mathbb{E}\big[\lvert X_{t}-X_{s}\lvert^{2q}}\big]$ (Cauchy--Schwarz inequality).
The first term is bounded by Lemma~\ref{ergodicity-of-CIR}'s (\romannumeral2); the second term is bounded by Lemma~\ref{ergodicity-of-CIR}'s (\romannumeral6). Combining the two bounds yields the claim.
\end{proof}

\begin{lemma}\label{lem:BDG-CIR}[Moment bounds for averaged CIR martingale integrals with positive orders]\\
For $q'>0$ and $q>0$ there exists a positive constant $C_{q,q'}<\infty$ (independent of $T$) such that
\begin{equation}\nonumber
	\mathbb{E}\Big[\big\lvert\frac{1}{T}\int_{0}^{T}X_{t}^{q'}dW_{t}\big\lvert^{q}\Big]\leq C_{q,q'}T^{-q/2}.
\end{equation}
\end{lemma}
\begin{proof}
\noindent \textit{Case $q\geq 2$:} Set $M_{T}:=\int_{0}^{T}X_{t}^{q'}dW_{t}$. By BDG (\ref{bdg}) and H\"{o}lder inequalities,
\begin{equation}\nonumber
	\mathbb{E}\big[\lvert M_{T}\lvert^{q}\big]\leq C_{q}\mathbb{E}\Big(\int_{0}^{T}X_{t}^{2q'}dt\Big)^{q/2}\leq C_{q}T^{q/2-1}\int_{0}^{T}\mathbb{E}\big[X_{t}^{qq'}\big]dt\leq C_{q,q'}T^{q/2},
\end{equation}
following from Lemma~\ref{ergodicity-of-CIR}'s (\romannumeral2) with order $qq'>0$. Thus $\mathbb{E}\big[\lvert T^{-1}M_{T}\lvert^{q}\big]\leq C_{q,q'}T^{-q/2}$.\\
\noindent\textit{Case $0<q<2$:} By Lyapunov's monotonicity of $\mathcal{L}_{q}$-norms, $\mathbb{E}\big[\lvert M_{T}\lvert^{q}\big]\leq\big(\mathbb{E}\big[\lvert M_{T}\lvert^{2}\big]\big)^{q/2}$. Using the It\^{o} isometry and Lemma~\ref{ergodicity-of-CIR}'s (\romannumeral2) with order $2q'>0$, $\mathbb{E}\big[\lvert M_{T}\lvert^{2}\big]=\mathbb{E}\big[\int_{0}^{T}X_{t}^{2q'}dt\big]\leq T\sup_{t\in[0,T]}\mathbb{E}\big[X_{t}^{2q'}\big]\leq C_{q'}T$. Hence $\mathbb{E}\big[\lvert M_{T}\lvert^{q}\big]\leq (C_{q'}T)^{q/2}$ and $\mathbb{E}\big[\lvert T^{-1}M_{T}\lvert^{q}\big]\leq C_{q,q'}T^{-q/2}$.
\renewcommand{\qed}{}
\end{proof}

\begin{lemma}\label{ergod-law-first-moment}[$\mathcal{L}^{q}$-Convergence Rate of $\frac{1}{T}\int_{0}^{T}X_{t}dt$]
\begin{align}\nonumber
    \frac{1}{T}\int_{0}^{T}X_{t}dt=\frac{\alpha}{\beta}+\frac{X_{0}-X_{T}}{\beta T}+\frac{\gamma}{\beta T}\int_{0}^{T}X_{t}^{\frac{1}{2}}dW_{t}.
\end{align}
Moreover, for any $q>0$
\begin{align}\nonumber
    \mathbb{E}\Big[\big\lvert\frac{1}{T}\int_{0}^{T}X_{t}dt-\frac{\alpha}{\beta}\big\lvert^{q}\Big]\lesssim T^{-q/2}.\nonumber
\end{align}
\end{lemma}
\begin{proof}
With $m_{1}:=\frac{\alpha}{\beta}$, integrating (\ref{CIR-process}) over $[0,T]$ and rearranging yields
\begin{equation}\nonumber
	\frac{1}{T}\int_{0}^{T}X_{t}dt-m_{1}=\frac{X_{0}-X_{T}}{\beta T}+\frac{\gamma}{bT}\int_{0}^{T}\sqrt{X_{s}}dW_{s}.
\end{equation}
Define $\rho_{1}(T):=\frac{X_{0}-X_{T}}{\beta \sqrt{T}}~\Leftrightarrow~\frac{X_{0}-X_{T}}{\beta T}=T^{-\frac{1}{2}}\rho_{1}(T)$. Hence, by the elementary inequality $\lvert x+y+z\lvert^{q}\leq C_{q}\sum\lvert x\lvert^{q}$ for $q\geq 1$ (respectively $\lvert x+y+z\lvert^{q}\leq \sum\lvert x\lvert^{q}$ for $0<q<1$):
\begin{equation}\nonumber
	\mathbb{E}\Big[\big\lvert\frac{1}{T}\int_{0}^{T}X_{t}dt-m_{1}\big\lvert^{q}\Big]\leq C_{q}\Big(\frac{\lvert X_{0}\lvert^{q}+\mathbb{E}[X_{T}^{q}]}{\beta^{q}T^{q}}+\frac{\gamma^{q}}{\beta^{q}T^{q}}\mathbb{E}\Big[\big\lvert\int_{0}^{T}\sqrt{X_{s}}dW_{s}\big\lvert^{q}\Big]\Big).
\end{equation}
By Lemma~\ref{ergodicity-of-CIR}'s (\romannumeral3), $\mathbb{E}[X_{T}^{q}]<\infty$, so the first two terms are $O(T^{-q})$. By Lemma~\ref{lem:BDG-CIR},
\begin{equation}\nonumber
	\mathbb{E}\Big[\big\lvert\int_{0}^{T}\sqrt{X_{s}}dW_{s}\big\lvert^{q}\Big]\leq C_{q}T^{q/2}\quad\Longrightarrow\quad\text{the third term is no larger than }\frac{\gamma^{q}}{\beta^{q}T^{q}}C_{q}T^{q/2}=C_{q}T^{-q/2}.
\end{equation}
Combining the three bounds and noting that $T^{-q}\leq T^{-q/2}$ for $q>0$ yields the desired result.
\renewcommand{\qed}{}
\end{proof}

\begin{lemma}\label{ergod-law-second-moment}[$\mathcal{L}^{q}$-Convergence Rate of $\frac{1}{T}\int_{0}^{T}X_{t}^{2}dt$]
\begin{align}\nonumber
    \frac{1}{T}\int_{0}^{T}X_{t}^{2}dt=\frac{\alpha\big(2\alpha+\gamma^{2}\big)}{2\beta^{2}}+\frac{X_{0}^{2}-X_{T}^{2}}{2\beta T}+\frac{2\alpha+\gamma^{2}}{2\beta T}\Big(\int_{0}^{T}X_{t}dt-\frac{\alpha T}{\beta}\Big)+\frac{\gamma}{\beta T}\int_{0}^{T}X_{t}^{\frac{3}{2}}dW_{t}.
\end{align}
Moreover, for any $q>0$
\begin{align}\nonumber
    \mathbb{E}\Big[\big\lvert\frac{1}{T}\int_{0}^{T}X_{t}^{2}dt-\frac{\alpha(\alpha+\frac{\gamma^{2}}{2})}{\beta^{2}}\big\lvert^{q}\Big]\lesssim T^{-q/2}.\nonumber
\end{align}    
\end{lemma}
\begin{proof}
Apply It\^{o}'s lemma to $x\mapsto x^{2}$ yields 
\begin{equation}\nonumber
    d(X_{t}^{2})=\big(2\alpha X_{t}-2\beta X_{t}^{2}+\gamma^{2}X_{t}\big)dt+2\gamma X_{t}^{\frac{3}{2}}dW_{t}.
\end{equation}
With $m_{2}:=\frac{\alpha(\alpha+\frac{\gamma^{2}}{2})}{\beta^{2}}$, integrating this over $[0,T]$ and solving for $\int_{0}^{T}X_{t}^{2}dt$ yields
\begin{equation}\nonumber
    \int_{0}^{T}X_{t}^{2}dt=m_{2}T+\frac{X_{0}^{2}-X_{T}^{2}}{2\beta}+\frac{\alpha+\frac{\gamma^{2}}{2}}{\beta}\Big(\int_{0}^{T}X_{t}dt-\frac{\alpha T}{\beta}\Big)+\frac{\gamma}{\beta}\int_{0}^{T}X_{t}^{\frac{3}{2}}dW_{t}.
\end{equation}
Integrate (\ref{CIR-process}) to represent $\int_{0}^{T}X_{t}dt$:
\begin{equation}\nonumber
	X_{T}-X_{0}=\alpha T-\beta\int_{0}^{T}X_{t}dt+\gamma\int_{0}^{T}X_{t}^{\frac{1}{2}}dW_{t}~\Longrightarrow~
    \int_{0}^{T}X_{t}dt=\frac{\alpha}{\beta}T+\frac{\gamma}{\beta}\int_{0}^{T}X_{t}^{\frac{1}{2}}dW_{t}-\frac{X_{T}-X_{0}}{\beta}.
\end{equation}
Substituting expression of $\int_{0}^{T}X_{t}dt$ into that of $\int_{0}^{T}X_{t}^{2}dt$ yields the centered decomposition
\begin{align}\nonumber
    &\frac{1}{T}\int_{0}^{T}X_{t}^{2}dt-m_{2}=\frac{\gamma(\alpha+\frac{\gamma^{2}}{2})}{\beta^{2}}\frac1T\int_{0}^{T}X_{t}^{\frac{1}{2}}dW_{t}+\frac{\gamma}{\beta}\frac{1}{T}\int_{0}^{T}X_{t}^{\frac{3}{2}}dW_{t}+T^{-\frac{1}{2}}\rho_{2}(T),\\
	&\text{where}~ T^{-\frac{1}{2}}\rho_{2}(T):=\frac{X_{0}^{2}-X_{T}^{2}}{2\beta T}+\frac{\alpha+\frac{\gamma^{2}}{2}}{\beta}\frac{X_{0}-X_{T}}{\beta T}.\nonumber
\end{align}
Control the $q$-th moments of the three terms on the right-hand side. Firstly, by Lemma~\ref{lem:BDG-CIR} with $q'=\frac{1}{2}$ and $q'=\frac{3}{2}$, there exists a common constant $C_{q}>0$ for the two such that
\begin{equation}\nonumber
	\mathbb{E}\Big[\big\lvert\frac{1}{T}\int_{0}^{T}X_{s}^{\frac{1}{2}}dW_{s}\big\lvert^{q}\Big]\leq C_{q}T^{-q/2},\quad \mathbb{E}\Big[\big\lvert\frac{1}{T}\int_{0}^{T}X_{s}^{\frac{3}{2}}dW_{s}\big\lvert^{q}\Big]\leq C_{q}T^{-q/2}.
\end{equation}
Secondly, using the uniform moment Lemma~\ref{ergodicity-of-CIR}'s (\romannumeral3), we obtain
\begin{equation}\nonumber
	\mathbb{E}\big[\lvert\rho_{2}(T)\lvert^{q}\big]\leq C_{q} \quad\Longrightarrow\quad \mathbb{E}\big[\lvert T^{-\frac{1}{2}}\rho_{2}(T)\lvert^{q}\big]\leq C_{q}T^{-q/2}.
\end{equation}
Combining bounds for the two integral terms and for $T^{-\frac{1}{2}}\rho_{2}(T)$, using $\lvert x+y+z\lvert^{q}\leq C_{q}\sum\lvert x\lvert^{q}$ for $q\geq 1$ (respectively $\lvert x+y+z\lvert^{q}\leq \sum\lvert x\lvert^{q}$ for $0<q<1$), we obtain the desired result.
\end{proof}

\begin{lemma}\label{ergod-law-third-moment}[$\mathcal{L}^{q}$-Convergence Rate of $\frac{1}{T}\int_{0}^{T}X_{t}^{3}dt$]
\begin{align}\nonumber
    &\frac{1}{T}\int_{0}^{T}X_{t}^{3}dt=\frac{\alpha(\alpha+\gamma^{2})(2\alpha+\gamma^{2})}{2\beta^{3}}+\frac{X_{0}^{3}-X_{T}^{3}}{3\beta T}+\frac{\alpha+\gamma^{2}}{\beta}\Big(\frac{1}{T}\int_{0}^{T}X_{t}^{2}dt+\frac{\alpha(2\alpha+\gamma^{2})}{2\beta^{2}}\Big)+\frac{\gamma}{\beta T}\int_{0}^{T}X_{t}^{\frac{5}{2}}dW_{t}.
\end{align}
Moreover, for any $q>0$
\begin{align}\nonumber
    \mathbb{E}\Big[\big\lvert\frac{1}{T}\int_{0}^{T}X_{t}^{3}dt-\frac{\alpha(\alpha+\gamma^{2})(\alpha+\frac{\gamma^{2}}{2})}{\beta^{3}}\big\lvert^{q}\Big]\lesssim T^{-q/2}.
\end{align}    
\end{lemma}
\begin{proof}
\noindent Apply It\^{o}'s lemma to $x\mapsto x^{3}$ yields: 
\begin{equation}\nonumber
    d(X_{t}^{3})=3\big((\alpha-\beta X_{t})X_{t}^{2}+\frac{\gamma^{2}}{2}X_{t}^{2}\big)dt+3\gamma X_{t}^{\frac{5}{2}}dW_{t}=3\big((\alpha+\gamma^{2})X_{t}^{2}-\beta X_{t}^{3}\big)dt+3\gamma X_{t}^{\frac{5}{2}}dW_{t}.
\end{equation}
Integrating this over $[0,T]$ and rearranging terms:
\begin{equation}\nonumber
	\int_{0}^{T}X_{t}^{3}dt=\frac{\alpha+\gamma^{2}}{\beta}\int_{0}^{T}X_{t}^{2}dt+\frac{X_{0}^{3}-X_{T}^{3}}{3\beta}+\frac{\gamma}{\beta}\int_{0}^{T}X_{t}^{\frac{5}{2}}dW_{t}.
\end{equation}
With $m_{3}:=\frac{\alpha(\alpha+\gamma^{2})(\alpha+\frac{\gamma^{2}}{2})}{\beta^{3}}=\frac{\alpha+\gamma^{2}}{\beta}\cdot\frac{\alpha(\alpha+\frac{\gamma^{2}}{2})}{\beta^{2}}$ and the identity $\int_{0}^{T}X_{t}^{2}dt\equiv\frac{\alpha(\alpha+\frac{\gamma^{2}}{2})}{\beta^{2}}T+\big(\int_{0}^{T}X_{t}^{2}dt-\frac{\alpha(\alpha+\frac{\gamma^{2}}{2})}{\beta^{2}}T\big)$, dividing by $T$, the centered decomposition yields
\begin{equation}\nonumber
	\frac{1}{T}\int_{0}^{T}X_{t}^{3}dt-m_{3}=\frac{X_{0}^{3}-X_{T}^{3}}{3\beta T}+\frac{\alpha+\gamma^{2}}{\beta}\Big(\frac{1}{T}\int_{0}^{T}X_{t}^{2}dt-\frac{\alpha(\alpha+\frac{\gamma^{2}}{2})}{\beta^{2}}\Big)+\frac{\gamma}{\beta}\frac{1}{T}\int_{0}^{T}X_{t}^{\frac{5}{2}}dW_{t},\nonumber
\end{equation}
which yields the first claim of the lemma. From the elementary inequality $\lvert x+y+z\lvert^{q}\leq C_{q}\sum\lvert x\lvert^{q}$ for $q\geq 1$ (respectively $\lvert x+y+z\lvert^{q}\leq \sum\lvert x\lvert^{q}$ for $0<q<1$), we obtain that
\begin{align}\nonumber
	&\mathbb{E}\Big[\big\lvert\frac{1}{T}\int_{0}^{T}X_{t}^{3}dt-m_{3}\big\lvert^{q}\Big]\\
    \lesssim&\frac{\lvert X_{0}\lvert^{3q}+\mathbb{E}\big[\lvert X_{T}\lvert^{3q}\big]}{(3\beta T)^{q}}+\Big(\frac{\alpha+\gamma^{2}}{\beta}\Big)^{q}\mathbb{E}\Big[\big\lvert\frac{1}{T}\int_{0}^{T}X_{t}^{2}dt-\frac{\alpha(\alpha+\frac{\gamma^{2}}{2})}{\beta^{2}}\big\lvert^{q}\Big]+\Big(\frac{\gamma}{\beta}\Big)^{q}\mathbb{E}\Big[\big\lvert\frac{1}{T}\int_{0}^{T}X_{t}^{\frac{5}{2}}dW_{t}\big\lvert^{q}\Big].\nonumber
\end{align}
By Lemma~\ref{ergodicity-of-CIR}'s (\romannumeral3) $\frac{X_{0}^{3}-X_{T}^{3}}{3\beta T}=O(T^{-q})$. For $\frac{\alpha+\gamma^{2}}{\beta}\big(\frac{1}{T}\int_{0}^{T}X_{t}^{2}dt-\frac{\alpha(\alpha+\frac{\gamma^{2}}{2})}{\beta^{2}}\big)$, invoke the result in Lemma~\ref{ergod-law-second-moment}. For $\frac{\gamma}{\beta}\frac{1}{T}\int_{0}^{T}X_{t,}^{\frac{5}{2}}dW_{t}$, apply Lemma~\ref{lem:BDG-CIR} with $q=\frac{5}{2}$: $\mathbb{E}\big[\lvert\frac{1}{T}\int_{0}^{T}X_{t}^{\frac{5}{2}}dW_{t}\big\lvert^{q}\big]\leq C T^{-q/2}$. Plugging these bounds into the error bound yields the second claim of the lemma.
\end{proof}

\begin{lemma}\label{lem:ergodic-sum-first-order}[$\int_{0}^{n\Delta}{X}_{t}dt$: $\mathcal{L}^{q}$ and almost-sure Discretization Error for the Riemann sum]\\
Fix $\omega>0$ ($\Delta\to0$ and $n\Delta\to\infty$). Then for any $q>0$:
\begin{equation}\nonumber
	\mathbb{E}\Big[\big\lvert\frac{1}{{n\Delta}}\int_{0}^{n\Delta}X_{t}dt-\frac{1}{n\Delta}\sum_{i=1}^{n}X_{(i-1)\Delta}\Delta\big\lvert^{q}\Big]\lesssim (n\Delta)^{-\omega q/2},
\end{equation}
which means $\frac{1}{n}\sum_{i=1}^{n}X_{(i-1)\Delta}\xrightarrow{\mathcal{L}^{q}}\frac{1}{{n\Delta}}\int_{0}^{n\Delta}X_{t}dt$ as $n\to\infty$. Furthermore, if $\sum_{n\geq 1}(n\Delta)^{-\omega q_{0}/2}<\infty$ for some $q_{0}>0$, and if the chosen $q$ is sufficiently large (e.g., $q>\frac{2}{\omega}$), then $\frac{1}{n}\sum_{i=1}^{n}X_{(i-1)\Delta}\xrightarrow{a.s.}\frac{1}{n\Delta}\int_{0}^{n\Delta}X_{t}dt$ as $n\to\infty$.
\end{lemma}
\begin{proof}
\noindent\textit{Case $q=1$.} Using Lemma~\ref{ergodicity-of-CIR}'s (\romannumeral6) yields
\begin{align}\nonumber
	&\mathbb{E}\Big[\big\lvert\frac{1}{n\Delta}\int_{0}^{n\Delta}X_{t}dt-\frac{1}{n\Delta}\sum_{i=1}^{n}X_{(i-1)\Delta}\Delta\big\lvert\Big]=\frac{1}{n\Delta}\mathbb{E}\Big[\int_{0}^{n\Delta}\sum_{i=1}^{n}(X_{t}-X_{(i-1)\Delta})\mathds{1}_{\{(i-1)\Delta<t\leq i\Delta\}}dt\Big]\\
    \leq&\frac{1}{n\Delta}\sum_{i=1}^{n}\int_{(i-1)\Delta}^{i\Delta}\mathbb{E}\big[\lvert X_{t}-X_{(i-1)\Delta}\lvert\big]dt\leq(n\Delta)^{-1}(n\Delta)C (n\Delta)^{-\omega/2}=C(n\Delta)^{-\omega/2}.\tag{$\ast$}
\end{align}
\noindent\textit{Case $q>1$.} Using Lemma~\ref{ergodicity-of-CIR}'s (\romannumeral6) together with H\"{o}lder's inequality yields
\begin{align}\nonumber
	&\mathbb{E}\Big[\big\lvert\frac{1}{n\Delta}\int_{0}^{n\Delta}X_{t}dt-\frac{1}{n\Delta}\sum_{i=1}^{n}X_{(i-1)\Delta}\Delta\big\lvert^{q}\Big]\\
    =&\frac{1}{(n\Delta)^{q}}\mathbb{E}\Big[\big\lvert\int_{0}^{n\Delta}\sum_{i=1}^{n}(X_{t}-X_{(i-1)\Delta})\mathds{1}_{\{(i-1)\Delta<t\leq i\Delta\}}dt\big\lvert^{q}\Big]\nonumber\\
    \leq&\frac{1}{(n\Delta)^{q}}\mathbb{E}\Big[\int_{0}^{n\Delta}\big\lvert\sum_{i=1}^{n}(X_{t}-X_{(i-1)\Delta})\mathds{1}_{\{
    (i-1)\Delta<t\leq i\Delta\}}\big\lvert^{q}dt\Big]\Big(\int_{0}^{n\Delta}dt\Big)^{q-1}\nonumber\\
    \leq&\frac{1}{(n\Delta)^{q}}(n\Delta)^{q-1}\sum_{i=1}^{n}\int_{(i-1)\Delta}^{i\Delta}\mathbb{E}\big[\lvert X_{t}-X_{(i-1)\Delta}\lvert^{q}\big]dt\leq (n\Delta)^{-1}C(n\Delta)^{-\omega q/2}(n\Delta)=C(n\Delta)^{-\omega q/2}.\nonumber
\end{align}
\noindent \textit{Case $0<q<1$.} By the concavity of $x\mapsto x^{q}$ (Lyapunov's monotonicity of $\mathcal{L}_{q}$ norms),
\begin{equation}\nonumber
    \mathbb{E}\big[\lvert Z_{n}\lvert^{q}\big]\leq \big(\mathbb{E}\big[\lvert Z_{n}\lvert\big]\big)^{q},\quad Z_{n}:=\frac{1}{n\Delta}\int_{0}^{n\Delta}X_{t}dt-\frac{1}{n}\sum_{i=1}^{n}X_{(i-1)\Delta}.
\end{equation}
Using the $q=1$ bound ($\ast$), $\mathbb{E}\big[\lvert Z_{n}\lvert\big]\lesssim (n\Delta)^{-\omega/2}$, hence $\mathbb{E}\big[\lvert Z_{n}\lvert^{q}\big]\lesssim (n\Delta)^{-\omega q/2}$.\\[0.5\baselineskip]
\noindent Fix $\varepsilon>0$ and define $A_{n}(\varepsilon):=\big\{\lvert Z_{n}\lvert \geq\varepsilon\big\}$. By Markov's inequality, $\mathbb{P}\big(A_{n}(\varepsilon)\big)\leq \varepsilon^{-q}\mathbb{E}\big[\lvert Z_{n}\lvert^{q}\big]\leq C\varepsilon^{-q}(n\Delta)^{-\omega q/2}$. Choose $q>0$ sufficiently large that $\sum_{n=1}^{\infty}(n\Delta)^{-\omega q/2}<\infty$, then it holds that $\sum_{n\geq 1}\mathbb{P}\big(A_{n}(\varepsilon)\big)<\infty$, and by the Borel--Cantelli lemma, we obtain $Z_{n}\to 0~a.s.\text{ as }n\to\infty$, which is the desired almost-sure convergence of the Riemann sum to the time average.
\end{proof}

\begin{lemma}\label{lem:ergodic-sum-second-order}[$\int_{0}^{n\Delta}X_{t}^{2}dt$: $\mathcal{L}^{q}$ and almost-sure Discretization Error for the Riemann sum]\\
Fix $\omega>0$ ($\Delta\to0$ and $n\Delta\to\infty$). Then for any $q>0$:
\begin{equation}\nonumber
	\mathbb{E}\Big[\big\lvert\frac{1}{{n\Delta}}\int_{0}^{n\Delta}X_{t}^{2}dt-\frac{1}{n\Delta}\sum_{i=1}^{n}X_{(i-1)\Delta}^{2}\Delta\big\lvert^{q}\Big]\lesssim (n\Delta)^{-\omega q/2},
\end{equation}
which means $\frac{1}{n}\sum_{i=1}^{n}X_{(i-1)\Delta}^{2}\xrightarrow{\mathcal{L}^{q}}\frac{1}{{n\Delta}}\int_{0}^{n\Delta}X_{t}^{2}dt$ as $n\to\infty$. Furthermore, if $\sum_{n\geq 1}(n\Delta)^{-\omega q_{0}/2}<\infty$ for some $q_{0}>0$, and if the chosen $q$ is sufficiently large (e.g., $q>\frac{2}{\omega}$), then $\frac{1}{n}\sum_{i=1}^{n}X_{(i-1)\Delta}^{2}\xrightarrow{a.s.}\frac{1}{n\Delta}\int_{0}^{n\Delta}X_{t}^{2}dt$ as $n\to\infty$.
\end{lemma}
\begin{proof}
    The proof is nearly identical to that of Lemma~\ref{lem:ergodic-sum-first-order}, with $X_{t}$ replaced by $X_{t}^{2}$ and Lemma~\ref{ergodicity-of-CIR}'s (\romannumeral6) replaced by Proposition~\ref{lem:r-sq-increment} in the cases $q=1$ and $q>1$.
\end{proof}

\begin{proposition}\label{prop:ergodic-sum-third-order}[$\int_{0}^{n\Delta}X_{t}^{3}dt$: $\mathcal{L}^{q}$ and almost-sure Discretization Error for the Riemann sum]\\
Fix $\omega>0$ ($\Delta\to0$ and $n\Delta\to\infty$). Then for any $q>0$:
\begin{equation}\nonumber
	\mathbb{E}\Big[\big\lvert\frac{1}{{n\Delta}}\int_{0}^{n\Delta}X_{t}^{3}dt-\frac{1}{n\Delta}\sum_{i=1}^{n}X_{(i-1)\Delta}^{3}\Delta\big\lvert^{q}\Big]\lesssim (n\Delta)^{-\omega q/2},
\end{equation}
which means $\frac{1}{n}\sum_{i=1}^{n}X_{(i-1)\Delta}^{3}\xrightarrow{\mathcal{L}^{q}}\frac{1}{{n\Delta}}\int_{0}^{n\Delta}X_{t}^{3}dt$ as $n\to\infty$. Furthermore, if $\sum_{n\geq 1}(n\Delta)^{-\omega q_{0}/2}<\infty$ for some $q_{0}>0$, and if the chosen $q$ is sufficiently large (e.g., $q>\frac{2}{\omega}$), then $\frac{1}{n}\sum_{i=1}^{n}X_{(i-1)\Delta}^{3}\xrightarrow{a.s.}\frac{1}{n\Delta}\int_{0}^{n\Delta}X_{t}^{3}dt$ as $n\to\infty$.
\end{proposition}
\begin{proof}
Using analogous arguments for $q=1$, $q>1$ and $0<q<1$ as in the above-mentioned two Lemmas via Lemma~\ref{lem:r-sq-increment}, together with the result that: There exists $C_{q}^{*}$ such that 
\begin{equation}\nonumber
    \mathbb{E}\big[\lvert X_{t}^{3}-X_{s}^{3}\lvert\big]\leq C_{q}^{*}(t-s)^{q/2}\text{ for every }q>0,~0\leq s<t \leq T\text{ with }0<t-s<1,
\end{equation}
followed by the Cauchy--Schwarz inequality: $\mathbb{E}\big[\lvert X_{t}^{3}-X_{s}^{3}\lvert^{q}\big]=\mathbb{E}\big[\lvert X_{t}-X_{s}\lvert^{q}\big(X_{t}^{2}+X_{t}X_{s}+X_{s}^{2}\big)^{q}\big]$ $\leq\sqrt{\mathbb{E}\big[(X_{t}^{2}+X_{t}X_{s}+X_{s}^{2})^{2q}\big]\mathbb{E}\big[\lvert X_{t}-X_{s}\lvert^{2q}\big]}$, in which the boundedness inside the root is ensured by Lemma~\ref{lem:r-sq-increment} and Lemma~\ref{ergodicity-of-CIR}'s (\romannumeral6).
\end{proof}

\begin{lemma}\label{discrete-continuous-error}[Asymptotic Equivalence of Left Riemann Sums and Integrated Processes]\\
\noindent Assume $\omega>1$ ($\Delta\to0$, $n\Delta\to\infty$ and $n\Delta^{2}\to\infty$), then the following results hold:
\begin{align}\nonumber
    (\romannumeral1) \quad (n\Delta)^{-\frac{1}{2}}\sum_{i=1}^{n}X_{(i-1)\Delta}\Delta&=(n\Delta)^{-\frac{1}{2}}\int_{0}^{n\Delta}X_{t}dt+\xi_{1}(n),\\
    (\romannumeral2) \quad(n\Delta)^{-\frac{1}{2}}\sum_{i=1}^{n}X_{(i-1)\Delta}^{2}\Delta&=(n\Delta)^{-\frac{1}{2}}\int_{0}^{n\Delta}X_{t}^{2}dt+\xi_{2}(n),\nonumber\\
    (\romannumeral3) \quad(n\Delta)^{-\frac{3}{2}}\Big(\sum_{i=1}^{n}X_{(i-1)\Delta}\Delta\Big)^{2}&=(n\Delta)^{-\frac{3}{2}}\Big(\int_{0}^{n\Delta}X_{t}dt\Big)^{2}+\xi_{3}(n),\nonumber
\end{align}
with $\xi_{l}(n)\xrightarrow{p}0$, as $n\to\infty$, $l=1,2,3$.
\end{lemma}
\begin{proof}
(\romannumeral1) and (\romannumeral2) follow from Lemmas~\ref{lem:ergodic-sum-first-order} and~\ref{lem:ergodic-sum-second-order}. (\romannumeral3) is a natural consequence of (\romannumeral1).\\
For (\romannumeral1), as there exists some $C$ such that
\begin{align}\nonumber
	&\mathbb{E}\Big[\big\lvert\frac{1}{{n\Delta}}\int_{0}^{n\Delta}X_{t}dt-\frac{1}{n\Delta}\sum_{i=1}^{n}X_{(i-1)\Delta}\Delta\big\lvert^{q}\Big]\leq C(n\Delta)^{-\omega q/2},\\
	\Longrightarrow\quad &\mathbb{E}\big[\lvert\xi_{1}(n) \lvert^{q}\big]=(n\Delta)^{q/2}\mathbb{E}\Big[\big\lvert\frac{1}{{n\Delta}}\int_{0}^{n\Delta}X_{t}dt-\frac{1}{n\Delta}\sum_{i=1}^{n}X_{(i-1)\Delta}\Delta\big\lvert^{q}\Big]\leq C(n\Delta)^{(1-\omega)q/2}.\nonumber
\end{align}
When $q=2$, this becomes $\mathbb{E}\big[\lvert \xi_{1}(n)\lvert^{2}\big]\leq C(n\Delta)^{1-\omega}$. By Chebyshev's inequality $\xi_{1}(n)=O_{p}\big((n\Delta)^{(1-\omega)/2}\big)$. If $\omega>1$ (i.e., $n\Delta^{2}\to0$) is further guaranteed, $\xi_{1}(n)\xrightarrow{p} 0$ as $n\to\infty$.\\
For (\romannumeral2), the result follows analogously as in (\romannumeral1) since $\xi_{2}(n)=O_{p}\big((n\Delta)^{(1-\omega)/2}\big)$ as well.\\
For (\romannumeral3), we note that
\begin{equation}\nonumber
    (n\Delta)^{-\frac{3}{2}}\Big(\sum_{i=1}^{n}X_{(i-1)\Delta}\Delta\Big)^{2}=(n\Delta)^{-\frac{1}{2}}\Big((n\Delta)^{-\frac{1}{2}}\int_{0}^{n\Delta}X_{t}dt+\xi_{1}(n)\Big)^{2}=(n\Delta)^{-\frac{3}{2}}\Big(\int_{0}^{n\Delta}X_{t}dt\Big)^{2}+\xi_{3}(n),
\end{equation}
where $\xi_{3}(n)=2\xi_{1}(n)\frac{1}{n\Delta}\int_{0}^{n\Delta}X_{t}dt+(n\Delta)^{-\frac{1}{2}}\xi_{1}^{2}(n)\xrightarrow{p}0$, as $n\to\infty$. This follows from the facts that $\xi_{1}(n)\xrightarrow{p}0$ and $\frac{1}{n\Delta}\int_{0}^{n\Delta}X_{t}dt\xrightarrow{p}0$, as $n\to\infty$.
\end{proof}

\begin{lemma}\label{continuous-convergence-in-Lq}[$\mathcal{L}^{q}$-expansions for time-integrated functionals of $X_{t}$]\\
Let $q>0$. We have the following equations and convergences in $\mathcal{L}^{q}$ as $T\to\infty$ (thus also in probability):
\begin{align}\nonumber
    (\romannumeral1)~&T^{-\frac{1}{2}}\int_{0}^{T}X_{t}dt=\frac{\alpha}{\beta}T^{\frac{1}{2}}+\frac{\gamma}{\beta T^{\frac{1}{2}}}\int_{0}^{T}X_{t}^{\frac{1}{2}}dW_{t}+\rho_{1}(T),\\
    &\text{ with }\rho_{1}(T)=\frac{X_{0}-X_{T}}{\beta T^{\frac{1}{2}}}\xrightarrow{\mathcal{L}^{q}}0.\nonumber\\
    (\romannumeral2)~&T^{-\frac{1}{2}}\int_{0}^{T}X_{t}^{2}dt=\frac{\alpha(\alpha+\frac{\gamma^{2}}{2})}{\beta^{2}}T^{\frac{1}{2}}+\frac{\gamma(\alpha+\frac{\gamma^{2}}{2})}{\beta^{2}T^{\frac{1}{2}}}\int_{0}^{T}X_{t}^{\frac{1}{2}}dW_{t}+\frac{\gamma}{\beta T^{\frac{1}{2}}}\int_{0}^{T}X_{t}^{\frac{3}{2}}dW_{t}+\rho_{2}(T),\nonumber\\
    &\text{with }\rho_{2}(T)=\frac{X_{0}^{2}-X_{T}^{2}}{2\beta T^{\frac{1}{2}} }+\frac{(\alpha+\frac{\gamma^{2}}{2})\rho_{1}(T)}{\beta}\xrightarrow{\mathcal{L}^{q}}0.\nonumber\\
    (\romannumeral3)~&T^{-\frac{3}{2}}\Big(\int_{0}^{T}X_{t} dt\Big)^{2}=\frac{\alpha^{2}}{\beta^{2}}T^{\frac{1}{2}}+\frac{2\alpha\gamma}{\beta^{2}T^{\frac{1}{2}}}\int_{0}^{T}X_{t}^{\frac{1}{2}}dW_{t}+\rho_{3}(T),\nonumber\\
    &\text{with }\rho_{3}(T)=\frac{\gamma^{2}}{\beta^{2}T^{\frac{3}{2}}}\Big(\int_{0}^{T}X_{t}^{\frac{1}{2}}dW_{t}\Big)^{2}+\frac{\rho_{1}^{2}(T)}{T^{\frac{1}{2}}}+\frac{2\alpha \rho_{1}(T)}{\beta}
    +\frac{2\gamma \rho_{1}(T)}{\beta T}\int_{0}^{T}X_{t}^{\frac{1}{2}}dW_{t}\xrightarrow{\mathcal{L}^{q}}0.\nonumber
\end{align}
\end{lemma} 
\begin{proof}
    Expressions of (\romannumeral1) and (\romannumeral2) are already given in the proofs of Lemmas~\ref{ergod-law-first-moment} and~\ref{ergod-law-second-moment}, where the respective $\mathcal{L}^{q}$-convergences are ensured by Lemma~\ref{ergodicity-of-CIR}'s (\romannumeral2). For (\romannumeral3), by squaring the expression of (\romannumeral1), we get the desired expression, while the convergence is derived from the bound Lemma~\ref{lem:BDG-CIR} together with the convergence of $\rho_{1}(T)$.
\end{proof}

\section{Proof of Selected Established Theorems}\label{ccc}
\noindent \textbf{Proof of Theorem~\ref{stable-slutsky}}
\begin{proof}  
To show (\romannumeral1): For any $\mathcal{F}$-measurable $\zeta'$, we have $(Z_{n},\zeta,\zeta')\xrightarrow{d}(Z,\zeta,\zeta')$. Since $\zeta_{n}\xrightarrow{p}\zeta$, we deduce that $(Z_{n},\zeta_{n},\zeta')\xrightarrow{d}(Z,\zeta,\zeta')$. This means $(Z_{n},\zeta_{n})\xrightarrow{d_{st}}(Z,\zeta)$.\\
\noindent To show (\romannumeral2): Similarly, for any $\mathcal{F}$-measurable random variable $\zeta$, we have $(Z_{n},\zeta)\xrightarrow{d}(Z,\zeta)$. Applying the classical continuous mapping theorem for convergence in law, we deduce that $(g(Z_{n}),\zeta)\xrightarrow{d}(g(Z),\zeta)$. Hence, we obtain $g(Z_{n})\xrightarrow{d_{st}}g(Z)$.\\
\noindent To show (\romannumeral3): By (\romannumeral1) we have the joint stable convergence in law $(Z_{n},\zeta_{n})\xrightarrow{d_{st}}(Z,\zeta)$. Then, applying (\romannumeral2) with 
continuous $h:\mathbb{R}^{2}\to\mathbb{R}$, we obtain $h(Z_{n},\zeta_{n})\xrightarrow{d_{st}}h(Z,\zeta)$.
Choosing $h$ as addition, multiplication, or division (with $\zeta\neq 0$) yields the usual Slutsky-type results. 
\end{proof}

\noindent \textbf{Proof of Theorem~\ref{preservation of convergence}}
\begin{proof}
For any bounded Borel function $g:\mathbb{R}\to\mathbb{R}$, let $K:=\lVert g\lVert_{\infty}<\infty$. For all $n\in\mathbb{N}$,
\begin{equation}\nonumber
	\lvert\mathsf{G}_{n}\lvert:=\big\lvert \Lambda_{T}\mathds{1}_{A}g(Y_{n})\big\lvert\leq K \Lambda_{T} \quad a.s.
\end{equation}
Because $\Lambda_{T}\in\mathcal{L}^{1}(\mathbb{P})$ (martingale with $\Lambda_{0}=1$ implies $\mathbb{E}^{\mathbb{P}}[\Lambda_{T}]=1$), we can take $\zeta:=K \Lambda_{T} \in \mathcal{L}^{1}(\mathbb{P})$. By Tool Lemma~$(\star)$ presented at the end of the proof, $\{\Lambda_{T}\mathds{1}_{A}g(Y_{n})\}_{n\geq 1}$ is uniformly integrable. As a result, for any $A\in\mathcal{F}_{T}$:
\begin{equation}\nonumber
    \mathbb{E}^{\mathbb{Q}}[\mathds{1}_{A}g(Y_{n})]=\mathbb{E}^{\mathbb{P}}[\Lambda_{T}\mathds{1}_{A}g(Y_{n})]\to\mathbb{E}^{\mathbb{P}}[\Lambda_{T}\mathds{1}_{A}g(Y)]=\mathbb{E}^{\mathbb{Q}}[\mathds{1}_{A}g(Y)],\quad \text{as }n\to\infty.
\end{equation}
by uniform integrability of $\mathsf{G}_{n}:=\Lambda_{T}\mathds{1}_{A}g(Y_{n})$. Since $\Lambda_{T}$ is $\mathcal{F}_{T}$-measurable and $\varphi \!\perp\!\!\!\perp\mathcal{F}_{T}$ (on a product extension) under $\mathbb{P}$, for any $A\in\mathcal{F}_{T}$ and $B=\{\varphi\in E\}$ for some Borel set $E\subset \mathbb{R}$:
\begin{equation}\nonumber
    \mathbb{Q}(A\cap B)=\mathbb{E}^{\mathbb{P}}\big[\Lambda_{T}\mathds{1}_{A}\mathds{1}_{B}\big]=\mathbb{E}^{\mathbb{P}}[\mathds{1}_{B}]\mathbb{E}^{\mathbb{P}}\big[\Lambda_{T}\mathds{1}_{A}]=\mathbb{P}(B)\mathbb{Q}(A),
\end{equation}
so $\varphi$ remains independent of $\mathcal{F}_{T}$ under $\mathbb{Q}$ with the same marginal law $\mathcal{N}(0,1)$.\\[0.5\baselineskip]
\noindent For the reverse implication from $\mathbb{Q}$ to $\mathbb{P}$: If $\Lambda_{T}:=\frac{d\mathbb{Q}}{d\mathbb{P}}\big\lvert_{\mathcal{F}_{T}}$ is the $\mathbb{P}$-derivative of $\mathbb{Q}$, then its reciprocal $\widetilde{\Lambda}_{T}:=\frac{d\mathbb{P}}{d\mathbb{Q}}\big\lvert_{\mathcal{F}_{T}}=\frac{1}{\Lambda_{T}}$ is the $\mathbb{Q}$-derivative of $\mathbb{P}$. By Bayes' formula, for $0\leq T'\leq T$ (thus $\mathcal{F}_{T'}\subseteq\mathcal{F}_{T}$),
\begin{equation}\nonumber
	\mathbb{E}^{\mathbb{Q}}\Big[\frac{1}{\Lambda_{T}}\Big\lvert\mathcal{F}_{T'}\Big]=\frac{\mathbb{E}^{\mathbb{P}}\big[\Lambda_{T}\cdot\frac{1}{\Lambda_{T}}\big\lvert\mathcal{F}_{T'}\big]}{\mathbb{E}^{\mathbb{P}}\big[\Lambda_{T}\lvert\mathcal{F}_{T'}\big]}=\frac{1}{\Lambda_{T'}}=\widetilde{\Lambda}_{T'}.
\end{equation}
Hence $\widetilde{\Lambda}_{T'}$ for $T'\leq T$ is a $\mathbb{Q}$-martingale, and the change-of-measure identity $\mathbb{E}^{\mathbb{P}}[\zeta]=\mathbb{E}^{\mathbb{Q}}[\widetilde{\Lambda}_{T}\zeta]$ holds for all integrable $\zeta\in\mathcal{L}^{1}$ under $\mathbb{Q}$ with respect to $\mathcal{F}_{T'}$. This allows one to repeat the proof verbatim with $\widetilde{\Lambda}_{T}$ in place of $\Lambda_{T}$ and $\mathbb{Q}$ in place of $\mathbb{P}$, establishing invariance of stable convergence in law under equivalent measures in the reverse direction.\\[0.5\baselineskip]
\noindent \textbf{Tool Lemma $(\star)$} [Domination $\Rightarrow$ Uniform Integrability]\label{UI}
If there exists $\zeta\in\mathcal{L}^{1}$ such that $\lvert \mathsf{G}_{n}\lvert\leq \zeta~a.s.$ for all $n$, then the family $\{\mathsf{G}_{n}\}$ is uniformly integrable.\\
\noindent \textit{Proof:}
For any $\varepsilon>0$, since $\zeta\in\mathcal{L}^{1}$, there exists $K>0$ such that $\mathbb{E}\big[\zeta\mathds{1}_{\{\zeta>K\}}\big]<\varepsilon$. Then 
\begin{equation}\nonumber
	\sup_{n}\mathbb{E}\big[\lvert \mathsf{G}_{n}\lvert\mathds{1}_{\{\lvert \mathsf{G}_{n}\lvert>K\}}\big]\leq\sup_{n} \mathbb{E}\big[\zeta\mathds{1}_{\{\zeta>K\}}\big]\leq\varepsilon,
\end{equation}
which is exactly the definition of uniform integrability. 
\end{proof}
\end{appendix}

\end{document}